\documentclass[10pt]{article}

\usepackage{amsmath}
\usepackage{amsmath, amsthm, amsfonts, amssymb, color}
\usepackage{mathrsfs}
\usepackage{latexsym,bm}
\usepackage{graphicx}
\usepackage[toc,page]{appendix}
\usepackage{mathtools}
\mathtoolsset{showonlyrefs}
\usepackage{comment}
\usepackage{xcolor}
\usepackage[numbers, sort&compress]{natbib}

\newtheorem{theorem}{Theorem}[section]
\newtheorem{proposition}[theorem]{Proposition}

\newtheorem{lemma}[theorem]{Lemma}

\newtheorem{remark}[theorem]{Remark}
\newtheorem{example}[theorem]{Example}

\definecolor{wco}{rgb}{0.5,0.2,0.3}

\newcommand{\pp}{\mathbb{P}}

\newcommand{\mf}{\mathcal{M}(E)}
\newcommand{\mfo}{\mathcal{M}^{0}(E)}

\newcommand{\ud}{\mathrm{d}}

\newcommand{\R}{\mathbb{R}}
\newcommand{\e}{\mathrm{e}}

\newcommand{\Var}{\mathrm{Var}}
\newcommand{\Cov}{\mathrm{Cov}}
\newcommand{\mm}{\mathfrak{M}}
\numberwithin{equation}{section} 

\begin{document}

\allowdisplaybreaks

\title{\bf Functional limit theorems for weighted occupation times of supercritical superprocesses}
\author{
{\bf Ting Yang }\thanks{School of Mathematics and Statistics, Beijing Institute of Technology, Beijing, 100081, P.R. China.
Email: yangt@bit.edu.cn.
The research of this author is supported by NSFC (Grant No. 12371143).}
}

\date{}
\maketitle

\begin{abstract}
Suppose $X=(X_t)_{t\ge0}$ is a supercritical superprocess with a non-local branching
mechanism on a Lusin space $E$, whose mean semigroup admits a positive leading
eigenvalue $\lambda_1$ with a bounded positive eigenfunction $\varphi$ and a finite eigenmeasure
$\widetilde{\varphi}$. Let $Y_t:=\int_0^t X_s\mathrm{d}s$ be its occupation measure. We
first prove a law of large numbers, that is, for every bounded measurable function $f$,
$\e^{-\lambda_1 t}\langle f,Y_t\rangle$ converges in $L^2$ to
$\lambda_1^{-1}\langle f,\widetilde{\varphi}\rangle W^\varphi_\infty$, where
$W^\varphi_\infty$ is the limit of the martingale
$W^\varphi_t:=\e^{-\lambda_1 t}\langle\varphi,X_t\rangle$. We then establish
functional central limit theorems for the fluctuations of $\langle f,Y_t\rangle$
around this random centering. We show that the normalisation and the limit are
determined by the spectral decay rate of $f$ compared with the critical value
$\lambda_1/2$, and there are three possible regimes: after suitable normalisation, the fluctuation process converges in distribution in $\mathcal{C}[0,+\infty)$
either to a mixed Gaussian process $\sqrt{W^{\varphi}_{\infty}}G(f)$, or to a constant Gaussian process with random variance depending on $W^{\varphi}_{\infty}$, or to a non-Gaussian limit
given by the limit of a martingale associated with the spectral projection of $f$. These results are the occupation-time counterparts of the fluctuation theorems for instantaneous linear functionals of supercritical non-local branching superprocesses.
\end{abstract}

\medskip


\medskip

\section{Introduction}

Superprocesses, which are an important class of measure-valued branching processes, arise as the short lifetime and high density limits of
branching particle systems, and have been extensively studied
over the past decades; see, e.g., the monograph \cite{Li11} for a systematic
account of the theory.  For any measure-valued branching process $X=(X_t)_{t\ge0}$ on a
state space $E$, a fundamental object is its occupation measure $Y_t:=\int_0^t X_s\,\mathrm{d}s$,
which captures the space-time behaviour of the process rather than its
behaviour at a fixed instant, and therefore contains essential information on the global structure of $X$.

The long-time behaviour of weighted occupation
times $\langle f,Y_{t}\rangle:=\int_{0}^{t}\langle f,X_{s}\rangle\mathrm{d}s$, in particular laws of large numbers (LLN) and central limit theorems (CLTs), forms a
classical topic in the theory of branching systems.
We recall that a branching system is called
supercritical, critical or subcritical according as its mean grows exponentially, remains constant in time, or decays to $0$.
In the \textit{critical} regime, the fluctuation theory of occupation
times is by now well developed.
Iscoe \cite{Iscoe86a} obtained CLTs
for a weighted occupation time of critical super-stable processes started off from Lebesgue measure, with Gaussian or stable generalised random fields arising
in the limit depending on the dimension. Cox and Griffeath \cite{CG85} obtained analogous results for critical binary branching Brownian motions.
Later, Dawson et al. \cite{DGW01} investigated
occupation time fluctuations of one- and two-level critical branching particle systems
on locally compact Abelian groups, and proved that the asymptotics are
determined by the Green potential operator of the individual particle motion.
In a series of
papers, Bojdecki et al.
\cite{BGT04,BGT06a,BGT06b,BGT07} proved functional limit theorems for rescaled
occupation time fluctuations of $(d,\alpha,\beta)$-branching particle systems.
Related functional CLTs were obtained by Hong \cite{Hong04} for super $\alpha$-stable processes and by
Birkner and Z\"{a}hle \cite{BZ07} for state-dependent branching random walks.
Recently, Fan et al. \cite{FHX25} studied a two-layer model
in which a critical super-Brownian motion evolves in a Gaussian random
environment that is white in time and coloured in space (this model
reduces to the classical critical super-Brownian motion of \cite{Iscoe86a,Iscoe86b}
when the environment degenerates), and established quenched and annealed
LLN and CLTs for the weighted occupation
times in dimensions $d\ge3$.
In the \textit{subcritical} regime, Mi{\l}o\'{s} \cite{Milos09} proved a functional
CLT for occupation time fluctuations of branching
immigration systems, where the limit is a generalized Wiener process taking values in the space of tempered distributions.

The fluctuation theory of \emph{supercritical} branching systems is quite
different in nature, since the process grows exponentially and the fluctuations have a random centering.
The study of CLTs in this setting goes back to Kesten and Stigum \cite{KS66} and Athreya
\cite{Ath69a,Ath69b} for multitype continuous-time Galton--Watson processes. Related functional CLTs were obtained by Janson \cite{Jan04}. These works revealed a trichotomy of regimes determined by
the spectral gap of the mean matrix of the processes.
This trichotomy was later extended to spatial branching processes, including branching Markov processes and superprocesses, see, e.g.,
\cite{AH83,AM15,Milos18,RSZ14,RSZ15,RSZ17,RSZ17b} and the references therein.
For non-local branching mechanisms, LLN and
CLTs were established in \cite{DH25} for branching Markov processes and in \cite{PY20,Yang25} for superprocesses. Again, three regimes
arise according to the decay rate of the projection of $f$ onto the
complement of the principal eigenfunction. We emphasize that all the works cited above in the supercritical setting concern
the linear functionals $\langle f,X_t\rangle$ at a single time $t$, rather than occupation times.
In fact, to the best of our knowledge, no central or functional limit theorem is available for the occupation times of supercritical superprocesses.
The purpose of this paper is to fill this gap.

In this paper we consider a superprocess $X:=(X_t)_{t\ge0}$ on a Lusin space
$E$ with a general non-local branching mechanism as in \cite{Li11}. Under the assumptions that the mean semigroup displays a natural Perron--Frobenius type behaviour with an eigentriplet $(\lambda_{1},\varphi,\widetilde{\varphi})$ and the second moment exists, we first prove an
LLN for the occupation measure (see Lemma \ref{lem2.2} below) in the sense that,
for every bounded measurable function $f$,
\begin{equation}\label{lln-intro}
\e^{-\lambda_1 t}\langle f,Y_t\rangle\ \xrightarrow[t\to+\infty]{L^2}\
\frac{1}{\lambda_1}\langle f,\widetilde{\varphi}\rangle W^\varphi_\infty ,
\end{equation}
where $W^\varphi_\infty$ is the limit of the non-negative martingale
$W^\varphi_t:=\e^{-\lambda_1 t}\langle\varphi,X_t\rangle$. Our main results,
Theorems \ref{them1}--\ref{them3} below, are functional CLTs for the
fluctuation process
\begin{equation}\label{Z-intro}
Z^{(t)}_s(f):=\e^{-\lambda_1(t+s)}\langle f,Y_{t+s}\rangle
-\frac{1}{\lambda_1}\langle f,\widetilde{\varphi}\rangle W^\varphi_\infty,
\quad\forall t,s\ge0.
\end{equation}
We show that as in the classical
theory of supercritical spatial branching processes, there exist three regimes corresponding to the decay rate of the centered
function $\tilde {f}:=f-\langle f,\widetilde{\varphi}\rangle\varphi$ under the mean
semigroup compared with the critical value $\lambda_1/2$, with each regime exhibiting a different normalisation factor and limit.
The precise statements are given in Section \ref{sec1.4} below.
These results extend the CLTs of \cite{Yang25}
from the instantaneous linear functionals $\langle f,X_{t}\rangle$ to the weighted occupation times $\langle f,Y_{t}\rangle$,
and may be viewed as superprocess analogues, in the
supercritical setting, of the functional limit theorems for occupation time
fluctuations of (sub)critical branching systems mentioned above.

The fluctuations of weighted occupation times studied here do not follow directly from the fluctuation theorems for $X_{t}$, and the passage from instantaneous linear functionals to occupation times brings in new features in the analysis. In particular, a one-time limit theorem cannot simply be integrated in time, since
$$\e^{-\lambda_{1}t}\langle f,Y_{t}\rangle=\int_{0}^{t}\e^{-\lambda_{1}(t-s)}\cdot \e^{-\lambda_{1}s}\langle f,X_{s}\rangle\mathrm{d}s,$$
and the fluctuations of the integrand accumulate rather than average out.
A first difference appears at the level of second moments: the covariance formulae of Lemma \ref{lem:moment formulae} give the exact two-point covariance structure of the occupation times. While the
covariance of $\langle f,X_t\rangle$ and $\langle g,X_{t}\rangle$ is expressed as a single time integral
$\int_0^t T_{t-s}(\vartheta[T_sf,T_{s}g])(x)\mathrm{d}s$ (see \eqref{eq:covforX_{t}} below), the covariance of
$\langle f,Y_t\rangle$ and $\langle g,Y_{t}\rangle$ involves the time-integrated operator
$M_t[f]=\int_0^t T_s f\mathrm{d}s$ through
$\int_0^t T_{t-s}\big(\vartheta[M_{s}[f],M_{s}[g]]\big)(x)\mathrm{d}s$ (see Lemma \ref{lem:moment formulae} below).
Hence, the precise asymptotics of $\Cov_{\delta_{x}}(\langle f,Y_{t}\rangle,\langle g,Y_{t}\rangle)$ in the three regimes require new estimates. It turns out that, despite this additional integration, the trichotomy persists with exactly the same normalisation factors as the linear functionals studied in \cite{Yang25}. However, the limiting laws are different from those in \cite{Yang25}. The covariance kernel of the Gaussian process $G(f)$ in Theorem \ref{them1} is constructed from the covariance structure of the occupation times, and contains cross terms with no counterpart for instantaneous linear functionals. Moreover, the limits in Theorems \ref{them2} and \ref{them3} carry the additional factors $2/\lambda_{1}$ and $\lambda^{-1}$, respectively, reflecting the time averaging $\int_{0}^{t}\e^{a(s-t)}\mathrm{d}s\to a^{-1}$ inherent in the occupation time, with $a=\lambda_{1}/2$ in Theorem \ref{them2} and $a=\lambda$ in Theorem \ref{them3}. A detailed discussion of the main ingredients of the proof, and a comparison of our approach with the $L^{2}$-spectral method of \cite{RSZ17} and with the results of \cite{Yang25} for linear functionals, is given in Section~\ref{sec1.5}. We also mention that our Perron--Frobenius assumption (H2) is weaker than the uniform $h$-transform type condition imposed in \cite{Yang25} (its Assumption 1.1), so that our results apply to a larger class of models; see Example \ref{eg1.4} and Remark \ref{rm1} below.

Finally, we briefly indicate the statistical potential of our results, whose full development is left to future work. The principal eigenvalue $\lambda_{1}$ (also known as the Malthusian growth rate of the population) is typically a function of parameters of the branching mechanism that cannot be observed directly, whereas the trajectory of $X$ can.
For a test function $g$ with $\langle g,\widetilde{\varphi}\rangle\not=0$, the statistic $\hat{\lambda}_{t}(g):=\langle g,X_{t}\rangle/\langle g,Y_{t}\rangle$ is, by Lemma \ref{lem2.2} below, a consistent estimator of $\lambda_{1}$ on the event $\{W^{\varphi}_{\infty}>0\}$. Its convergence rate and limit law depend on the spectral decay of $g$ and are expected to reflect the trichotomy of Theorems \ref{them1}--\ref{them3}. Their precise description is left to future work. Moreover, for $g=\varphi$, writing $\mathrm{M}_{t}(\varphi)=\langle\varphi,Y_{t}\rangle-\lambda_{1}^{-1}\langle\varphi,X_{t}\rangle$ for the martingale of Lemma \ref{lem3.n1}, we have $\hat{\lambda}_{t}(\varphi)-\lambda_{1}=-\lambda_{1}\mathrm{M}_{t}(\varphi)/\langle\varphi,Y_{t}\rangle$. The stable limit theorem for $\mathrm{M}_{t}(\varphi)$ together with the LLN \eqref{lln-intro} therefore yields that $\hat{\lambda}_{t}(\varphi)-\lambda_{1}$ is asymptotically normal on $\{W^{\varphi}_{\infty}>0\}$ at rate $\e^{\lambda_{1}t/2}$ with explicit mixed normal variance $\lambda_{1}\langle\vartheta[\varphi],\widetilde{\varphi}\rangle/W^{\varphi}_{\infty}$, so that $\lambda_{1}$ can be estimated from a single trajectory of $X$. This continues the classical line of research on estimating the growth parameters of a branching system from a single trajectory; see the monograph \cite{Guttorp91} for a comprehensive treatment.

The rest of this paper is organised as follows. In the remainder of this
section, we introduce the non-local branching superprocess, the weighted
occupation times, and state our assumptions,
examples and main
results.
We also provide
a detailed discussion of the approach and related literature
at the end of this section.
Section \ref{sec2} contains the LLN and the asymptotic behaviour of the
first two moments. Section \ref{sec3} is devoted to the proofs of the main results.
The convergence of finite-dimensional distributions and tightness are established in Sections \ref{sec3.1} and \ref{sec3.2}, respectively,
and the main theorems are proved in Section \ref{sec3.3}.
The appendix contains the proof of the covariance formulae and two auxiliary
results on stable convergence of martingales and on tightness.

\medskip

\noindent\textit{Notation:} Throughout this paper, we use ``:=" as a way of
definition.
The letter $c$, with or without subscript, denotes a finite positive constant whose value may vary from place to place.
Let $E$ be a Lusin topological space with Borel $\sigma$-algebra $\mathcal{B}(E)$. Let $E_{\partial}:=E\cup\{\partial\}$ be the one-point compactification of $E$. Any function on $E$ will be automatically extended to $E_{\partial}$ by setting $f(\partial)=0$.
We use $\mathcal{B}_{b}(E)$ (resp. $\mathcal{B}^{+}(E)$) to denote the space of bounded (resp. non-negative) measurable functions on $(E,\mathcal{B}(E))$.
Let $\mf$ denote the space of finite Borel measures on $E$ topologized by weak convergence.
For a measure $\mu$ and any integrable function $f$, we write
$$\langle f,\mu\rangle=\mu(f)=\int f(x)\mu(\mathrm{d}x).$$
For $k\ge 1$ and a probability space $(\Omega,\mathcal{G},\mathrm{P})$ we use $L^{k}(\mathrm{P})$ to denote the space of random variables $\eta$ on $\Omega$ satisfying that $\int_{\Omega}|\eta|^{k}d\mathrm{P}<+\infty$. If $X,Y$ are random variables, we define $\mathrm{Cov}(X,Y):=\mathrm{E}\big[(X-\mathrm{E}[X])(Y-\mathrm{E}[Y])\big]$, and $\mathrm{Var}[X]:=\mathrm{Cov}(X,X)=\mathrm{E}\big[\big(X-\mathrm{E}[X]\big)^{2}\big]$.
For any two positive functions $f$ and $g$,
we use $f\stackrel{c}\lesssim g$ to denote that there is a positive constant
$c$ such that $f\le c g$ on their common domain of definition. We also
write ``$\lesssim$" if $c$ is unimportant or
understood.
We say a sequence of functions $\{f_{n}:n\ge 1\}$ converges boundedly and pointwise to $f$ if there is a constant $c\ge 0$ such that $\|f_{n}\|_{\infty}\le c$ for $n$ sufficiently large and $f_{n}(x)\to f(x)$ as $n\to+\infty$ for all $x\in E$. We use `$\stackrel{b.p.}{\longrightarrow}$' for bounded pointwise convergence, `$\rightrightarrows$' for uniform convergence, `$\stackrel{d}{\to}$' for convergence in distribution, and `$\stackrel{f.d.d.}{\longrightarrow}$' for convergence in finite-dimensional distributions.

\subsection{Non-local branching superprocess}\label{sec1.1}

Suppose $\xi=(\Omega,\mathcal{H},\mathcal{H}_{t},\xi_{t},\Pi_{x})$ is a Borel right process in $E$ with
transition semigroup $(P_{t})_{t\ge 0}$. Here $\{\mathcal{H}_{t}:t\ge 0\}$ is the minimal admissible filtration.
In this paper, we consider a superprocess $X:=(X_{t})_{t\ge 0}$ with spatial motion $\xi$ and a non-local branching mechanism $\psi$ given by
\begin{eqnarray}
\psi(x,f)&=&a(x)f(x)+b(x)f(x)^{2}-\eta(x,f)\nonumber\\
&&\quad+\int_{\mfo }\left(\e^{-\nu(f)}-1+\nu(\{x\})f(x)\right)H(x,d\nu)\quad\forall x\in E,\ f\in\mathcal{B}^{+}(E),\label{def:branching mechanism}
\end{eqnarray}
where $a\in\mathcal{B}_{b}(E)$, $b\in\mathcal{B}^{+}_{b}(E)$, $\eta(x,dy)$ is a bounded kernel on $E$, and $\left(\nu(1)\wedge \nu(1)^{2}+\nu_{x}(1)\right)H(x,d\nu)$ is a bounded kernel from $E$ to $\mfo :=\mf\setminus \{0\}$ with $\nu_{x}(\mathrm{d}y)$ defined by $\nu_{x}(A):=\nu(A\setminus\{x\})$ for all $A\in\mathcal{B}(E)$.
To be specific, $X$ is a $\mf$-valued Markov process satisfying that for every $f\in\mathcal{B}^{+}_{b}(E)$ and $\mu\in\mf$,
\begin{equation}\label{eq1}
\pp_{\mu}\left[\e^{-\langle f,X_{t}\rangle}\right]=\e^{-\langle V_{t}f,\mu\rangle},
\end{equation}
where $V_{t}f(x):=-\log\pp_{\delta_{x}}\left[\e^{-\langle f,X_{t}\rangle}\right]$ is the unique nonnegative locally bounded solution to the integral equation
\begin{equation}\label{eq:integral equation for V_{t}}
V_{t}f(x)=P_{t}f(x)-\Pi_{x}\left[\int_{0}^{t}\psi(\xi_{s},V_{t-s}f)\mathrm{d}s\right].
\end{equation}
By `locally bounded' we mean that $\sup_{x\in E,t\in [0,T]}|V_{t}f(x)|<+\infty$ for every $T\in (0,+\infty)$.
Such a process is defined in \cite{Li11} via its log-Laplace functional and is referred to as the $(P_{t},\psi)$-superprocess.

By \cite[Theorem 5.12]{Li11}, a $(P_{t},\psi)$-superprocess $X$ has a Borel right realization in $\mf$. We denote by $\mathcal{W}^{+}_{0}$ the space of c\`{a}d\`{a}g paths from $[0,+\infty)$ to $\mf$ having zero as a trap. Here, we assume that $X$ is the coordinate process in $\mathcal{W}^{+}_{0}$ and  $(\mathcal{F}_{t})_{t\in[0,\infty]}$ is the filtration generated by the coordinate process, which is completed  with the class of $\pp_{\mu}$-negligible measurable sets for every $\mu\in\mf$.

\subsection{Weighted occupation times and related moment formulae}\label{sec1.2}
{
If $X$ is a $(P_{t},\psi)$-superprocess, the running occupation measure $Y_{t}$ of $X$ is defined by $Y_{t}(\cdot):=\int_{0}^{t}X_{s}(\cdot)\mathrm{d}s$. For $f\in\mathcal{B}_{b}(E)$,
$$\langle f,Y_{t}\rangle:=\int_{0}^{t}\langle f,X_{s}\rangle\mathrm{d}s$$
is called a weighted occupation time of the superprocess. \cite[Corollary 5.17]{Li11} has established the following characterization of Laplace transform of the occupation measure:
for every $f,g\in\mathcal{B}^{+}_{b}(E)$ and $\mu\in\mf$,
\begin{equation}\label{eq:U_{t}}
\pp_{\mu}\big[\exp\{-\langle f,X_{t}\rangle -\langle g,Y_{t}\rangle\}\big]=\e^{-\langle U_{t}[f,g],\mu\rangle},
\end{equation}
where $(t,x)\mapsto U_{t}[f,g](x)$ is the unique locally bounded nonnegative solution of
$$U_{t}[f,g](x)=P_{t}f(x)-\Pi_{x}\big[\int_{0}^{t}\psi(\xi_{s},V_{t-s}[f,g])\mathrm{d}s\big]+\Pi_{x}\big[\int_{0}^{t}g(\xi_{s})\mathrm{d}s\big].$$
Our assumptions on the branching mechanism guarantee the corresponding superprocess has finite first moments in the sense that, for every $f\in\mathcal{B}_{b}(E)$ and $\mu\in\mf$,
$$\pp_{\mu}\big[\langle f,X_{t}\rangle\big]=\langle T_{t}f,\mu\rangle\mbox{ and }\pp_{\mu}\big[\langle f,Y_{t}\rangle\big]=\langle M_{t}[f],\mu\rangle.$$
where
$$T_{t}f(x)=\pp_{\delta_{x}}\big[\langle f,X_{t}\rangle\big]\mbox{ and }M_{t}[f](x)=\pp_{\delta_{x}}\big[\langle f,Y_{t}\rangle\big]=\int_{0}^{t}T_{s}f(x)\mathrm{d}s.$$
Furthermore, it is proved by \cite[Proposition 2.27]{Li11} that $T_{t}f(x)$ is the unique locally bounded solution to the integral equation
\begin{equation}\label{eq:many-to-one}
T_{t}f(x)=P_{t}f(x)-\Pi_{x}\left[\int_{0}^{t}a(\xi_{s})T_{t-s}f(\xi_{s})\mathrm{d}s\right]+\Pi_{x}\left[\int_{0}^{t}m[T_{t-s}f](\xi_{s})\mathrm{d}s\right],
\end{equation}
where $m[f](y):=\eta(y,f)+\int_{\mfo }\nu_{x}(f)H(y,d\nu)$ for $y\in E$. We can easily show by the above equation and Gronwall's inequality that there exists some constant $c_{0}>0$ such that
\begin{equation}\label{eq0}
\|T_{t}f\|_{\infty}\le \e^{c_{0}t}\|f\|_{\infty}\quad\forall t\ge 0,\ f\in\mathcal{B}_{b}(E).
\end{equation}

We now introduce the following boundedness assumption on the branching mechanism.
\begin{description}
\item{(Mk)}
$\sup_{x\in E}\int_{\nu(1)>1}\nu(1)^{k}H(x,d\nu)<+\infty$.
\end{description}
We have already assumed in the definition of superprocess that $\int_{\mathcal{M}^{0}(E)}\nu(1)\wedge\nu(1)^{2}H(x,\mathrm{d}\nu)$ is uniformly bounded for $x\in E$, and hence (M1) holds automatically.
For $k>1$,
{(Mk)} is a sufficient condition ensuring the superprocess has finite $k$th moments.
In particular, when {(M2)} holds, let us define
$$\vartheta[f,g](x):=2b(x)f(x)g(x)+\int_{\mfo }\nu(f)\nu(g)H(x,d\nu)\mbox{ and }\vartheta[f](x):=\vartheta[f,f](x)$$
for $f,g\in \mathcal{B}_{b}(E)$ and $x\in E$, then it follows by \cite[Proposition 2.38]{Li11} that, $\pp_{\mu}\left[\langle f,X_{t}\rangle^{2}\right]$ is finite for every $t\ge 0$, $\mu\in\mf$ and $f\in\mathcal{B}_{b}(E)$, and
\begin{equation}\label{eq:second moment}
\pp_{\mu}\left[\langle f,X_{t}\rangle^{2}\right]=\langle T_{t}f,\mu\rangle^{2}+\langle\int_{0}^{t}T_{t-s}(\vartheta[T_{s}f])\mathrm{d}s,\mu\rangle.
\end{equation}
The above identity yields that
\begin{equation}\label{eq:variance}
\mathrm{Var}_{\delta_{x}}\left[\langle f,X_{t}\rangle\right]=\int_{0}^{t}T_{t-s}(\vartheta[T_{s}f])(x)\mathrm{d}s\mbox{ and }\mathrm{Var}_{\mu}\left[\langle f,X_{t}\rangle\right]=\langle \mathrm{Var}_{\delta_{\cdot}}\left[\langle f,X_{t}\rangle\right],\mu\rangle.
\end{equation}
Moreover, since $\Cov_{\mu}(\langle f,X_{t}\rangle,\langle g,X_{t}\rangle)=4^{-1}\big(\Var_{\mu}[\langle f+g,X_{t}\rangle]-\Var_{\mu}[\langle f-g,X_{t}\rangle]\big)$, we have
\begin{equation}\label{eq:covforX_{t}}
\Cov_{\delta_{x}}(\langle f,X_{t}\rangle,\langle g,X_{t}\rangle)=\int_{0}^{t}T_{t-s}\big(\vartheta[T_{s}f,T_{s}g]\big)(x)\mathrm{d}s
\end{equation}
and $\Cov_{\mu}(\langle f,X_{t}\rangle,\langle g,X_{t}\rangle)=\langle\Cov_{\delta_{\cdot}}(\langle f,X_{t}\rangle,\langle g,X_{t}\rangle),\mu\rangle$.

The following lemma extends the covariance formulae \eqref{eq:variance}-\eqref{eq:covforX_{t}} from $\langle f,X_{t}\rangle$ to the weighted occupation times.
\begin{lemma}\label{lem:moment formulae}
Suppose (M2) holds. For $f,g\in\mathcal{B}^{+}_{b}(E)$, $x\in E$ and $t\ge 0$,
\begin{align*}
\mathrm{Cov}_{\delta_{x}}\big(\langle f,Y_{t}\rangle,\langle g,Y_{t}\rangle\big)=&\int_{0}^{t}T_{s}\big(\vartheta\big[M_{t-s}[f],M_{t-s}[g]\big]\big)(x)\mathrm{d}s,\\
\mathrm{Cov}_{\delta_{x}}\big(\langle f,Y_{t}\rangle,\langle g,X_{t}\rangle\big)=&\int_{0}^{t}T_{s}\big(\vartheta[M_{t-s}[f],T_{t-s}g]\big)(x)\mathrm{d}s.
\end{align*}
Moreover, for every $\mu\in\mf$,
\begin{align*}
\Cov_{\mu}(\langle f,Y_{t}\rangle,\langle g,Y_{t}\rangle)=&\langle \Cov_{\delta_{\cdot}}(\langle f,Y_{t}\rangle,\langle g,Y_{t}\rangle),\mu\rangle,\\
\Cov_{\mu}(\langle f,X_{t}\rangle,\langle g,Y_{t}\rangle)=&\langle \Cov_{\delta_{\cdot}}(\langle f,Y_{t}\rangle,\langle g,X_{t}\rangle),\mu\rangle.
\end{align*}
\end{lemma}
The proof of this lemma is deferred to Appendix \ref{sec:A.1}.

}

\subsection{Assumptions and examples}\label{sec1.3}
The Borel right process $\xi$ generates an intrinsic topology, called the fine topology, which is the smallest topology on $E$ that makes all $\alpha$-excessive functions ($\alpha\ge 0$)  relative to $\xi$ continuous; see, e.g., \cite[Appendix A.3]{Li11}. We use $C^{\xi}_{b}(E)$ to denote the space of bounded measurable functions on $E$ that are finely continuous. It is known that a bounded measurable function $f$ belongs to $C^{\xi}_{b}(E)$ if and only if $t\mapsto f(\xi_{t})$ is a.s. right continuous on $[0,+\infty)$.

Throughout this paper we assume the following conditions hold.
\begin{description}
\item{(H1)} The operators $f\mapsto\psi(\cdot,f)$ and $f\mapsto -a f+m[f]$ preserve $C^{\xi}_{b}(E)$.
\item{({H2})} There exist a real number $\lambda_{1}>0$, a corresponding positive bounded function $\varphi$ and a finite measure $\widetilde{\varphi}$ on $E$, with normalisation $\langle \varphi,\widetilde{\varphi}\rangle=1$, such that for $f\in\mathcal{B}^{+}_{b}(E)$, $\mu\in\mf$ and $t\ge 0$,
    \begin{equation}\label{H2.1}
    \langle T_{t}\varphi,\mu\rangle=\e^{\lambda_{1}t}\langle \varphi,\mu\rangle\mbox{ and }\langle T_{t}f,\widetilde{\varphi}\rangle=\e^{\lambda_{1}t}\langle f,\widetilde{\varphi}\rangle.
    \end{equation}
    Further, we define
$$\triangle_{t}:=\sup_{x\in E,f\in\mathcal{B}^{+}_{1}(E)}\left|\e^{-\lambda_{1}t}T_{t}f(x)-\langle f,\widetilde{\varphi}\rangle\varphi(x)\right|\quad\forall t\ge 0.$$
Here $\mathcal{B}^{+}_{1}(E)$ denotes the set of nonnegative measurable functions on $E$ that are bounded by $1$.
We assume that
\begin{equation}\label{H2.2}
\lim_{t\to+\infty}\triangle_{t}=0.
\end{equation}
\end{description}

Let us give some remarks on the assumptions here.
Assumption (H1), inherited from \cite{PY20}, serves as a regularity condition on the branching mechanism. This condition is required to ensure the existence of a stochastic integral representation for the non-local superprocess (see, e.g., \cite[Proposition 4.3]{PY20}) which will be employed in our argument.

Assumption ({H2}) is a Perron--Frobenius assumption that ensures the existence of the leading eigenvalue and corresponding eigenfunction. The superprocess is known as {\it supercritical} when $\lambda_{1}>0$. Without restriction on the sign of $\lambda_{1}$, ({H2}) has been recently named the Asmussen--Hering class of branching processes (see, \cite{HK23}).
Assumption ({H2}) implies in particular that for every $f\in \mathcal{B}_{b}(E)$,
\begin{equation}\label{ieq:H3}
\|\e^{-\lambda_{1}t}T_{t}f-\langle f,\widetilde{\varphi}\rangle\varphi\|_{\infty}\le 2\|f\|_{\infty}\triangle _{t}\to 0\mbox{ as }t\to+\infty.
\end{equation}
Moreover, \eqref{H2.2} of (H2) combined with \eqref{eq0} yields that $\sup_{t\ge 0}\triangle_{t}<+\infty$.
Consequently, there exists some constant $c_{1}>0$ such that
\begin{equation}\label{eq1.9}
\big|\e^{-\lambda_{1}t}T_{t}f(x)\big|\le \|f\|_{\infty}\left(2\triangle_{t}+\|\varphi\|_{\infty}\langle 1,\widetilde{\varphi}\rangle\right)\le c_{1}\|f\|_{\infty}\quad\forall x\in E,\ t\ge 0,\ f\in\mathcal{B}_{b}(E).
\end{equation}
Although (H2) may seem like a strong assumption, it can be shown to hold for a large class of superprocesses.
	Here we give some concrete examples to illustrate the generality of this setting.
	
\begin{example}\rm
A continuous-state branching process (CSBP) may be considered as the simplest example.
In this setting, we can take $E=\{0\}$, $\xi$ to be the Markov process which remains at $0$, and a spatially-independent branching mechanism $\psi$ given by
$$\psi(\lambda)=-\alpha\lambda+\frac{1}{2}\sigma^{2}\lambda^{2}+\int_{0}^{+\infty}\big(\e^{-\lambda r}-1+\lambda r\big)\pi(\mathrm{d}r)\quad\forall \lambda\ge 0,$$
where $\alpha>0$, $\sigma^{2}\ge 0$ and $\pi$ is a Radon measure on $(0,+\infty)$ such that $\int_{0}^{+\infty}(r\wedge r^{2})\pi(\mathrm{d}r)<+\infty$.
Condition (H2) trivially holds if we take $\lambda_{1}=\alpha$, $\varphi=1$ and $\widetilde{\varphi}=\delta_{\{0\}}$.
\end{example}
	
	\begin{example}\label{eg1.3}\rm	
		Multitype continuous-state branching processes (MCSBP) are the natural analogues of multitype Galton-Watson processes in the context of continuous mass. They are introduced in \cite{BLP15} and correspond to the setting that $E$ is a finite set of types. In essence, (H2) is the classical Perron-Frobenius behaviour for the matrix of the mean semigroup, and will hold under a simple irreducibility assumption; see, e.g., \cite[Example 3.7]{PY20}.
\end{example}

\begin{example}\label{eg1.4}\rm
			Assume \eqref{H2.1} holds. Let
			\begin{equation}
				\tilde{\triangle}_{t}:=\sup_{x\in E,\,f\in\mathcal{B}^{+}_{1}(E)}\big|\varphi(x)^{-1}\e^{-\lambda_{1}t}T_{t}f(x)-\langle f,\widetilde{\varphi}\rangle\big|.
			\end{equation}
Since $\varphi$ is bounded and positive, \eqref{H2.1} is implied by the stronger convergence
\begin{equation}\label{H2.2'}
\tilde{\triangle}_{t}\to 0\mbox{ as }t\to+\infty.
\end{equation}
This type of convergence condition (with or without further conditions on the convergence rate) has recently been one of the fundamental assumptions imposed in
a series of papers that investigate the limiting theories of non-local spatial branching processes, see, e.g.
\cite{HHKW,HKMPR,Yang25} and the references therein, and has been verified for a wide class of non-local spatial branching processes, including the neutron branching process, branching diffusion (superdiffusion) on a compact domain, and others. For a detailed discussion of these models, we refer the reader to the aforementioned papers.
		\end{example}

\begin{remark}\label{rm1}
\rm
Beyond \eqref{H2.1} and \eqref{H2.2'}, it is further assumed in \cite{HHKW,HKMPR,Yang25} that
$$\sup_{t>0}\tilde{\triangle}_{t}<+\infty.$$
Compared with the uniform boundedness of $\triangle_{t}$ which comes directly from \eqref{H2.2} and \eqref{eq0}, such property of $\tilde{\triangle}_{t}$ can fail or be difficult to verify. In fact, in some interesting models, such as, multitype superdiffusion on a bounded domain (\cite[Example 3.8]{PY20}) and super-$\alpha$-stable process on a bounded smooth domain (\cite[Example 3.9]{PY20}), \eqref{H2.2'} can be easily verified by the intrinsic ultracontractivity of the underlying mean semigroup, but the boundedness of $\tilde{\triangle}_{t}$ fails near $t=0$.
Therefore, condition (H2) is formulated in the present paper so as to cover these models.
\end{remark}

\subsection{Main results}\label{sec1.4}
For the remainder of this paper, we assume (H1), (H2) and (M2) hold. Other conditions needed will be stated explicitly.
Using \eqref{H2.1} and the Markov property, one can easily show that
$${ W^{\varphi}_{t}}:=\e^{-\lambda_{1}t}\langle \varphi,X_{t}\rangle \quad\forall t\ge 0$$
is a nonnegative $\pp_{\mu}$-martingale with respect to the filtration $(\mathcal{F}_{t})_{t\ge 0}$ for every $\mu\in\mf$.
Let ${ W^{\varphi}_{\infty}}$ be the martingale limit.
We will show in Lemma \ref{lem2.2} below that, under (H1), (H2) and (M2), the occupation measure $Y_{t}$ satisfies the LLN in the sense of \eqref{lln-intro}.
The goal of this paper is to prove the corresponding functional central limit theorem.

To describe the main results, we define a parameter $\rho(f)$ for $f\in\mathcal{B}_{b}(E)$, and subclasses $\mathcal{D}_{\mathrm{S}}$, $\mathcal{D}(\lambda,r)$ of $\mathcal{B}_{b}(E)$ for $\lambda,r\ge 0$ by the following relations: for $f\in \mathcal{B}_{b}(E)$, let $\tilde{f}:=f-\langle f,\widetilde{\varphi}\rangle\varphi$.
Define
\begin{align*}
{\rho(}f):=&-\limsup_{t\to+\infty}\frac{1}{t}\log\sup_{x\in E}\big|\e^{-\lambda_{1}t}T_{t}\tilde{f}(x)\big|,\\
\mathcal{D}_{\mathrm{S}}:=&\{f\in\mathcal{B}_{b}(E):\ {\rho}(f)>\lambda_{1}/2\},\\
\mathcal{D}(\lambda,r):=&\{f\in\mathcal{B}_{b}(E):\ t^{-r}\e^{-\lambda t}T_{t}\tilde{f}\rightrightarrows f^{*}\mbox{ for some }f^{*}\in\mathcal{B}_{b}(E)\}.
\end{align*}
In what follows, for $f\in\mathcal{D}(\lambda,r)$, we use $f^{*}_{\lambda,r}$ to denote the corresponding limiting function. When no confusion arises, we omit the subscript and simply write $f^{*}$.
\begin{remark}\rm
\begin{description}
\item[(i)] The parameter $\rho(f)$ measures the exponential decay rate of the centered function $\tilde{f}$ under the rescaled mean semigroup $\e^{-\lambda_{1}t}T_{t}$.
Intuitively, the larger $\rho(f)$ is, the faster $\tilde{f}$ decays under $\e^{-\lambda_{1}t}T_{t}$, and hence the smaller the fluctuations of $\langle f,Y_{t}\rangle$ are.
By \eqref{ieq:H3} and the fact that
$\sup_{t\ge 0}\triangle_{t}<+\infty$, we have ${\rho(}f)\in [0,+\infty]$, and particularly $\rho(\varphi)=+\infty$.
\item[(ii)] From the definition, for fixed $r$ (resp. $\lambda$), the class $\mathcal{D}(\lambda,r)$ becomes larger as $\lambda$ (resp. $r$) increases.
Moreover, by \eqref{ieq:H3}, we have $\e^{-\lambda_{1}t}T_{t}\tilde{f}\rightrightarrows 0$ for every $f\in\mathcal{B}_{b}(E)$, and hence $\mathcal{D}(\lambda_{1},0)=\mathcal{B}_{b}(E)$.
Furthermore, Lemma \ref{lem0} (i) below implies that every function $f$ in $\mathcal{D}_{\mathrm{S}}$ also belongs to $\mathcal{D}(\lambda_{1}/2,0)$ with the limiting function equal to $0$. Therefore, the following relation holds:
$$\mathcal{D}_{\mathrm{S}}\subset \mathcal{D}(\lambda_{1}/2,0)\subset\mathcal{D}(\lambda,r)\subset \mathcal{D}(\lambda_{1},0)=\mathcal{B}_{b}(E)\quad\forall (\lambda,r)\in [\lambda_{1}/2,\lambda_{1})\times [0,+\infty).$$
Theorems \ref{them1}, \ref{them2} and \ref{them3} below are stated for $f$ in $\mathcal{D}_{\mathrm{S}}$, $\mathcal{D}(\lambda_{1}/2,r)$ and $\mathcal{D}(\lambda,r)$ with $(\lambda,r)\in (\lambda_{1}/2,\lambda_{1})\times [0,+\infty)\cup\{(\lambda_{1},0)\}$, respectively. Moreover, by the above inclusions, the limiting function $f^{*}$ vanishes whenever $f$ is viewed in a strictly larger class. Hence, a nondegenerate limit is obtained only when the normalisation matches the exact decay rate of $\tilde{f}$. All three regimes are realized, for instance, in the setting of Example \ref{eg1.3} when the mean matrix has only real eigenvalues (as is the case for symmetric mean matrices). In this setting, an eigenvector with eigenvalue $\lambda_{k}<\lambda_{1}/2$, $\lambda_{k}=\lambda_{1}/2$ or $\lambda_{k}\in (\lambda_{1}/2,\lambda_{1})$ then belongs to $\mathcal{D}_{\mathrm{S}}$, $\mathcal{D}(\lambda_{1}/2,0)$ or $\mathcal{D}(\lambda_{k},0)$, respectively. Moreover, if a real eigenvalue has a Jordan block of size $r+1$, then the top generalized eigenvector of the block belongs to $\mathcal{D}(\lambda_{k},r)$. In particular, the classes with $r>0$, which give rise to the polynomial correction factors in Theorems \ref{them2} and \ref{them3}, are non-empty in this model. We refer to \cite[Section 1.3]{Yang25} for the infinite-dimensional counterpart of this spectral picture.
\end{description}
\end{remark}

Let $\mathcal{C}[0,+\infty)$ denote the space of continuous functions from $[0,+\infty)$ to $\R$ equipped with the local uniform topology for which it is a Polish space, and let $\mathcal{P}(\mathcal{C}[0,+\infty))$ be the space of all probability measures on $\mathcal{C}[0,+\infty)$ endowed with the weak topology.
We consider the fluctuation of the occupation time defined by \eqref{Z-intro}.
The process $Z^{(t)}(f):=(Z^{(t)}_{s}(f))_{s\ge 0}$ may be considered as a random variable taking values in $\mathcal{C}[0,+\infty)$ and its law is an element of $\mathcal{P}(\mathcal{C}[0,+\infty))$. To describe the limit of $Z^{(t)}(f)$ in $\mathcal{P}(\mathcal{C}[0,+\infty))$, we need the following covariance functions.
For $f\in\mathcal{B}_{b}(E)$ and $t\ge 0$, put
\begin{align*}
\sigma_{\tilde{f}}(\infty,t):=&\int_{0}^{+\infty}\e^{-\lambda_{1}r}\langle\vartheta\big[M_{r+t}[\tilde{f}],
M_{r}[\tilde{f}]\big],\widetilde{\varphi}\rangle\mathrm{d}r,\\
\varsigma_{\varphi,\tilde{f}}(\infty,t):=&-\frac{1}{\lambda_{1}}\Big(\int_{0}^{t}\langle\vartheta\big[M_{r}[\tilde{f}],\varphi\big]
,\widetilde{\varphi}\rangle\mathrm{d}r+\int_{0}^{+\infty}\e^{-\lambda_{1}r}\langle\vartheta[M_{r+t}[\tilde{f}],\varphi]
,\widetilde{\varphi}\rangle\mathrm{d}r\Big),\\
\kappa_{\varphi}(t):=&\lambda^{-3}_{1}(2+\lambda_{1}t)\langle\vartheta[\varphi],\widetilde{\varphi}\rangle.
\end{align*}
We will show in Lemma \ref{lem3.n2} below that both $\sigma_{\tilde{f}}(\infty,t)$ and $\varsigma_{\varphi,\tilde{f}}(\infty,t)$ are finite constants.
We are now ready to state our main results which are derived under assumptions (H1), (H2) and (M2).

\begin{theorem}\label{them1}
For every $f\in\mathcal{D}_{\mathrm{S}}$, under $\pp_{\delta_{x}}$,
$$\big(\e^{\frac{\lambda_{1}}{2}(t+s)}Z^{(t)}_{s}(f)\big)_{s\ge 0}\stackrel{d}{\to}\sqrt{W^{\varphi}_{\infty}}G(f)\mbox{ in }\mathcal{C}[0,+\infty)\mbox{ as }t\to+\infty,$$
where $G(f):=(G_{s}(f))_{s\ge 0}$ is a continuous centered Gaussian process with covariance
$$\Cov(G_{s}(f),G_{t}(f))=\e^{-\frac{\lambda_{1}}{2}|s-t|}\big(\sigma_{\tilde{f}}(\infty,|s-t|)+\langle f,\widetilde{\varphi}\rangle^{2}\kappa_{\varphi}(|s-t|)+
\langle f,\widetilde{\varphi}\rangle({\varsigma}_{\varphi,\tilde{f}}(\infty,0)+{\varsigma}_{\varphi,\tilde{f}}(\infty,|s-t|))\big)$$
for all $s,t\ge 0$, and $G(f)$ is independent of $W^{\varphi}_{\infty}$.
\end{theorem}

\begin{theorem}\label{them2}
Assume, in addition, (M4) holds. Then for every $r\ge 0$ and $f\in \mathcal{D}(\lambda_{1}/2,r)$ with limiting function $f^{*}$, under $\pp_{\delta_{x}}$,
$$\big((t+s)^{-(r+\frac{1}{2})}\e^{\frac{\lambda_{1}}{2}(t+s)}Z^{(t)}_{s}(f)\big)_{s\ge 0}\stackrel{d}{\to}\frac{2}{\lambda_{1}}\sqrt{\frac{\langle\vartheta[f^{*}],\widetilde{\varphi}\rangle}{1+2r}W^{\varphi}_{\infty}}C\mbox{ in }\mathcal{C}[0,+\infty)\mbox{ as }t\to+\infty,$$
where $C\equiv N$ is a constant process with $N$ a standard normal random variable independent of $W^{\varphi}_{\infty}$.
\end{theorem}

\begin{theorem}\label{them3}
Suppose $(\lambda,r)\in (\lambda_{1}/2,\lambda_{1})\times [0,+\infty)\cup\{(\lambda_{1},0)\}$. For every $f\in \mathcal{D}(\lambda,r)$ with limiting function $f^{*}$, set $W^{f^{*}}_{t}
:=\e^{-\lambda t}\langle f^{*},X_{t}\rangle$. Then the following hold.
\begin{description}
\item[(i)] $\{W^{f^{*}}_{t}:\ t\ge 0\}$ is an $L^{2}(\pp_{\delta_{x}})$-bounded martingale, and hence converges $\pp_{\delta_{x}}$-a.s. and in $L^{2}(\pp_{\delta_{x}})$ to a finite limit $W^{f^{*}}_{\infty}$. Furthermore, $\Var_{\delta_{x}}[W^{f^{*}}_{\infty}]=\varpi^{2}_{f^{*}}(x)$ where $\varpi^{2}_{f^{*}}(x):=\int_{0}^{+\infty}\e^{-2\lambda s}T_{s}\big(\vartheta[f^{*}]\big)(x)\mathrm{d}s\in [0,+\infty)$.
\item[(ii)] Under $\pp_{\delta_{x}}$,
$$\big((t+s)^{-r}\e^{(\lambda_{1}-\lambda)(t+s)}Z^{(t)}_{s}(f)\big)_{s\ge 0}\stackrel{d}{\to}\frac{1}{\lambda}W^{f^{*}}_{\infty}\mbox{ in }\mathcal{C}[0,+\infty)\mbox{ as }t\to+\infty.$$
\end{description}
\end{theorem}

\subsection{Heuristics and related literature}\label{sec1.5}
Let us start with a short discussion of the main ingredients of the proof. As usual, the convergence in $\mathcal{C}[0,+\infty)$ is established through convergence of finite-dimensional distributions and tightness.
Firstly, using the moment formulae for occupation times, we derive exact asymptotics for means, variances and covariances in each of the three regimes (Lemmas \ref{lem0} and \ref{lem1}). For the finite-dimensional distributions, we decompose $Z^{(t)}_{s}(f)$ into the centered part driven by $\tilde{f}$ and the eigenfunction part driven by $\varphi$, and compute the limiting characteristic functions by conditioning on $\mathcal{F}_{t}$.
A key observation is that, $\mathrm{M}_{t}(\varphi):=\langle \varphi,Y_{t}\rangle-\lambda^{-1}_{1}\langle \varphi,X_{t}\rangle$ is a c\`{a}dl\`{a}g martingale with predictable quadratic variation $\lambda^{-2}_{1}\langle\vartheta[\varphi],Y_{t}\rangle$. Combined with the LLN \eqref{lln-intro} and a stable central limit theorem for martingales due to Touati \cite{Tou92}, this yields the stable convergence of $\e^{-\frac{\lambda_{1}}{2}t}\mathrm{M}_{t}(\varphi)$ to a mixed normal limit (Lemma \ref{lem3.n1}).
Another key observation is that, by the regular branching property of superprocess,
the relevant linear combinations of the fluctuation process
are infinitely divisible random variables, and thus admit a canonical representation of their characteristic functions. The jump part of this representation is shown to vanish in the limit under suitable moment conditions. Hence, the problem is reduced to the convergence of certain random variances determined by the martingale limit $W^{\varphi}_{\infty}$. The critical regime is more delicate than the other two regimes, where a backward
recursion along the time sequence is used to peel off one increment at a time, and the
fourth moment condition (M4) is needed to control the L\'{e}vy kernels of the jump parts.
Tightness is proved separately for the centered part and the eigenfunction part. The former is handled by second moment increment estimates (Lemma \ref{lem4.1}) combined with a Billingsley-type tightness criterion (Lemma \ref{lemA.3}), and the latter is proved by
Aldous' criterion applied to the underlying martingale. The two parts are then merged by standard criteria for sums of tight families.

We end this section by comparing the above approach with the two works closest to ours, \cite{RSZ17} and \cite{Yang25}.
The paper \cite{RSZ17} provides functional CLTs of instantaneous linear functionals for supercritical superprocesses with symmetric spatial motions and local branching mechanisms.
Their proofs rely on an $L^{2}$-spectral decomposition: the mean semigroup is assumed to be compact with purely discrete spectrum and the test functions are expanded in a complete eigenfunction basis. Under these assumptions, one martingale $e^{\lambda_{k}t}\langle\varphi^{(k)}_{j},X_{t}\rangle$ is available for each eigenpair $(\lambda_{k},\varphi^{(k)}_{j})$, and the random centering of the fluctuations is assembled from the corresponding martingale limits. None of these ingredients
passes
to occupation times.
In our setting, the mean semigroup is assumed to display a natural Perron-Frobenius type behaviour, which allows for non-symmetric spatial motions and non-local branching mechanisms. $L^{2}$-spectral theory is not available in this setting. More importantly, the occupation measure admits no eigenfunction-type martingale. The fluctuation analysis instead is built on the compensated martingale $\mathrm{M}_{t}(\varphi)$ introduced above.
The tightness step also requires new estimates.
The second-moment increment bounds for the centered part (Lemma \ref{lem4.1}) have no instantaneous analogue and are controlled through the second-moment asymptotics of Lemmas \ref{lem0}--\ref{lem1}. The eigenfunction part, on the other hand, relies on Aldous' criterion applied to the martingale $\mathrm{M}_{t}(\varphi)$.

The paper \cite{Yang25} treats the same model as ours but for instantaneous linear functionals.
Compared with \cite{Yang25}, 
the conditioning argument based on infinitely divisible representations carries over to our proof of finite-dimensional convergence.
The difficulties specific to the occupation times arise in the moment asymptotics, covariance structure and limit identification.
The moment asymptotics for the occupation times (Lemma \ref{lem1}) must control the growing operator $M_{t}[f]$ and have no analogue in \cite{Yang25}.
The finite-dimensional convergence relies on the covariance structure of the occupation times, and the cross terms in the covariance kernels of Theorem \ref{them1} have no instantaneous counterpart. The limit processes themselves differ qualitatively: the additional factors $2/\lambda_{1}$ and $\lambda^{-1}$, and the martingale limits in Theorems \ref{them2} and \ref{them3}, arise from the time averaging $\int_{0}^{t}e^{a(s-t)}\mathrm{d}s\to a^{-1}$ and cannot be deduced from the results of \cite{Yang25}.

\section{Preliminaries}\label{sec2}

\subsection{A law of large numbers for the occupation measure}

\begin{lemma}\label{lem2.2}
Suppose $f_{t}$ and $f$ are bounded measurable functions on $E$ and $f_{t}\stackrel{b.p.}{\longrightarrow}f$ as $t\to+\infty$, then the following hold for every $\mu\in\mf$.
\begin{description}
\item[(i)]
$\e^{-\lambda_{1}t}\langle f_{t},X_{t}\rangle \to \langle f,\widetilde{\varphi}\rangle { W^{\varphi}_{\infty}}\mbox{ in }L^{2}(\pp_{\mu})\mbox{ as }t\to+\infty.$
\item[(ii)] $\lim_{t\to+\infty}\e^{-\lambda_{1}t}\int_{0}^{t}\langle f_{s},X_{s}\rangle \mathrm{d}s=\frac{1}{\lambda_{1}}\langle f,\widetilde{\varphi}\rangle W^{\varphi}_{\infty}\mbox{ in }L^{2}(\pp_{\mu}).$
\end{description}
\end{lemma}

\proof (i) Set $\tilde{f}_{t}:=f_{t}-\langle f_{t},\widetilde{\varphi}\rangle\varphi$. Then $\langle\tilde{f}_{t},\widetilde{\varphi}\rangle=0$ and $\e^{-\lambda_{1}t}\langle f_{t},X_{t}\rangle=\e^{-\lambda_{1}t}\langle\tilde{f}_{t},X_{t}\rangle+\langle f_{t},\widetilde{\varphi}\rangle W^{\varphi}_{t}$. Clearly, we have $\lim_{t\to+\infty}\langle f_{t},\widetilde{\varphi}\rangle=\langle f,\widetilde{\varphi}\rangle$ by the bounded convergence theorem.
Define
 \begin{equation}\label{def:Theta(x)}
\Theta(x):=\int_{0}^{+\infty}\e^{-2\lambda_{1} s}T_{s}(\vartheta[\varphi])(x)\mathrm{d}s\quad\forall x\in E.
\end{equation}
It is easy to verify by \eqref{eq1.9} that $\Theta$ is a nonnegative bounded function on $E$. By \eqref{eq:variance}, we have
$$\mathrm{Var}_{\delta_{x}}[{ W^{\varphi}_{t}}]
=\int_{0}^{t}\e^{-2\lambda_{1} s}T_{s}(\vartheta[\varphi])(x)\mathrm{d}s\mbox{ and }\mathrm{Var}_{\mu}[{ W^{\varphi}_{t}}]=\langle \mathrm{Var}_{\delta_{\cdot}}[{ W^{\varphi}_{t}}],\mu\rangle\quad\forall \mu\in\mathcal{M}(E).$$
Obviously, as $t\to+\infty$, $\mathrm{Var}_{\delta_{x}}[{ W^{\varphi}_{t}}]\uparrow \Theta(x)$ and $\mathrm{Var}_{\mu}[{ W^{\varphi}_{t}}]\uparrow \langle \Theta,\mu\rangle$.
This implies that $({ W^{\varphi}_{t}})_{t\ge 0}$ is an $L^{2}(\pp_{\mu})$-bounded martingale and thus converges to ${ W^{\varphi}_{\infty}}$ in $L^{2}(\pp_{\mu})$.
Consequently, we have $\langle f_{t},\widetilde{\varphi}\rangle W^{\varphi}_{t}\to \langle f,\widetilde{\varphi}\rangle W^{\varphi}_{\infty}$ in $L^{2}(\pp_{\mu})$.
Hence to prove (i), it suffices to prove that
\begin{equation}\label{lem2.2.n1}
\e^{-\lambda_{1}t}\langle\tilde{f}_{t},X_{t}\rangle\to 0\mbox{ in }L^{2}(\pp_{\mu})\mbox{ as }t\to+\infty.
\end{equation}
Since by \eqref{ieq:H3} $\|\e^{-\lambda_{1}s}T_{s}\tilde{f}_{t}\|_{\infty}\le 2\|\tilde{f}_{t}\|_{\infty}\triangle_{s}$ for all $t,s\ge 0$, we have
\begin{align}
\pp_{\mu}\big[\langle\tilde{f}_{t},X_{t}\rangle^{2}\big]=&\langle T_{t}\tilde{f}_{t},\mu\rangle^{2}+\langle\int_{0}^{t}T_{t-s}\big(\vartheta[T_{s}\tilde{f}_{t}]\big)\mathrm{d}s,\mu\rangle\nonumber\\
\le&4\e^{2\lambda_{1}t}\|\tilde{f}_{t}\|_{\infty}^{2}\triangle^{2}_{t}\langle 1,\mu\rangle^{2}+4\|\tilde{f}_{t}\|^{2}_{\infty}\langle\int_{0}^{t}\e^{2\lambda_{1}s}\triangle^{2}_{s}T_{t-s}(\vartheta[1])\mathrm{d}s,\mu\rangle\nonumber\\
\le&4\e^{2\lambda_{1}t}\|\tilde{f}_{t}\|_{\infty}^{2}\triangle^{2}_{t}\langle 1,\mu\rangle^{2}+4c_{1}\e^{\lambda_{1}t}\|\tilde{f}_{t}\|^{2}_{\infty}\|\vartheta[1]\|_{\infty}
\langle\int_{0}^{t}\e^{\lambda_{1}s}\triangle^{2}_{s}\mathrm{d}s,\mu\rangle.\label{lem2.2.n2}
\end{align}
The final inequality follows from \eqref{eq1.9}.
Let $\epsilon>0$. Since $\triangle_{t}\to 0$, there is some $t_{0}>0$ such that $\triangle_{t}\le \epsilon$ for all $t\ge t_{0}$. Thus
$$\int_{0}^{t}\e^{\lambda_{1}s}\triangle^{2}_{s}\mathrm{d}s=\big(\int_{0}^{t_{0}}+\int_{t_{0}}^{t}\big)\e^{\lambda_{1}s}\triangle^{2}_{s}\mathrm{d}s\le \e^{\lambda_{1}t_{0}}\sup_{s\in [0,t_{0}]}\triangle^{2}_{s}+\lambda^{-1}_{1}\e^{\lambda_{1}t}\epsilon^{2}=:c_{2}+c_{3}\e^{\lambda_{1}t}\epsilon^{2}.$$
Inserting this into the right hand side of \eqref{lem2.2.n2} yields that for $t\ge t_{0}$,
$$\e^{-2\lambda_{1}t}\pp_{\mu}\big[\langle \tilde{f}_{t},X_{t}\rangle^{2}\big]\le 4\|\tilde{f}_{t}\|_{\infty}^{2}\triangle^{2}_{t}\langle 1,\mu\rangle^{2}+4c_{1}\|\tilde{f}_{t}\|^{2}_{\infty}\|\vartheta[1]\|_{\infty}(c_{2}\e^{-\lambda_{1}t}+c_{3}\epsilon^{2}).$$
Note that $\|\tilde{f}_{t}\|_{\infty}$ is bounded from above for large $t$. Letting first $t\to+\infty$ and then $\epsilon\to 0$ in the above inequality, we get
$\e^{-2\lambda_{1}t}\pp_{\mu}\big[\langle \tilde{f}_{t},X_{t}\rangle^{2}\big]\to 0$, which in turn yields \eqref{lem2.2.n1}.

\medskip

\noindent(ii) It follows by (i) that for any $\epsilon>0$, there exists $t_{1}>0$ such that $\|\e^{-\lambda_{1}t}\langle f_{t},X_{t}\rangle -\langle f,\widetilde{\varphi}\rangle W^{\varphi}_{\infty}\|_{L^{2}(\pp_{\mu})}\le \epsilon$. For $t\ge t_{1}$, we may write
\begin{equation}\label{lem2.1.1}
\int_{0}^{t}\langle f_{s},X_{s}\rangle\mathrm{d}s=\int_{0}^{t_{1}}\langle f_{s},X_{s}\rangle\mathrm{d}s+\int_{0}^{t-t_{1}}\langle f_{t-s},X_{t-s}\rangle\mathrm{d}s.
\end{equation}
Since $\int_{0}^{t_{1}}\langle f_{s},X_{s}\rangle \mathrm{d}s$ has a finite second moment, we have $\e^{-\lambda_{1}t}\int_{0}^{t_{1}}\langle f_{s},X_{s}\rangle \mathrm{d}s\to 0$ in $L^{2}(\pp_{\mu})$ as $t\to+\infty$. On the other hand, by Minkowski's inequality,
\begin{align*}
&\|\e^{-\lambda_{1}t}\int_{0}^{t-t_{1}}\langle f_{t-s},X_{t-s}\rangle \mathrm{d}s-\frac{1-\e^{-\lambda_{1}(t-t_{1})}}{\lambda_{1}}\langle f,\widetilde{\varphi}\rangle W^{\varphi}_{\infty}\|_{L^{2}(\pp_{\mu})}\\
=&\|\int_{0}^{t-t_{1}}\e^{-\lambda_{1}s}\big(\e^{-\lambda_{1}(t-s)}\langle f_{t-s},X_{t-s}\rangle -\langle f,\widetilde{\varphi}\rangle W^{\varphi}_{\infty}\big)\mathrm{d}s\|_{L^{2}(\pp_{\mu})}\\
\le&\int_{0}^{t-t_{1}}\e^{-\lambda_{1}s}\|\e^{-\lambda_{1}(t-s)}\langle f_{t-s},X_{t-s}\rangle -\langle f,\widetilde{\varphi}\rangle W^{\varphi}_{\infty}\|_{L^{2}(\pp_{\mu})}\mathrm{d}s\\
\le&\frac{\epsilon}{\lambda_{1}}\big(1-\e^{-\lambda_{1}(t-t_{1})}\big).
\end{align*}
Letting $t\to+\infty$ and then $\epsilon\to 0$, we get
$$\e^{-\lambda_{1}t}\int_{0}^{t-t_{1}}\langle f_{t-s},X_{t-s}\rangle \mathrm{d}s-\frac{1-\e^{-\lambda_{1}(t-t_{1})}}{\lambda_{1}}\langle f,\widetilde{\varphi}\rangle W^{\varphi}_{\infty}\to 0\mbox{ in }L^{2}(\pp_{\mu}),$$
which in turn yields that
$\e^{-\lambda_{1}t}\int_{0}^{t-t_{1}}\langle f_{t-s},X_{t-s}\rangle \mathrm{d}s\to \frac{1}{\lambda_{1}}\langle f,\widetilde{\varphi}\rangle W^{\varphi}_{\infty}$ in $L^{2}(\pp_{\mu})$. Combining this with \eqref{lem2.1.1}, we get (ii).
\qed

\subsection{Asymptotic behaviour of the first two moments}\label{sec:asymptotic moments}

We denote by $\mathcal{D}^{0}:=\{f\in\mathcal{B}_{b}(E):\ \langle f,\widetilde{\varphi}\rangle=0\}$ the class of centered functions, and write $\mathcal{D}^{0}_{\mathrm{S}}:=\mathcal{D}_{\mathrm{S}}\cap \mathcal{D}^{0}$ and
$\mathcal{D}^{0}(\lambda,r):=\mathcal{D}(\lambda,r)\cap\mathcal{D}^{0}$.

\begin{lemma}\label{lem0}
\begin{description}
\item{(i)} For every $f\in\mathcal{D}^{0}_{\mathrm{S}}$,
$$ \e^{-\frac{\lambda_{1}}{2}t}T_{t}f(x)\rightrightarrows 0\mbox{ as }t\to+\infty,$$
and
$$ \e^{-\lambda_{1}t}\Var_{\delta_{x}}[X_{t}(f)]\stackrel{b.p.}{\longrightarrow}\varrho^{2}_{f}\varphi(x) \mbox{ as }t\to+\infty,$$
where $\varrho^{2}_{f}:=\int_{0}^{+\infty}\e^{-\lambda_{1}s}\langle\vartheta[T_{s}f],\widetilde{\varphi}\rangle\mathrm{d}s\in [0,+\infty)$.
\item{(ii)} For every $f\in \mathcal{D}^{0}(\lambda_{1}/2,r)$ with $r\in [0,+\infty)$,
$$ \frac{\e^{-\lambda_{1}t}}{t^{1+2r}}\Var_{\delta_{x}}\big[X_{t}(f)\big]\stackrel{b.p.}{\longrightarrow}\frac{\langle \vartheta[f^{*}],\widetilde{\varphi}\rangle}{1+2r}\varphi(x)\mbox{ as }t\to+\infty.$$
\item{(iii)} For every $f\in\mathcal{D}^{0}(\lambda,r)$ with $(\lambda,r)\in (\lambda_{1}/2,\lambda_{1})\times [0,+\infty)\cup\{(\lambda_{1},0)\}$,
$$\frac{\e^{-2\lambda t}}{t^{2r}}\Var_{\delta_{x}}\big[X_{t}(f)\big]\stackrel{b.p.}{\longrightarrow}\varpi^{2}_{f^{*}}(x)\mbox{ as }t\to+\infty,$$
where for $x\in E$, $\varpi^{2}_{f^{*}}(x)=\int_{0}^{+\infty}\e^{-2\lambda s}T_{s}\big(\vartheta[f^{*}]\big)(x)\mathrm{d}s\in [0,+\infty)$.
\end{description}
\end{lemma}
This lemma can be proved similarly to \cite[Lemmas 4.1 and 4.2]{Yang25}, with minor modifications. Hence we omit the details here.

\begin{lemma}\label{lem1}
\begin{description}
\item{(i)} For every $f,g\in\mathcal{D}^{0}_{\mathrm{S}}$,
$$ \e^{-\frac{\lambda_{1}}{2}t}M_{t}[g](x)\rightrightarrows 0\mbox{ as }t\to+\infty,$$
$$ \e^{-\lambda_{1}t}\mathrm{Var}_{\delta_{x}}[\langle g,Y_{t}\rangle]\stackrel{b.p.}{\longrightarrow}{ \chi^{2}_{g}}\varphi(x)\mbox{ as }t\to+\infty,$$
where ${ \chi^{2}_{g}}:=\int_{0}^{+\infty}\e^{-\lambda_{1} s}\langle\vartheta[M_{s}[g]],\widetilde{\varphi}\rangle \mathrm{d}s\in [0,+\infty)$, and
$$ \e^{-\lambda_{1}t}\mathrm{Cov}_{\delta_{x}}\big(X_{t}(f),Y_{t}(g)\big)\stackrel{b.p.}{\longrightarrow}\bar{\varsigma}_{f,g}\varphi(x)\mbox{ as }t\to+\infty,$$
where $\bar{\varsigma}_{f,g}:=\int_{0}^{+\infty}\e^{-\lambda_{1}s}\langle\vartheta[T_{s}f,M_{s}[g]] ,\widetilde{\varphi}\rangle\mathrm{d}s\in \R$.
\item{(ii)} For every $f,g\in\mathcal{D}^{0}(\lambda_{1}/2,r)$ with $r\in [0,+\infty)$,
$$\frac{\e^{-\frac{\lambda_{1}}{2}t}}{t^{r}}M_{t}[g](x)\rightrightarrows \frac{2}{\lambda_{1}}g^{*}(x)\mbox{ as }t\to+\infty,$$
$$ \frac{\e^{-\lambda_{1}t}}{t^{1+2r}}\mathrm{Var}_{\delta_{x}}[\langle g,Y_{t}\rangle]\stackrel{b.p.}{\longrightarrow}\frac{4}{\lambda^{2}_{1}(1+2r)}\langle\vartheta[g^{*}],\widetilde{\varphi}\rangle \varphi(x)\mbox{ as }t\to+\infty,$$
and
$$ \frac{\e^{-\lambda_{1}t}}{t^{1+2r}}\mathrm{Cov}_{\delta_{x}}\big(X_{t}(f),Y_{t}(g)\big)\stackrel{b.p.}{\longrightarrow}\frac{2}{\lambda_{1}(1+2r)}\langle\vartheta[f^{*},g^{*}],
\widetilde{\varphi}\rangle \varphi(x)\mbox{ as }t\to+\infty.$$
\item{(iii)} For every $g\in \mathcal{D}^{0}(\lambda,r)$ with $(\lambda,r)\in (\lambda_{1}/2,\lambda_{1})\times [0,+\infty)\cup\{(\lambda_{1},0)\}$,
$$\frac{\e^{-\lambda t}}{t^{r}}M_{t}[g](x)\rightrightarrows \frac{1}{\lambda}g^{*}(x)\mbox{ as }t\to+\infty,$$
and
$$\frac{\e^{-2\lambda t}}{t^{2 r}}\Var_{\delta_{x}}\big[Y_{t}(g)\big]\stackrel{b.p.}{\longrightarrow}\frac{1}{\lambda^{2}}\varpi^{2}_{g^{*}}(x)\mbox{ as }t\to+\infty.$$
\end{description}
\end{lemma}

\proof
(i) Let $\delta\in (0,({\rho}(g)-\lambda_{1}/2)\wedge ({\rho}(f)-\lambda_{1}/2))$. By definition, there exists some $t_{0}>0$ such that for all $t\ge t_{0}$,
\begin{equation}\label{lem1.7}
\frac{1}{t}\log\sup_{x\in E}\big| \e^{-\lambda_{1}t}T_{t}g(x)\big|\le -{\rho}(g)+\delta\mbox{ and }\frac{1}{t}\log\sup_{x\in E}\big| \e^{-\lambda_{1}t}T_{t}f(x)\big|\le -{\rho}(f)+\delta.
\end{equation}
Recall that $M_{t}[g](x)=\int_{0}^{t}T_{s}g(x)\mathrm{d}s$. For $0\le t\le t_{0}$, we have by \eqref{ieq:H3}
\begin{align}\label{lem1.1}
\sup_{x\in E}\big| M_{t}[g](x)\big|\le &\int_{0}^{t_{0}}\e^{\lambda_{1}s}\sup_{x\in E}\big| \e^{-\lambda_{1}s}T_{s}g(x)\big|\mathrm{d}s\nonumber\\
\le &\frac{2}{\lambda_{1}}\big(\e^{\lambda_{1}t_{0}}-1\big)\|g\|_{\infty}\sup_{s>0}\triangle_{s}=:c_{1}.
\end{align}
Thus by \eqref{lem1.7}, for $t\ge t_{0}$,
\begin{align}\label{lem1.2}
\sup_{x\in E}\big| M_{t}[g](x)\big|=&\sup_{x\in E}\big|M_{t_{0}}[g](x)+\int_{t_{0}}^{t}T_{s}g(x)\mathrm{d}s\big|\le c_{1}+\int_{t_{0}}^{t}\sup_{x\in E}\big| T_{s}g(x)\big|\mathrm{d}s\nonumber\\
\le &c_{1}+\int_{t_{0}}^{t}\e^{(\lambda_{1}-{\rho}(g)+\delta)s}\mathrm{d}s
\le c_{2}+c_{3}\e^{(\lambda_{1}-{\rho}(g)+\delta)t}.
\end{align}
Consequently, we get for all $t\ge 0$,
\begin{equation}\label{lem1.9}
\sup_{x\in E}\big| \e^{-\frac{\lambda_{1}}{2}t}M_{t}[g](x)\big|\le (c_{1}+c_{2})\e^{-\frac{\lambda_{1}}{2}t}+c_{3}\e^{-({\rho}(g)-\delta-\lambda_{1}/2)t}\to 0\mbox{ as }t\to+\infty,
\end{equation}
which in turn yields the first conclusion.

Let $\epsilon\in (\lambda_{1}/2,({\rho}(g)-\delta)\wedge ({\rho}(f)-\delta)\wedge \lambda_{1})$.
By \eqref{lem1.1} and \eqref{lem1.2}, $M_{1}:=\sup_{s\ge 0,\, x\in E}\big|\e^{-(\lambda_{1}-\epsilon)s}M_{s}[g](x)\big|<+\infty$.
Note that
\begin{align}\label{lem1.3}
\e^{-\lambda_{1}t}\Var_{\delta_{x}}\big[Y_{t}(g)\big]
=&\int_{0}^{+\infty}\e^{-(2\epsilon-\lambda_{1})s} 1_{\{s\le t\}}
\e^{-\lambda_{1}(t-s)}T_{t-s}\big(\vartheta\big[\e^{-(\lambda_{1}-\epsilon)s}M_{s}[g]\big]\big)(x)\mathrm{d}s\nonumber\\
=:&\int_{0}^{+\infty}\e^{-(2\epsilon-\lambda_{1})s} h_{t}(s,x)\mathrm{d}s.
\end{align}
It follows by \eqref{eq1.9} and {(H2)} that for every $s\ge 0$ and $x\in E$,
$$\big| h_{t}(s,x)\big|\le  1_{\{s\le t\}}\e^{-\lambda_{1}(t-s)}T_{t-s}\big(M^{2}_{1}\vartheta[1]\big)(x)
\le c_{4}M^{2}_{1}\|\vartheta[1]\|_{\infty},$$
and
$$\lim_{t\to+\infty} h_{t}(s,x)=\e^{-2(\lambda_{1}-\epsilon)s}\langle\vartheta[M_{s}[g]] ,\widetilde{\varphi}\rangle \varphi(x).$$
Letting $t\to+\infty$ in \eqref{lem1.3}, we get by the dominated convergence theorem that $ \e^{-\lambda_{1}t}\Var_{\delta_{x}}\big[Y_{t}(g)\big]\stackrel{b.p.}{\longrightarrow}\chi^{2}_{g}\varphi(x)$ as $t\to+\infty$.

We have
\begin{align*}
& \e^{-\lambda_{1}t}\mathrm{Cov}_{\delta_{x}}\big(X_{t}(f),Y_{t}(g)\big)\\
=& \e^{-\lambda_{1}t}\int_{0}^{t}T_{t-s}\big(\vartheta[T_{s}f,M_{s}[g]] \big)(x)\mathrm{d}s\\
=&\int_{0}^{+\infty}\e^{-(2\epsilon-\lambda_{1})s} 1_{\{s\le t\}}\e^{-\lambda_{1}(t-s)}T_{t-s}\big(\vartheta\big[\e^{-(\lambda_{1}-\epsilon)s}T_{s}f,
\e^{-(\lambda_{1}-\epsilon)s}M_{s}[g]\big]\big)(x)\mathrm{d}s\\
=:&\int_{0}^{+\infty}\e^{-(2\epsilon-\lambda_{1})s} g_{t}(s,x)\mathrm{d}s.
\end{align*}
It follows by \eqref{lem1.7} and \eqref{lem1.2} that $\sup_{s\ge 0,\, x\in E}\big(\big|\e^{-(\lambda_{1}-\epsilon)s}T_{s}f(x)\big|\vee\big|\e^{-(\lambda_{1}-\epsilon)s}M_{s}[g](x)\big|\big)<+\infty$. We can show similarly as above that
$$\sup_{x\in E,\,s\ge 0}\big| g_{t}(s,x)\big|<+\infty,$$ and
$$\lim_{t\to+\infty} g_{t}(s,x)=\e^{-2(\lambda_{1}-\epsilon)s}\langle\vartheta\big[T_{s}f,M_{s}[g]\big] ,\widetilde{\varphi}\rangle \varphi(x).$$
Thus, the third conclusion follows by the dominated convergence theorem.

\medskip

\noindent (ii) For $\theta>0$ and $r\ge 0$, let $F_{\theta,r}(t):=\int_{0}^{t}\e^{\theta s}s^{r}\mathrm{d}s$. Clearly, we have $t^{-r}\e^{-\theta t}F_{\theta,r}(t)\sim 1/\theta$ as $t\to+\infty$.
Let $\epsilon>0$.
Since $t^{-r}\e^{-\frac{\lambda_{1}}{2}t}T_{t}g\rightrightarrows g^{*}$ as $t\to+\infty$, there is some $t_{0}>0$ such that
$$\sup_{x\in E}\big|t^{-r}\e^{-\frac{\lambda_{1}}{2}t}T_{t}g(x)-g^{*}(x)\big|\le \epsilon\quad\forall t\ge t_{0}.$$
For $t\ge t_{0}$, we have
\begin{align*}
\frac{\e^{-\frac{\lambda_{1}}{2}t}}{t^{r}}\big(M_{t}[g](x)-F_{\frac{\lambda_{1}}{2},r}(t)g^{*}(x)\big)=&\frac{\e^{-\frac{\lambda_{1}}{2}t}}{t^{r}}
\Big[\int_{0}^{t_{0}}\big(T_{s}g(x)-s^{r}\e^{\frac{\lambda_{1}}{2}s}g^{*}(x)\big)\mathrm{d}s\\
&+\int_{t_{0}}^{t}s^{r}\e^{\frac{\lambda_{1}}{2}r}
\big(s^{-r}\e^{-\frac{\lambda_{1}}{2}s}T_{s}g(x)-g^{*}(x)\big)\mathrm{d}s\Big].
\end{align*}
Hence
\begin{align*}
\sup_{x\in E}\Big|\frac{\e^{-\frac{\lambda_{1}}{2}t}}{t^{r}}\big(M_{t}[g](x)-F_{\frac{\lambda_{1}}{2},r}(t)g^{*}(x)\big)\Big|
\le&\frac{\e^{-\frac{\lambda_{1}}{2}t}}{t^{r}}
\Big(\int_{0}^{t_{0}}\sup_{x\in E}\big|T_{s}g(x)-s^{r}\e^{\frac{\lambda_{1}}{2}s}g^{*}(x)\big|\mathrm{d}s+\int_{0}^{t}s^{r}\e^{\frac{\lambda_{1}}{2}r}
\epsilon\mathrm{d}s\Big)\\
=:&\frac{\e^{-\frac{\lambda_{1}}{2}t}}{t^{r}}\big(c_{4}+\epsilon F_{\frac{\lambda_{1}}{2},r}(t)\big).
\end{align*}
Recall that $t^{-r}\e^{-\frac{\lambda_{1}}{2}t}F_{\frac{\lambda_{1}}{2},r}(t)\sim 2/\lambda_{1}$ as $t\to+\infty$. Letting $t\to+\infty$ and then $\epsilon\to 0$ in the above inequality, we obtain that
\begin{equation}\label{lem3.2.n1}
\frac{\e^{-\frac{\lambda_{1}}{2}t}}{t^{r}}\big(M_{t}[g](x)-F_{\frac{\lambda_{1}}{2},r}(t)g^{*}(x)\big)\rightrightarrows 0\mbox{ as }t\to+\infty,
\end{equation}
which in turn implies that $t^{-r}\e^{-\frac{\lambda_{1}}{2}t}M_{t}[g]\rightrightarrows 2g^{*}/\lambda_{1}$ as $t\to+\infty$.

We have
\begin{align*}
\frac{\e^{-\lambda_{1}t}}{t^{1+2r}}\Var_{\delta_{x}}\big[Y_{t}(g)\big]=& \frac{\e^{-\lambda_{1}t}}{t^{1+2r}}\int_{0}^{t}T_{t-s}\big(\vartheta[M_{s}[g]]\big)(x)\mathrm{d}s\\
=&\frac{1}{t}\int_{0}^{t}\big(\frac{s}{t}\big)^{2r} \e^{-\lambda_{1}(t-s)}T_{t-s}\big(\vartheta\big[s^{-r}\e^{-\frac{\lambda_{1}}{2}s}M_{s}[g]\big]\big)(x)\mathrm{d}s\\
=&\int_{0}^{1}u^{2r} \e^{-\lambda_{1}t(1-u)}T_{t(1-u)}\big(\vartheta\big[(tu)^{-r}\e^{-\frac{\lambda_{1}}{2}tu}M_{tu}[g]\big]\big)(x)\mathrm{d}u\\
=:&\int_{0}^{1}u^{2 r}l_{t}(u,x)\mathrm{d}u.
\end{align*}
Using $t^{-r}\e^{-\frac{\lambda_{1}}{2}t}M_{t}[g]\rightrightarrows 2g^{*}/\lambda_{1}$, we can apply similar argument as in the proof of \cite[Lemma 3.4(i)]{Yang25} to show that for every $u\in (0,1)$,
$$\vartheta\big[(tu)^{-r}\e^{-\frac{\lambda_{1}}{2}tu}M_{tu}[g]\big]\stackrel{b.p.}{\longrightarrow}\frac{4}{\lambda^{2}_{1}}\vartheta[g^{*}]\mbox{ as }t\to+\infty.$$
Then by {(H2)} and the dominated convergence theorem, we have for every $u\in (0,1)$ and $x\in E$,
$$\lim_{t\to+\infty} l_{t}(u,x)=\frac{4}{\lambda^{2}_{1}}\langle
\vartheta[g^{*}],\widetilde{\varphi}\rangle \varphi(x).$$
If there exists some nonnegative integrable function $h_{1}(u)$ on $(0,1)$, such that, for $t$ sufficiently large,
\begin{equation}\label{lem1.4}
u^{2r} l_{t}(u,x)\le h_{1}(u)\quad\forall u\in (0,1),\ x\in E,
\end{equation}
then by the dominated convergence theorem, we have
\begin{equation}\label{lem1.6}
\frac{\e^{-\lambda_{1}t}}{t^{1+2r}}\Var_{\delta_{x}}[Y_{t}(g)]\stackrel{b.p.}{\longrightarrow}\int_{0}^{1}u^{2r}\frac{4}{\lambda^{2}_{1}}\langle\vartheta[g^{*}],\widetilde{\varphi}\rangle \varphi(x)\mathrm{d}u=\frac{4}{\lambda^{2}_{1}(1+2r)}\langle\vartheta[g^{*}],
\widetilde{\varphi}\rangle \varphi(x)\mbox{ as }t\to+\infty.
\end{equation}
In fact, by the uniform convergence of $t^{-r}\e^{-\frac{\lambda_{1}}{2}t}M_{t}[g]$ as $t\to+\infty$, there exists some $t_{1}>0$ such that $M_{1}:=\sup_{t\ge t_{1},x\in E}\big|t^{-r}\e^{-\frac{\lambda_{1}}{2}t}M_{t}[g](x)\big|<+\infty$. By this and \eqref{eq1.9} we have for $t\ge t_{1}$ and $u\ge t_{1}/t$,
$$l_{t}(u,x)\le M^{2}_{1}\e^{-\lambda_{1}t(1-u)}T_{t(1-u)}\big(\vartheta[1]\big)(x)\le c_{5}M^{2}_{1}\|\vartheta[1]\|_{\infty}=:c_{6}.$$
On the other hand, by \eqref{eq1.9}, for $t\ge t_{1}$ and $u\in (0,t_{1}/t]$,
$$\|M_{tu}[g]\|_{\infty}=\|\int_{0}^{tu}T_{s}g\mathrm{d}s\|_{\infty}\le c_{5}\|g\|_{\infty}\int_{0}^{t_{1}}\e^{\lambda_{1}s}\mathrm{d}s=:c_{7}.$$
Thus in this case,
\begin{align*}
l_{t}(u,x)=&(tu)^{-2r}\e^{-\lambda_{1}tu}\cdot\e^{-\lambda_{1}t(1-u)}T_{t(1-u)}\big(\vartheta\big[M_{tu}[g]\big]\big)(x)\\
\le&u^{-2r}t_{1}^{-2r}c^{2}_{7}\e^{-\lambda_{1}t(1-u)}T_{t(1-u)}\big(\vartheta[1]\big)(x)\le c_{8}u^{-2r}.
\end{align*}
Therefore, we have proved that $l_{t}(u,x)\le c_{6}+c_{8}u^{-2r}$ for all $t\ge 0$, $u>0$ and $x\in E$. Hence, \eqref{lem1.4} holds with $h_{1}(u)=c_{6}u^{2r}+c_{8}$.

Note that
\begin{align*}
\frac{\e^{-\lambda_{1}t}}{t^{1+2r}} \mathrm{Cov}_{\delta_{x}}\big(X_{t}(f),Y_{t}(g)\big)
=&\frac{\e^{-\lambda_{1}t}}{t^{1+2r}} \int_{0}^{t}T_{t-s}\big(\vartheta[T_{s}f,M_{s}[g]] \big)(x)\mathrm{d}s\\
=&\int_{0}^{1}u^{2r} \e^{-\lambda_{1}t(1-u)}T_{t(1-u)}\big(\vartheta\big[(tu)^{-r}\e^{-\frac{\lambda_{1}}{2}tu}T_{tu}f,(tu)^{-r}\e^{-\frac{\lambda_{1}}{2}tu}M_{tu}[g]\big]\big)(x)\mathrm{d}u\\
=:&\int_{0}^{1}u^{2r}k_{t}(u,x)\mathrm{d}u.
\end{align*}
Based on the facts that $t^{-r}\e^{-\frac{\lambda_{1}}{2}t}T_{t}f\rightrightarrows f^{*}$ and $t^{-r}\e^{-\frac{\lambda_{1}}{2}t}M_{t}[g]\rightrightarrows 2g^{*}/\lambda_{1}$, we can show similarly as above that
$$\frac{\e^{-\lambda_{1}t}}{t^{1+2r}}\Cov_{\delta_{x}}(X_{t}(f),Y_{t}(g))\stackrel{b.p.}{\longrightarrow}\varphi(x)\int_{0}^{1}u^{2r}\langle\vartheta[f^{*},\frac{2}{\lambda_{1}}g^{*}],\widetilde{\varphi}\rangle\mathrm{d}u=\frac{2}{\lambda_{1}(1+2r)}\langle
\vartheta[f^{*},g^{*}],\widetilde{\varphi}\rangle \varphi(x)\mbox{ as }t\to+\infty,$$ which yields the third conclusion.

\medskip

\noindent (iii) Recall that $F_{\lambda,r}(t)=\int_{0}^{t}\e^{\lambda s}s^{r}\mathrm{d}s$.
Since $t^{-r}\e^{-\lambda t}T_{t}g\rightrightarrows g^{*}$ and $t^{-r}\e^{-\lambda t}F_{\lambda,r}(t)\sim 1/\lambda$ as $t\to+\infty$, we can prove similarly to \eqref{lem3.2.n1} that
$$\frac{\e^{-\lambda t}}{t^{r}}\big(M_{t}[g](x)-F_{\lambda,r}(t)g^{*}(x)\big)\rightrightarrows 0\mbox{ as }t\to+\infty.$$
and consequently, $t^{-r}\e^{-\lambda t}M_{t}[g]\rightrightarrows g^{*}/\lambda$ as $t\to+\infty$.

We have
\begin{align*}
\frac{\e^{-2\lambda t}}{t^{2r}}\Var_{\delta_{x}}\big[Y_{t}(g)\big]=&\frac{\e^{-2\lambda t}}{t^{2r}}\int_{0}^{t}T_{s}\big(\vartheta[M_{t-s}[g]]\big)(x)\mathrm{d}s\\
=&\int_{0}^{+\infty}\e^{-(2\lambda-\lambda_{1})s}\cdot 1_{\{s\le t\}}\big(1-\frac{s}{t}\big)^{2r}\e^{-\lambda_{1}s}T_{s}\big(\vartheta\big[(t-s)^{-r}\e^{-\lambda(t-s)}M_{t-s}[g]\big]\big)(x)\mathrm{d}s\\
=:&\int_{0}^{+\infty}\e^{-(2\lambda-\lambda_{1})s}f_{t}(s,x)\mathrm{d}s.
\end{align*}

{Using the fact that $t^{-r}\e^{-\lambda t}M_{t}[g]\rightrightarrows g^{*}/\lambda$ as $t\to+\infty$, one can easily show that for every $s\ge 0$,
$$\vartheta\big[(t-s)^{-r}\e^{-\lambda(t-s)}M_{t-s}[g]\big]\stackrel{b.p.}{\longrightarrow} \frac{1}{\lambda^{2}}\vartheta[g^{*}]\mbox{ as }t\to+\infty,$$}
and thus by the bounded convergence theorem, for every $s\ge 0$ and $x\in E$,
$$f_{t}(s,x)\stackrel{b.p.}{\longrightarrow}\frac{\e^{-\lambda_{1} s}}{\lambda^{2}}T_{s}\big(\vartheta[g^{*}]\big)(x)\mbox{ as }t\to+\infty.$$
Using the uniform convergence of $t^{-r}\e^{-\lambda t}M_{t}[g]$ and \eqref{eq1.9}, we can show that there is a constant $c_{9}>0$ such that for $t$ sufficiently large,
\begin{equation}\label{lem1.8}
|f_{t}(s,x)|\le c_{9}\quad\forall s>0,\ x\in E.
\end{equation}
It thus follows by the dominated convergence theorem that $t^{-2r}\e^{-2\lambda t}\Var_{\delta_{x}}[Y_{t}[g]]\stackrel{b.p.}{\longrightarrow}\varpi^{2}_{g^{*}}(x)/\lambda^{2}$ as $t\to+\infty$. \qed

\section{Proofs of the main results}\label{sec3}
Let us first recall some basic notions and facts about convergence in distribution for random variables and processes, following \cite[Chapter VI]{JS}.

{ Let $S$ be a Polish space with Borel $\sigma$-field $\mathcal{S}$. We denote by $\mathcal{P}(S)$ the space of all probability measures on $(S,\mathcal{S})$ endowed with the weak topology. Then $\mathcal{P}(S)$ itself is a Polish space for this topology. A subset $A$ of $\mathcal{P}(S)$ is called tight if for every $\epsilon>0$, there exists a compact set $K$ in $S$ such that $\sup_{\nu\in A}\nu(S\setminus K)\le \epsilon$.
Prokhorov's theorem tells us that, a subset $A$ of $\mathcal{P}(S)$ is relatively compact (i.e. it has a compact closure) for the weak topology if and only if it is tight.
Now we consider random variables.
If $\{\xi^{n}:n\ge 1\}$ is a sequence of $S$-valued random variables, we say $\{\xi^{n}:\ n\ge 1\}$ is tight in $S$ if the set of its laws is tight in $\mathcal{P}(S)$.
We say $\xi^{n}$ converges in distribution to $\xi$, and write ``$\xi^{n}\stackrel{d}{\to}\xi$ in $S$", if the law of $\xi^{n}$ converges weakly to the law of $\xi$.}

Let $\mathcal{D}[0,+\infty)$ denote the space of c\`{a}dl\`{a}g functions from $[0,+\infty)$ to $\R$ equipped with the Skorokhod topology for which it is a Polish space, and let $\mathcal{P}(\mathcal{D}[0,+\infty))$ be the space of all probability measures on $\mathcal{D}[0,+\infty)$ endowed with the weak topology. An $\R$-valued c\`{a}dl\`{a}g process $Y$ may be considered as a random variable taking values in $\mathcal{D}[0,+\infty)$, and consequently its law, denoted by $\mathcal{L}(Y)$, is an element of $\mathcal{P}(\mathcal{D}[0,+\infty))$.
For c\`{a}dl\`{a}g processes $Y^{(n)}$ and $Y$,
the most common method of proving $Y^{(n)}\stackrel{d}{\to}Y$ in $\mathcal{D}[0,+\infty)$ goes through the following procedure:
\begin{description}
\item[(i)] prove $\{Y^{(n)}\}$ is tight in $\mathcal{D}[0,+\infty)$,
\item[(ii)] prove that $\mathcal{L}(Y)$ is the only possible limit point for the sequence $\{\mathcal{L}(Y^{(n)})\}$.
\end{description}
Since the law of a c\`{a}dl\`{a}g process is uniquely determined by its finite-dimensional distributions on a dense set of $[0,+\infty)$, step (ii) can be replaced by the following:
\begin{description}
\item[(ii')] prove there is a dense set $D\subset [0,+\infty)$ such that $(Y^{(n)}_{t_{1}},\cdots,Y^{(n)}_{t_{k}})\stackrel{d}{\to}(Y_{t_{1}},\cdots,Y_{t_{k}})$ in $\R^{k}$ for all $t_{1},\cdots,t_{k}\in D$.
\end{description}

Recall that $\mathcal{C}[0,+\infty)$ denotes the space of continuous functions equipped with the local uniform topology and $\mathcal{P}(\mathcal{C}[0,+\infty))$ denotes the space of probability measures on $\mathcal{C}[0,+\infty)$ with the corresponding weak topology.
Since $\mathcal{C}[0,+\infty)$ is a closed subset of $\mathcal{D}[0,+\infty)$ under the Skorokhod topology and the relative topology coincides with the local uniform topology, $\mathcal{P}(\mathcal{C}[0,+\infty))$ can be canonically embedded as a closed subset of $\mathcal{P}(\mathcal{D}[0,+\infty))$, and the two weak topologies coincide on it.
 It follows that a sequence of probability measures $\{\nu_{n}:\ n\ge 1\}$ supported on $\mathcal{C}[0,+\infty)$ converges in $\mathcal{P}(\mathcal{C}[0,+\infty))$ if and only if it converges in $\mathcal{P}(\mathcal{D}[0,+\infty))$. In particular, if $Y^{(n)}$ are continuous processes and $Y^{(n)}\stackrel{d}{\to}Y$ in $\mathcal{D}[0,+\infty)$ for some c\`{a}dl\`{a}g process $Y$, then $\mathcal{L}(Y)$ is supported on $\mathcal{C}[0,+\infty)$, so that $Y$ is almost surely continuous and $Y^{(n)}\stackrel{d}{\to}Y$ in $\mathcal{C}[0,+\infty)$ as well.
 Therefore, to prove a sequence of continuous processes converges in distribution in $\mathcal{C}[0,+\infty)$, one may use the same procedure (i)-(ii) or (i)-(ii') as above, and the limit process is then automatically continuous.

\subsection{Convergence of finite-dimensional distributions}\label{sec3.1}

{
We say a sequence $\{\eta_{n}:\ n\ge 1\}$ of random variables on a probability space $(\Omega,\mathcal{G},\mathrm{P})$ converges stably in distribution to a random variable $\eta$ (defined, if necessary, on an extension of $(\Omega,\mathcal{G},\mathrm{P})$), if for every bounded $\mathcal{G}$-adapted random variable $\zeta$, $(\eta_{n},\zeta)\stackrel{d}{\to}(\eta,\zeta)$ as $n\to+\infty$; see, e.g., \cite[Chapter VIII.5]{JS} or \cite{Tou92} for this definition and its equivalent formulations.
\begin{lemma}\label{lem3.n1}
Let
$$\mathrm{M}_{t}(\varphi):= Y_{t}(\varphi)-\frac{1}{\lambda_{1}} X_{t}(\varphi)\quad\forall t\ge 0.$$
Then $\{\mathrm{M}_{t}(\varphi):\ t\ge 0\}$ is a c\`{a}dl\`{a}g square-integrable $\pp_{\delta_{x}}$-martingale. Moreover, as $t\to+\infty$, $\e^{-\frac{\lambda_{1}}{2}t}\mathrm{M}_{t}(\varphi)$ converges stably in distribution
to $\sqrt{\lambda^{-3}_{1}\langle\vartheta[\varphi],\widetilde{\varphi}\rangle W^{\varphi}_{\infty}}\cdot N$, where $N$ is a standard normal random variable independent of $W^{\varphi}_{\infty}$.
\end{lemma}

\begin{proof}
By the Markov property and the fact that $T_{s}\varphi=\e^{\lambda_{1}s}\varphi$ for all $s\ge 0$, we have $M_{s}[\varphi]=\lambda^{-1}_{1}(\e^{\lambda_{1}s}-1)\varphi$ and
$$\pp_{\delta_{x}}\big[\mathrm{M}_{t+s}(\varphi)|\mathcal{F}_{t}\big]=Y_{t}(\varphi)+X_{t}(M_{s}[\varphi])-\frac{1}{\lambda_{1}}X_{t}(T_{s}\varphi)= Y_{t}(\varphi)-\frac{1}{\lambda_{1}} X_{t}(\varphi).$$
Clearly, under (M2), $\pp_{\delta_{x}}[\mathrm{M}_{t}(\varphi)^{2}]<+\infty$ for every $t\ge 0$.
Thus, $\{\mathrm{M}_{t}(\varphi):\ t\ge 0\}$ is a square-integrable $\pp_{\delta_{x}}$-martingale.
Next we shall show $\{\mathrm{M}_{t}(\varphi):\ t\ge 0\}$ has c\`{a}dl\`{a}g paths. Replacing $f$ with $\varphi$ in \eqref{eq:many-to-one} yields that
\begin{equation}\label{lemn2.1}
\varphi(x)=\e^{-\lambda_{1}t}P_{t}\varphi(x)-\Pi_{x}\big[\int_{0}^{t}a(\xi_{s})\e^{-\lambda_{1}s}\varphi(\xi_{s})\mathrm{d}s\big]+\Pi_{x}\big[\int_{0}^{t}\e^{-\lambda_{1}s}m[\varphi](\xi_{s})
\mathrm{d}s\big]\quad\forall x\in E,
\end{equation}
which in turn yields that
\begin{align*}
\e^{-\lambda_{1}t}P_{t}\varphi(x)=&\varphi(x)+\Pi_{x}\big[\int_{0}^{t}a(\xi_{s})\e^{-\lambda_{1}s}\varphi(\xi_{s})\mathrm{d}s\big]-\Pi_{x}\big[\int_{0}^{t}\e^{-\lambda_{1}s}m[\varphi](\xi_{s})
\mathrm{d}s\big]\\
\le&\varphi(x)+\|a\|_{\infty}\int_{0}^{t}\e^{-\lambda_{1}s}P_{s}\varphi(x)\mathrm{d}s.
\end{align*}
Thus, we get by Gronwall's inequality that $\e^{-\lambda_{1}t}P_{t}\varphi(x)\le \e^{\|a\|_{\infty}t}\varphi(x)$ for all $t\ge 0$ and $x\in E$. Moreover, letting $t\downarrow 0$ in \eqref{lemn2.1}, we get $\e^{-\lambda_{1}t}P_{t}\varphi(x)\to \varphi(x)$ for every $x\in E$, and thus $\lim_{t\downarrow 0}\e^{-(\lambda_{1}+\|a\|_{\infty})t}P_{t}\varphi(x)=\varphi(x)$. Hence, $\varphi$ is a $(\lambda_{1}+\|a\|_{\infty})$-excessive function relative to the process $\xi$, and thus belongs to $C^{\xi}_{b}(E)$. It follows by \cite[Lemma 4.1]{PY20} that, $t\mapsto X_{t}(\varphi)$ is a.s. right continuous on $[0,+\infty)$. This combined with \cite[Proposition 4.3]{PY20} yields that $\pp_{\delta_{x}}$-a.s.
\begin{equation}\label{lemn2.2}
 X_{t}(\varphi)=\e^{\lambda_{1}t}\varphi(x)+\e^{\lambda_{1}t}\int_{0}^{t}\int_{E}\e^{-\lambda_{1}s}\varphi(y)M(\mathrm{d}s,\mathrm{d}y)\quad\forall t\ge 0,
\end{equation}
where $M(\mathrm{d}s,\mathrm{d}y)$ is the martingale measure defined from the martingale problem of superprocesses, and $\int_{0}^{t}\int_{E}\e^{-\lambda_{1}s}\varphi(y)M(\mathrm{d}s,\mathrm{d}y)$ stands for the stochastic integral of $\e^{-\lambda_{1}s}\varphi(y)$ with respect to $M(\mathrm{d}s,\mathrm{d}y)$.
It is known that $\{\int_{0}^{t}\int_{E}\e^{-\lambda_{1}s}\varphi(y)M(\mathrm{d}s,\mathrm{d}y):\ t\ge 0\}$ is a c\`{a}dl\`{a}g local martingale. Hence by \eqref{lemn2.2} $\{X_{t}(\varphi):\ t\ge 0\}$ has c\`{a}dl\`{a}g paths. The same holds for $\{\mathrm{M}_{t}(\varphi):\ t\ge 0\}$, since $t\mapsto Y_{t}(\varphi)$ is continuous.

By \eqref{lemn2.2} and the stochastic Fubini's theorem (see, e.g., \cite[Theorem 7.24]{Li11}), it holds that $\pp_{\delta_{x}}$-a.s. for all $t\ge 0$,
\begin{align*}
 Y_{t}(\varphi)=&\int_{0}^{t}\langle\varphi,X_{r}\rangle\mathrm{d}r=\varphi(x)\int_{0}^{t}\e^{\lambda_{1}r}\mathrm{d}r+\int_{0}^{t}\int_{E}\big(\int_{s}^{t}\e^{\lambda_{1}(r-s)}
\varphi(y)\mathrm{d}r\big)M(\mathrm{d}s,\mathrm{d}y)\\
=&\frac{\e^{\lambda_{1}t}-1}{\lambda_{1}}\varphi(x)+\int_{0}^{t}\int_{E}\frac{\e^{\lambda_{1}(t-s)}-1}{\lambda_{1}}\varphi(y)M(\mathrm{d}s,\mathrm{d}y).
\end{align*}
Consequently, we have $\pp_{\delta_{x}}$-a.s.,
\begin{equation}\label{lemn2.3}
\mathrm{M}_{t}(\varphi)=-\frac{1}{\lambda_{1}}\varphi(x)-\frac{1}{\lambda_{1}}\int_{0}^{t}\int_{E}\varphi(y)M(\mathrm{d}s,\mathrm{d}y)=:-\frac{1}{\lambda_{1}}\varphi(x)+\tilde{\mathrm{M}}_{t}(\varphi)\quad \forall t\ge 0,
\end{equation}
where $\{\tilde{\mathrm{M}}_{t}(\varphi):\ t\ge 0\}$ is a c\`{a}dl\`{a}g square-integrable martingale with $\tilde{\mathrm{M}}_{0}(\varphi)=0$. Moreover, it is shown in \cite{PY20} that the predictable quadratic variation of $\tilde{\mathrm{M}}_{t}(\varphi)$ is given by $\langle \tilde{\mathrm{M}}(\varphi)\rangle_{t}=\lambda^{-2}_{1}\int_{0}^{t}\langle\vartheta[\varphi],X_{s}\rangle\mathrm{d}s=\lambda^{-2}_{1}\langle \vartheta[\varphi],Y_{t}\rangle$.
Hence, by Lemma \ref{lem2.2}(ii),
$$\lim_{t\to+\infty}\e^{-\lambda_{1}t}\langle\tilde{\mathrm{M}}(\varphi)\rangle_{t}=\lambda^{-3}_{1}\langle\vartheta[\varphi],\widetilde{\varphi}\rangle W^{\varphi}_{\infty}\mbox{ in } L^{2}(\pp_{\delta_{x}}).$$
Let $\nu^{\tilde{\mathrm{M}}}(\mathrm{d}s,\mathrm{d}x)$ be the predictable compensator of the jump measure of $\tilde{\mathrm{M}}_{t}(\varphi)$. It follows by \cite[Section 4]{PY20} that for any $\epsilon>0$,
$$\int_{0}^{t}\int_{|x|>\epsilon\e^{\frac{\lambda_{1}}{2}t}}|x|^{2}\nu^{\tilde{\mathrm{M}}}(\mathrm{d}s,\mathrm{d}x)
=\int_{0}^{t}\mathrm{d}s\int_{E}X_{s}(\mathrm{d}y)\int_{|\nu(\varphi)|>\epsilon\e^{\frac{\lambda_{1}}{2}t}}\nu(\varphi)^{2}H(y,\mathrm{d}\nu)$$
Under (M2), we have $$\int_{|\nu(\varphi)|>\epsilon\e^{\frac{\lambda_{1}}{2}t}}\nu(\varphi)^{2}H(y,\mathrm{d}\nu)\stackrel{b.p.}{\longrightarrow}0\mbox{ as }t\to+\infty.$$
Hence, by Lemma \ref{lem2.2}(ii),
\begin{align*}
\e^{-\lambda_{1}t}\int_{0}^{t}\int_{|x|>\epsilon\e^{\frac{\lambda_{1}}{2}t}}|x|^{2}\nu^{\tilde{\mathrm{M}}}(\mathrm{d}s,\mathrm{d}x)
\to&0\mbox{ in }L^{2}(\pp_{\delta_{x}})\mbox{ as }t\to+\infty.
\end{align*}
Therefore, both conditions of Lemma \ref{lemA.0} are satisfied by the martingale $\{\tilde{\mathrm{M}}_{t}(\varphi):\ t\ge 0\}$. It follows that, $\e^{-\frac{\lambda_{1}}{2}t}\tilde{\mathrm{M}}_{t}(\varphi)$ converges stably in distribution
to $\sqrt{\lambda^{-3}_{1}\langle\vartheta[\varphi],\widetilde{\varphi}\rangle W^{\varphi}_{\infty}}\cdot N$, where $N$ is a standard normal random variable independent of $W^{\varphi}_{\infty}$. Noting \eqref{lemn2.3} and the fact that $\lim_{t\to+\infty}\e^{-\lambda_{1}t/2}\varphi(x)=0$, we get the second conclusion.
\end{proof}

\begin{lemma}\label{lem3.n2}
For every $f\in\mathcal{D}^{0}_{\mathrm{S}}$ and $0\le t\le s<+\infty$, we define
\begin{align*}
\sigma_{f}(s,t):=&\int_{0}^{s-t}\e^{-\lambda_{1}r}\langle\vartheta[M_{r+t}[f],M_{r}[f]],\widetilde{\varphi}\rangle\mathrm{d}r,\\
\kappa_{\varphi}(t):=&\frac{1}{\lambda^{3}_{1}}(2+\lambda_{1}t)\langle\vartheta[\varphi],\widetilde{\varphi}\rangle,\\
{\varsigma}_{\varphi,f}(s,t):=&-\frac{1}{\lambda_{1}}\big(\int_{0}^{t}\langle\vartheta[M_{r}[f],\varphi],\widetilde{\varphi}\rangle\mathrm{d}r+\int_{0}^{s-t}
\e^{-\lambda_{1}r}\langle\vartheta[M_{r+t}[f],\varphi],\widetilde{\varphi}\rangle\mathrm{d}r\big).
\end{align*}
Then the following limits exist and are finite.
\begin{align*}
\sigma_{f}(\infty,t):=&\lim_{s\to+\infty}\sigma_{f}(s,t)=\int_{0}^{+\infty}\e^{-\lambda_{1}r}\langle\vartheta[M_{r+t}[f],M_{r}[f]],\widetilde{\varphi}\rangle\mathrm{d}r,\\
{\varsigma}_{\varphi,f}(\infty,t):=&\lim_{s\to+\infty}{\varsigma}_{\varphi,f}(s,t)=-\frac{1}{\lambda_{1}}\big(\int_{0}^{t}\langle\vartheta[M_{r}[f],\varphi],\widetilde{\varphi}\rangle
\mathrm{d}r+\int_{0}^{+\infty}\e^{-\lambda_{1}r}\langle\vartheta[M_{r+t}[f],\varphi],\widetilde{\varphi}\rangle\mathrm{d}r\big).
\end{align*}
Moreover, if we write $B_{\varphi}(t)$ for $\e^{\frac{\lambda_{1}}{2}t}\big(\e^{-\lambda_{1}t}Y_{t}(\varphi)-\lambda^{-1}_{1}W^{\varphi}_{\infty}\big)$ and $\hat{A}_{f}(t,s)$ for $\e^{-\frac{\lambda_{1}}{2}(t+s)}\big(Y_{t+s}(f)-Y_{t}(f)- X_{t}(M_{s}[f])\big)$, then for any $0\le u\le v\le s<+\infty$,
\begin{align}\label{lem3.2.2}
&\langle \Cov_{\delta_{\cdot}}(\hat{A}_{f}(u,s),\hat{A}_{f}(v,s)),\widetilde{\varphi}\rangle=\e^{-\frac{\lambda_{1}}{2}(v-u)}\sigma_{f}(s,v-u),\nonumber\\
&\langle\Cov_{\delta_{\cdot}}(B_{\varphi}(s+u),B_{\varphi}(s+v)),\widetilde{\varphi}\rangle=\e^{-\frac{\lambda_{1}}{2}(v-u)}\kappa_{\varphi}(v-u)-\frac{1}{\lambda^{3}_{1}}\e^{-\lambda_{1}
(s+\frac{u+v}{2})}\langle\vartheta[\varphi],\widetilde{\varphi}\rangle,\nonumber\\
&\langle \Cov_{\delta_{\cdot}}(\hat{A}_{f}(u,s),B_{\varphi}(s+v)),\widetilde{\varphi}\rangle=\e^{-\frac{\lambda_{1}}{2}(v-u)}{\varsigma}_{\varphi,f}(s,0),\nonumber\\
&\langle\Cov_{\delta_{\cdot}}(\hat{A}_{f}(v,s),B_{\varphi}(s+u)),\widetilde{\varphi}\rangle=\e^{-\frac{\lambda_{1}}{2}(v-u)}{\varsigma}_{\varphi,f}(s,v-u).
\end{align}
\end{lemma}
The proof of Lemma \ref{lem3.n2} involves lengthy but straightforward computations and is deferred to Appendix \ref{sec:A.3}.

{For two distributions $\nu_{1}$ and $\nu_{2}$ on $\R^{d}$, define
$$\mathbf{d}(\nu_{1},\nu_{2}):=\sup\{\big|\langle f,\nu_{1}\rangle -\langle f,\nu_{2}\rangle\big|:\ f:\R^{d}\to\R\mbox{ satisfies }\|f\|_{\infty}+\sup_{x,y\in\R^{d}:x\not=y}\frac{|f(x)-f(y)|}{\|x-y\|}\le 1\}.$$
Here $\|\cdot\|$ denotes the Euclid norm on $\R^{d}$.
Then by \cite[Theorem 11.3.3]{Dudley}, $\mathrm{d}$ is a metric and it generates the same topology as the weak convergence of distributions. For notational simplicity, if $X=(X_{1},\cdots,X_{d})$ and $Y=(Y_{1},\cdots,Y_{d})$ are two $\R^{d}$-valued random variables, we use $\mathbf{d}(X,Y)$ to denote the distance between their distributions. In particular, if $X,Y$ are defined in the same probability space, then from the definition we have
\begin{equation}\label{ineq1}
\mathbf{d}(X,Y)\le \mathrm{E}[\|X-Y\|]\le \sum_{k=1}^{d}\mathrm{E}[|X_{k}-Y_{k}|^{2}]^{1/2}.
\end{equation}
}

\begin{proposition}\label{prop:fddsuper}
Suppose $f\in\mathcal{D}^{0}_{\mathrm{S}}$. Under $\pp_{\delta_{x}}$,
\begin{equation*}
\big(\e^{-\frac{\lambda_{1}}{2}(t+s)}Y_{t+s}(f),\e^{\frac{\lambda_{1}}{2}(t+s)}{Z^{(t)}_{s}(\varphi)}\big)_{s\ge 0}
\stackrel{f.d.d.}{\longrightarrow}\sqrt{W^{\varphi}_{\infty}}(G^{1}(f),G^{2}(\varphi))\mbox{ as }t\to+\infty,
\end{equation*}
where $(G^{1}(f),G^{2}(\varphi)):=(G^{1}_{s}(f),G^{2}_{s}(\varphi))_{s\ge 0}$ is an $\R^{2}$-valued centered Gaussian process with covariances given by
\begin{align*}
\Cov(G^{1}_{s}(f),G^{1}_{t}(f))=&\e^{-\frac{\lambda_{1}}{2}|s-t|}\sigma_{f}(\infty,|s-t|),\\
\Cov(G^{2}_{s}(\varphi),G^{2}_{t}(\varphi))=&\e^{-\frac{\lambda_{1}}{2}|s-t|}\kappa_{\varphi}(|s-t|),\\
\Cov(G^{1}_{s}(f),G^{2}_{t}(\varphi))=&
\e^{-\frac{\lambda_{1}}{2}|s-t|}{\varsigma}_{\varphi,f}(\infty,(s-t)\vee 0)
\end{align*}
for all $s,t\ge 0$, and $(G^{1}_{s}(f),G^{2}_{s}(\varphi))_{s\ge 0}$ is independent of $W^{\varphi}_{\infty}$.
\end{proposition}

\begin{proof}
For notational simplicity, we use $A(t)$, $B(t)$ and $\hat{A}(t,s)$ to denote $\e^{-\frac{\lambda_{1}}{2}t}Y_{t}(f)$, $\e^{\frac{\lambda_{1}}{2}t}\big(\e^{-\lambda_{1}t}Y_{t}(\varphi)-\lambda^{-1}_{1}W^{\varphi}_{\infty}\big)$ and $\e^{-\frac{\lambda_{1}}{2}(t+s)}\big(Y_{t+s}(f)-Y_{t}(f)- X_{t}\big(M_{s}[f]\big)\big)$, respectively.

\noindent Step 1. Suppose $0=r_{1}\le r_{2}\le\cdots r_{k}<+\infty$. We shall show that for every $s\in [r_{k},+\infty)$,
\begin{equation}\label{lem3.1.0}
(\hat{A}(t+r_{j},s),B(t+s+r_{j}))_{1\le j\le k}\stackrel{d}{\to}\sqrt{W^{\varphi}_{\infty}}(G^{1,s}_{r_{j}}(f),G^{2,s}_{r_{j}}(\varphi))_{1\le j\le k}\mbox{ as }t\to+\infty,
\end{equation}
where $(G^{1,s}_{r_{j}}(f),G^{2,s}_{r_{j}}(\varphi))_{1\le j\le k}$ is an $\R^{2k}$-valued Gaussian random variable with mean $0$ and covariances given by
\begin{align*}
\Cov(G^{1,s}_{r_{i}}(f),G^{1,s}_{r_{j}}(f))=&\e^{-\frac{\lambda_{1}}{2}|r_{i}-r_{j}|}\sigma_{f}(s,|r_{i}-r_{j}|),\\
\Cov(G^{2,s}_{r_{i}}(\varphi),G^{2,s}_{r_{j}}(\varphi))=&\e^{-\frac{\lambda_{1}}{2}|r_{i}-r_{j}|}\kappa_{\varphi}(|r_{i}-r_{j}|),\\
\Cov(G^{1,s}_{r_{i}}(f),G^{2,s}_{r_{j}}(\varphi))=&
\e^{-\frac{\lambda_{1}}{2}|r_{i}-r_{j}|}{\varsigma}_{\varphi,f}(s,(r_{i}-r_{j})\vee 0),
\end{align*}
and $(G^{1,s}_{r_{j}}(f),G^{2,s}_{r_{j}}(\varphi))_{1\le j\le k}$ is independent of $W^{\varphi}_{\infty}$.

For $\vec{\theta}=(\theta_{1},\cdots,\theta_{k})$ and $\vec{\eta}=(\eta_{1},\cdots,\eta_{k})\in \R^{k}$, let $K^{s}_{t}(\vec{\theta},\vec{\eta})$ be the characteristic function of
$(\hat{A}(t+r_{j},s),B(t+s+r_{j}))_{1\le j\le k}$ given by
$$K^{s}_{t}(\vec{\theta},\vec{\eta})=\pp_{\delta_{x}}\big[\exp\{\mathrm{i}\sum_{j=1}^{k}\big(\theta_{j}\hat{A}(t+r_{j},s)+\eta_{j}B(t+s+r_{j})\big)\}\big].$$
We note that, by the Markov property, given $\mathcal{F}_{t}$, $(\hat{A}(t+r_{j},s),B(t+s+r_{j})-\e^{-\frac{\lambda_{1}}{2}(t+s+r_{j})}Y_{t}(\varphi))_{1\le j\le k}$ is equal in distribution to $\big((\e^{-\frac{\lambda_{1}}{2}t}\hat{A}(r_{j},s),\e^{-\frac{\lambda_{1}}{2}t}B(s+r_{j}))_{1\le j\le k},\pp_{X_{t}}\big)$. Hence we have
\begin{align}\label{lem3.1.1}
K^{s}_{t}(\vec{\theta},\vec{\eta})=&\pp_{\delta_{x}}\big[\exp\{\mathrm{i}\sum_{j=1}^{k}\eta_{j}\e^{-\frac{\lambda_{1}}{2}(t+s+r_{j})}Y_{t}(\varphi)\}\pp_{X_{t}}\big[
\exp\{\mathrm{i}\e^{-\frac{\lambda_{1}}{2}t}\sum_{j=1}^{k}\big(\theta_{j}\hat{A}(r_{j},s)+\eta_{j}B(s+r_{j})\big)\}\big]\big].
\end{align}
Let $\zeta(\vec{r},\vec{\theta},\vec{\eta},s):=\sum_{j=1}^{k}\big(\theta_{j}\hat{A}(r_{j},s)+\eta_{j}B(s+r_{j})\big)$. The regular branching property of superprocess implies that for every $\mu\in\mf$, under $\pp_{\mu}$, $\zeta(\vec{r},\vec{\theta},\vec{\eta},s)$ is an infinitely divisible random variable. Moreover, under (M2), $\zeta(\vec{r},\vec{\theta},\vec{\eta},s)$ has finite mean and variance. Therefore, the characteristic function of $\zeta(\vec{r},\vec{\theta},\vec{\eta},s)$ has the following canonical representation:
$$\pp_{\mu}\big[\e^{\mathrm{i}\theta\zeta(\vec{r},\vec{\theta},\vec{\eta},s)}\big]=\exp\{\langle \mathrm{i}\theta\pp_{\delta_{\cdot}}\big[\zeta(\vec{r},\vec{\theta},\vec{\eta},s)\big]-\frac{1}{2}\theta^{2}\Var_{\delta_{\cdot}}[\zeta(\vec{r},\vec{\theta},\vec{\eta},s)]
+\int_{\R\setminus\{0\}}\big(\e^{\mathrm{i}\theta y}-1-\mathrm{i}\theta y+\frac{1}{2}\theta^{2}y^{2}\big)\pi_{\zeta}(\cdot,\mathrm{d}y),\mu\rangle\}$$
for all $\theta\in \R$, where $\pi_{\zeta}(z,\mathrm{d}y)$ is a kernel from $E$ to $\R\setminus\{0\}$ satisfying that
\begin{equation}\label{lem3.1.2}
\int_{\R\setminus\{0\}}y^{2}\pi_{\zeta}(z,\mathrm{d}y)\le \Var_{\delta_{z}}[\zeta(\vec{r},\vec{\theta},\vec{\eta},s)]\quad\forall z\in E.
\end{equation}
Since $\pp_{\delta_{y}}\big[\hat{A}(r_{j},s)\big]=0$ for all $y\in E$, we have
$$\pp_{\delta_{y}}[\zeta(\vec{r},\vec{\theta},\vec{\eta},s)]=\sum_{j=1}^{k}\eta_{j}\big(\e^{-\frac{\lambda_{1}}{2}(s+r_{j})}
M_{s+r_{j}}[\varphi](y)-\frac{1}{\lambda_{1}}\e^{\frac{\lambda_{1}}{2}(s+r_{j})}\varphi(y)\big)=-\frac{1}{\lambda_{1}}\sum_{j=1}^{k}\eta_{j}\e^{-\frac{\lambda_{1}}{2}(s+r_{j})}\varphi(y).$$
Inserting this into the right hand side of \eqref{lem3.1.1} yields that
\begin{align}
K^{s}_{t}(\vec{\theta},\vec{\eta})=&\pp_{\delta_{x}}\big[\exp\{\mathrm{i}\sum_{j=1}^{k}\eta_{j}\e^{-\frac{\lambda_{1}}{2}(t+s+r_{j})}Y_{t}(\varphi)\}\pp_{X_{t}}\big[\exp\{\mathrm{i}
\e^{-\frac{\lambda_{1}}{2}t}\zeta(\vec{r},\vec{\theta},\vec{\eta},s)\}\big]\big]\nonumber\\
=&\pp_{\delta_{x}}\big[\exp\big\{\mathrm{i}\sum_{j=1}^{k}\eta_{j}\e^{-\frac{\lambda_{1}}{2}(s+r_{j}+t)}\big(Y_{t}(\varphi)-\frac{1}{\lambda_{1}}X_{t}(\varphi)\big)\nonumber\\
&-\frac{1}{2}\e^{-\lambda_{1}t}\langle \Var_{\delta_{\cdot}}[\zeta(\vec{r},\vec{\theta},\vec{\eta},s)],X_{t}\rangle+\langle \mathcal{L}_{\pi_{\zeta}}[\e^{-\frac{\lambda_{1}}{2}t}],X_{t}\rangle\big\}\big],\label{lem3.1.4}
\end{align}
where $\mathcal{L}_{\pi_{\zeta}}[\theta](z):=\int_{\R\setminus\{0\}}\big(\e^{\mathrm{i}\theta y}-1-\mathrm{i}\theta y+\frac{1}{2}\theta^{2}y^{2}\big)\pi_{\zeta}(z,\mathrm{d}y)$ for $z\in E$ and $\theta\in\R$. Using the inequality that
\begin{equation}\label{ineq2}
\big|\e^{\mathrm{i}y}-1-\mathrm{i}y+\frac{1}{2}y^{2}\big|\le |y|^{2}\big(\frac{|y|}{6}\wedge 1\big)\quad\forall y\in\R,
\end{equation}
we get for every $z\in E$,
\begin{align}\label{lem3.1.3}
\big|\e^{\lambda_{1}t}\mathcal{L}_{\pi_{\zeta}}[\e^{-\frac{\lambda_{1}}{2}t}](z)\big|\le &\e^{\lambda_{1}t}\int_{\R\setminus\{0\}}\big|\e^{\mathrm{i}\e^{-\frac{\lambda_{1}}{2}t} y}-1-\mathrm{i}\e^{-\frac{\lambda_{1}}{2}t} y+\frac{1}{2}\e^{-\lambda_{1}t}y^{2}\big|\pi_{\zeta}(z,\mathrm{d}y)\nonumber\\
\le&\int_{\R\setminus\{0\}}|y|^{2}\big(\frac{1}{6}\e^{-\frac{\lambda_{1}}{2}t}|y|\wedge 1\big)\pi_{\zeta}(z,\mathrm{d}y).
\end{align}
We can easily verify that, under (M2), $z\mapsto \Var_{\delta_{z}}[\zeta(\vec{r},\vec{\theta},\vec{\eta},s)]$ is a bounded function on $E$. This combined with \eqref{lem3.1.2} and the dominated convergence theorem implies that the final integral of \eqref{lem3.1.3} converges boundedly and pointwise to $0$ as $t\to+\infty$, which in turn yields that
\begin{equation} \label{lem3.1.8}
\e^{\lambda_{1}t}\mathcal{L}_{\pi_{\zeta}}[\e^{-\frac{\lambda_{1}}{2}t}](z)\stackrel{b.p.}{\longrightarrow}0\mbox{ as }t\to+\infty.
\end{equation}
 Hence by Lemma \ref{lem2.2}(i), we have
$$\lim_{t\to+\infty}\e^{-\lambda_{1}t}\langle\Var_{\delta_{\cdot}}[\zeta(\vec{r},\vec{\theta},\vec{\eta},s)],X_{t}\rangle=
\langle\Var_{\delta_{\cdot}}[\zeta(\vec{r},\vec{\theta},\vec{\eta},s)],\widetilde{\varphi}\rangle W^{\varphi}_{\infty}\mbox{ in }L^{2}(\pp_{\delta_{x}}),$$
and $\lim_{t\to+\infty}\langle \mathcal{L}_{\pi_{\zeta}}[\e^{-\frac{\lambda_{1}}{2}t}],X_{t}\rangle =0$ in $L^{2}(\pp_{\delta_{x}})$.
It follows that
\begin{align*}
&\lim_{t\to+\infty}\pp_{\delta_{x}}\big[\big|\exp\big\{-\frac{1}{2}\e^{-\lambda_{1}t}\langle \Var_{\delta_{\cdot}}[\zeta(\vec{r},\vec{\theta},\vec{\eta},s)],X_{t}\rangle+\langle \mathcal{L}_{\pi_{\zeta}}[\e^{-\frac{\lambda_{1}}{2}t}],X_{t}\rangle\big\}-\exp\big\{-\frac{1}{2}\langle \Var_{\delta_{\cdot}}[\zeta(\vec{r},\vec{\theta},\vec{\eta},s)],\widetilde{\varphi}\rangle W^{\varphi}_{\infty}\big\}\big|\big]\\
=&0.
\end{align*}
Combining this with \eqref{lem3.1.4}, we get
\begin{align*}
\lim_{t\to+\infty}K^{s}_{t}(\vec{\theta},\vec{\eta})=&\lim_{t\to+\infty}\pp_{\delta_{x}}\big[\exp\big\{\mathrm{i}\big(\sum_{j=1}^{k}\eta_{j}\e^{-\frac{\lambda_{1}}{2}(s+r_{j})}\big)
\e^{-\frac{\lambda_{1}}{2}t}\big(Y_{t}(\varphi)-\frac{1}{\lambda_{1}}X_{t}(\varphi)\big)-\frac{1}{2}\langle\Var_{\delta_{\cdot}}[\zeta(\vec{r},\vec{\theta},\vec{\eta},s)],\widetilde{\varphi}
\rangle W^{\varphi}_{\infty}\big\}\big]\\
=&\pp_{\delta_{x}}\big[\exp\big\{-\frac{1}{2}W^{\varphi}_{\infty}\big[\big(\sum_{j=1}^{k}\eta_{j}\e^{-\frac{\lambda_{1}}{2}(s+r_{j})}\big)^{2}\frac{1}{\lambda_{1}^{3}}\langle\vartheta[\varphi]
,\widetilde{\varphi}\rangle+\langle\Var_{\delta_{\cdot}}[\zeta(\vec{r},\vec{\theta},\vec{\eta},s)],\widetilde{\varphi}\rangle\big]\big\}\big].
\end{align*}
The second equality follows from Lemma \ref{lem3.n1}. From the definition of $\zeta(\vec{r},\vec{\theta},\vec{\eta},s)$, we have
\begin{align*}
\langle \Var_{\delta_{\cdot}}[\zeta(\vec{r},\vec{\theta},\vec{\eta},s)],\widetilde{\varphi}\rangle=&\sum_{j=1}^{k}\theta^{2}_{j}\langle\Var_{\delta_{\cdot}}[\hat{A}(r_{j},s)],\widetilde{\varphi}\rangle
+2\sum_{1\le i<j\le k}\theta_{i}\theta_{j}\langle\Cov_{\delta_{\cdot}}(\hat{A}(r_{i},s),\hat{A}(r_{j},s)),\widetilde{\varphi}\rangle\\
&+\sum_{j=1}^{k}\eta^{2}_{j}\langle\Var_{\delta_{\cdot}}[B(r_{j}+s)],\widetilde{\varphi}\rangle+2\sum_{1\le i<j\le k}\eta_{i}\eta_{j}
\langle\Cov_{\delta_{\cdot}}(B(r_{i}+s),B(r_{j}+s)),\widetilde{\varphi}\rangle\\
&+2\sum_{i=1}^{k}\sum_{j=1}^{k}\theta_{i}\eta_{j}\langle \Cov_{\delta_{\cdot}}(\hat{A}(r_{i},s),B(r_{j}+s)),\widetilde{\varphi}\rangle.
\end{align*}
Simple computation using Lemma \ref{lem3.n2} gives that
\begin{align*}
\lim_{t\to+\infty}K^{s}_{t}(\vec{\theta},\vec{\eta})=&\pp_{\delta_{x}}\big[\exp\big\{-\frac{1}{2}W^{\varphi}_{\infty}\big(\sum_{j=1}^{k}\theta_{j}^{2}\sigma_{f}(s,0)+2\sum_{1\le i<j\le k}\theta_{i}\theta_{j}\e^{-\frac{\lambda_{1}}{2}(r_{j}-r_{i})}\sigma_{f}(s,r_{j}-r_{i})\\
&+\sum_{j=1}^{k}\eta_{j}^{2}\kappa_{\varphi}(0)+2\sum_{1\le i<j\le k}\eta_{i}\eta_{j}\e^{-\frac{\lambda_{1}}{2}(r_{j}-r_{i})}\kappa_{\varphi}(r_{j}-r_{i})\\
&+2\sum_{1\le i\le j\le k}\theta_{i}\eta_{j}\e^{-\frac{\lambda_{1}}{2}(r_{j}-r_{i})}{\varsigma}_{\varphi,f}(s,0)+2\sum_{1\le j<i\le k}\theta_{i}\eta_{j}\e^{-\frac{\lambda_{1}}{2}(r_{i}-r_{j})}
{\varsigma}_{\varphi,f}(s,r_{i}-r_{j})
\big)\big\}\big].
\end{align*}
Note that the expectation in the right hand side is equal to the characteristic function of the random vector $\sqrt{W^{\varphi}_{\infty}}(G^{1,s}_{r_{j}}(f),G^{2,s}_{r_{j}}(\varphi))_{1\le j\le k}$. Hence, we have proved \eqref{lem3.1.0}.

\smallskip

\noindent Step 2. We shall show that
\begin{equation}\label{lem3.1.5}
\e^{-\frac{\lambda_{1}}{2}(t+s)}Y_{t}(f)\to 0\mbox{ in }L^{2}(\pp_{\delta_{x}})\mbox{ as }t\to+\infty\mbox{ and then }s\to+\infty.
\end{equation}
\begin{equation}\label{lem3.1.6}
\e^{-\frac{\lambda_{1}}{2}(t+s)} X_{t}\big(M_{s}[f]\big)\to 0\mbox{ in }L^{2}(\pp_{\delta_{x}})\mbox{ as }t\to+\infty\mbox{ and then }s\to+\infty.
\end{equation}

\noindent Proof of \eqref{lem3.1.5}: Note that
$$\pp_{\delta_{x}}\big[\big(\e^{-\frac{\lambda_{1}}{2}(t+s)} Y_{t}(f)\big)^{2}\big]=\e^{-\lambda_{1}(t+s)}\Var_{\delta_{x}}[ Y_{t}(f)]+\e^{-\lambda_{1}(t+s)}M_{t}[f](x)^{2}.$$
By Lemma \ref{lem1}(i), we have $\pp_{\delta_{x}}\big[\big(\e^{-\frac{\lambda_{1}}{2}(t+s)} Y_{t}(f)\big)^{2}\big]\to 0$ as $t\to+\infty$ and then $s\to+\infty$, and hence \eqref{lem3.1.5} follows immediately.

\smallskip

\noindent Proof of \eqref{lem3.1.6}: Let $h_{s}(y):=\e^{-\frac{\lambda_{1}}{2}s}M_{s}[f](y)$ for $y\in E$ and $s\ge 0$. Then $\langle h_{s},\widetilde{\varphi}\rangle=\e^{-\frac{\lambda_{1}}{2}s}\int_{0}^{s}\langle T_{s}f,\widetilde{\varphi}\rangle\mathrm{d}s=\e^{\frac{\lambda_{1}}{2}s}\int_{0}^{s}\langle f,\widetilde{\varphi}\rangle\mathrm{d}s=0$.
Note that $\e^{-\lambda_{1}t}T_{t}h_{s}=\e^{-\frac{\lambda_{1}}{2}s}\int_{0}^{s}\e^{\lambda_{1}u}\cdot\e^{-\lambda_{1}(t+u)}T_{t+u}f\mathrm{d}u$.
By the definition of ${\rho}(f)$, for any $\epsilon>0$, when $t$ is sufficiently large,
\begin{align*}
\big| \e^{-\lambda_{1}t}T_{t}h_{s}\big|\le&\e^{-\frac{\lambda_{1}}{2}s}\int_{0}^{s}\e^{\lambda_{1}u}\big| \e^{-\lambda_{1}(t+u)}T_{t+u}f\big|\mathrm{d}u\\
\le&\e^{-\frac{\lambda_{1}}{2}s}\int_{0}^{s}\e^{\lambda_{1}u}\cdot\e^{(-{\rho}(f)+\epsilon)(t+u)}\mathrm{d}u=:c_{1}(s)\e^{(-\rho(f)+\epsilon)t}.
\end{align*}
This implies that ${\rho}(h_{s})\ge {\rho}(f)>\lambda_{1}/2$, and hence $h_{s}\in\mathcal{D}^{0}_{\mathrm{S}}$.
By Lemma \ref{lem0}(i), we have
\begin{equation}\label{lem3.1.7}
\pp_{\delta_{x}}\big[\e^{-\frac{\lambda_{1}}{2}t} X_{t}(h_{s})\big]=\e^{-\frac{\lambda_{1}}{2}t}T_{t}h_{s}(x)\to 0\mbox{ as }t\to+\infty,
\end{equation}
and
$$\Var_{\delta_{x}}\big[\e^{-\frac{\lambda_{1}}{2}t} X_{t}(h_{s})\big]=\e^{-\lambda_{1}t}\Var_{\delta_{x}}\big[ X_{t}(h_{s})\big]\to
\varphi(x)\varrho^{2}_{h_{s}}\mbox{ as }t\to+\infty,$$
where $\varrho^{2}_{h_{s}}=\int_{0}^{+\infty}\e^{-\lambda_{1}u}\langle\vartheta[T_{u}h_{s}],\widetilde{\varphi}\rangle\mathrm{d}u=\int_{0}^{+\infty}\e^{-\lambda_{1}(u+s)}\langle\vartheta[M_{s+u}[f]-M_{u}[f]],\widetilde{\varphi}\rangle\mathrm{d}u$.
Using the fact that $\vartheta[f_{1}-f_{2}]\le 2(\vartheta[f_{1}]+\vartheta[f_{2}])$ for any $f_{1},f_{2}\in\mathcal{B}_{b}(E)$, we have
\begin{align}\label{prop1.10}
\varrho^{2}_{h_{s}}\le &2\int_{0}^{+\infty}\e^{-\lambda_{1}(u+s)}\langle\vartheta[M_{u+s}[f]]+\vartheta[M_{u}[f]],\widetilde{\varphi}\rangle\mathrm{d}u\nonumber\\
=&2\Big(\int_{s}^{+\infty}\e^{-\lambda_{1}u}\langle\vartheta[M_{u}[f]],\widetilde{\varphi}\rangle\mathrm{d}u+
\e^{-\lambda_{1}s}\int_{0}^{+\infty}\e^{-\lambda_{1}u}\langle\vartheta[M_{u}[f]],\widetilde{\varphi}\rangle\mathrm{d}u\Big).
\end{align}
We can easily show by \eqref{lem3.2.1} that the integral $\int_{0}^{+\infty}\e^{-\lambda_{1}u}\langle\vartheta[M_{u}[f]],\widetilde{\varphi}\rangle\mathrm{d}u$ is finite.
It thus follows by \eqref{prop1.10} that $\lim_{s\to+\infty}\varrho^{2}_{h_{s}}=0$. Hence we get $\lim_{s\to+\infty}\lim_{t\to+\infty}\Var_{\delta_{x}}[\e^{-\frac{\lambda_{1}}{2}t} X_{t}(h_{s})]=0,$
which together with \eqref{lem3.1.7} yields \eqref{lem3.1.6}.
\smallskip

\noindent Step 3. Recall the metric `$\mathbf{d}$' of two distributions defined above \eqref{ineq1}. We have for $t,s\ge 0$,
\begin{align*}
&\mathbf{d}\big((A(t+s+r_{j}),B(t+s+r_{j}))_{1\le j\le k},\sqrt{W^{\varphi}_{\infty}}(G^{1}_{r_{j}}(f),G^{2}_{r_{j}}(\varphi))_{1\le j\le k}\big)\\
\le&\mathbf{d}\big((A(t+s+r_{j}),B(t+s+r_{j}))_{1\le j\le k},(\hat{A}(t+r_{j},s),B(t+s+r_{j}))_{1\le j\le k}\big)\\
&+\mathbf{d}\big((\hat{A}(t+r_{j},s),B(t+s+r_{j}))_{1\le j\le k},\sqrt{W^{\varphi}_{\infty}}(G^{1,s}_{r_{j}}(f),G^{2,s}_{r_{j}}(\varphi))_{1\le j\le k}\big)\\
&+\mathbf{d}\big(\sqrt{W^{\varphi}_{\infty}}(G^{1,s}_{r_{j}}(f),G^{2,s}_{r_{j}}(\varphi))_{1\le j\le k},\sqrt{W^{\varphi}_{\infty}}(G^{1}_{r_{j}}(f),G^{2}_{r_{j}}(\varphi))_{1\le j\le k}\big).
\end{align*}
Clearly by \eqref{lem3.1.0}, the second term  in the right hand side converges to $0$ as $t\to+\infty$. Note that $A(t+s+r_{j})-\hat{A}(t+r_{j},s)=\e^{-\frac{\lambda_{1}}{2}(t+s+r_{j})}\big(Y_{t+r_{j}}(f)+X_{t+r_{j}}(M_{s}[f])\big)$.
By \eqref{ineq1}, the first term in the right hand side is bounded by
$$\sum_{j=1}^{k}\|\e^{-\frac{\lambda_{1}}{2}(t+s+r_{j})}\big(Y_{t+r_{j}}(f)+X_{t+r_{j}}\big(M_{s}[f]\big)\big)\|_{L^{2}(\pp_{\delta_{x}})},$$
and hence by \eqref{lem3.1.5} and \eqref{lem3.1.6} it converges to $0$ as first $t\to+\infty$ and then $s\to+\infty$.

We note by the first conclusion of Lemma \ref{lem3.n2} that the covariance functions of the Gaussian random vector $(G^{1,s}_{r_{j}}(f),G^{2,s}_{r_{j}}(\varphi))_{1\le j\le k}$ converges to that of $(G^{1}_{r_{j}}(f),G^{2}_{r_{j}}(\varphi))_{1\le j\le k}$ as $s\to+\infty$. Note also that both random vectors are independent of $W^{\varphi}_{\infty}$. Using these facts, we can easily verify that $\sqrt{W^{\varphi}_{\infty}}(G^{1,s}_{r_{j}}(f),G^{2,s}_{r_{j}}(\varphi))_{1\le j\le k}$ converges in distribution to $\sqrt{W^{\varphi}_{\infty}}(G^{1}_{r_{j}}(f),G^{2}_{r_{j}}(\varphi))_{1\le j\le k}$ as $s\to+\infty$. Hence
$$\lim_{s\to+\infty}\mathbf{d}\big(\sqrt{W^{\varphi}_{\infty}}(G^{1,s}_{r_{j}}(f),G^{2,s}_{r_{j}}(\varphi))_{1\le j\le k},\sqrt{W^{\varphi}_{\infty}}(G^{1}_{r_{j}}(f),G^{2}_{r_{j}}(\varphi))_{1\le j\le k}\big)=0.$$
Therefore, we have proved
$$\lim_{s\to+\infty}\lim_{t\to+\infty}\mathbf{d}\big((A(t+s+r_{j}),B(t+s+r_{j}))_{1\le j\le k},\sqrt{W^{\varphi}_{\infty}}(G^{1}_{r_{j}}(f),G^{2}_{r_{j}}(\varphi))_{1\le j\le k}\big)=0,$$
which in turn yields that $(A(t+r_{j}),B(t+r_{j}))_{1\le j\le k}\stackrel{d}{\to}\sqrt{W^{\varphi}_{\infty}}(G^{1}_{r_{j}}(f),G^{2}_{r_{j}}(\varphi))_{1\le j\le k}$ as $t\to+\infty$.
Since $A(t+r_{j})=\e^{\frac{-\lambda_{1}}{2}(t+r_{j})}Y_{t+r_{j}}(f)$ and $B(t+r_{j})=\e^{\frac{\lambda_{1}}{2}(t+r_{j})}Z^{(t)}_{r_{j}}(\varphi)$, we have proved this proposition.
\end{proof}

}

\begin{lemma}\label{lemA.1}
Suppose in addition that (M4) holds.
Then for every $x\in E$ and $f\in\mathcal{D}^{0}(\lambda_{1}/2,r)$,
$$\pp_{\delta_{x}}\big[Y_{t}(f)^{4}+X_{t}(f)^{4}\big]=O(t^{2+4r}\e^{2\lambda_{1}t})\mbox{ as }t\to+\infty.$$
\end{lemma}
\begin{proof}
Fix an arbitrary $x\in E$ and $f\in\mathcal{D}^{0}(\lambda_{1}/2,r)$. By \cite[Proposition 4.3]{PY20}, under (H1) and (H2), for every $ t\in (0,+\infty)$ and $g\in \mathcal{B}_{b}(E)$, $ X_{t}(g)$ can be represented by
\begin{equation}\label{lemA.1.1}
 X_{t}(g)=T_{ t}g(x)+\int_{0}^{ t}\int_{E}T_{ t-s}g(y)M(\mathrm{d}s,\mathrm{d}y),
\end{equation}
where $M(\mathrm{d}s,\mathrm{d}y)$ is the martingale measure related to the superprocess $X$.
{
Assuming additionally (M4) holds, \cite[Lemma A.3]{Yang25} further proves that there exists some $c_{1}>0$ such that for any $g\in\mathcal{B}_{b}(E)$, $ t>0$ and $x\in E$,
\begin{align}\label{lemA.1.4}
\pp_{\delta_{x}}\big[\big(\int_{0}^{ t}\int_{E}T_{ t-s}g(y)M(\mathrm{d}s,\mathrm{d}y)\big)^{4}\big]
\le&c_{1}\Big(\Var_{\delta_{x}}[ X_{t}(g)]^{2}+\int_{0}^{ t}T_{s}\big(\vartheta\big[\Var_{\delta_{\cdot}}\big[ X_{t-s}(g)\big]\big]\big)(x)\mathrm{d}s\Big).
\end{align}
It follows by \eqref{lemA.1.1} and \eqref{lemA.1.4} that there is a constant $c_{2}>0$ such that
\begin{equation}\label{lemA.1.6}
\pp_{\delta_{x}}\big[ X_{t}(f)^{4}\big]\le c_{2}\Big((T_{t}f(x))^{4}+\Var_{\delta_{x}}[ X_{t}(f)]^{2}+\int_{0}^{t}T_{s}\big(\vartheta\big[\Var_{\delta_{\cdot}}[ X_{t-s}(f)]\big]\big)(x)\mathrm{d}s\Big)\quad\forall x\in E,\ t>0.
\end{equation}
Since $f\in \mathcal{D}^{0}(\lambda_{1}/2,r)$, we have $t^{-2-4r}\e^{-2\lambda_{1}t}T_{t}f(x)^{4}\to 0$ as $t\to+\infty$, and by Lemma \ref{lem0}(ii),
\begin{equation}\label{lemA.1.7}
\frac{\e^{-\lambda_{1}t}}{t^{1+2r}}\Var_{\delta_{y}}\big[ X_{t}(f)\big]\stackrel{b.p.}{\longrightarrow}\frac{\langle\vartheta[f^{*}],\widetilde{\varphi}\rangle}{1+2r}\varphi(y)\mbox{ as }t\to+\infty.
\end{equation}
Thus, there exists some $t_{1}\ge 1$, such that $M_{1}:=\sup_{t\ge t_{1},\,y\in E}t^{-1-2r}\e^{-\lambda_{1}t}\Var_{\delta_{y}}[ X_{t}(f)]<+\infty$. Recall from \eqref{eq1.9}
that there is a constant $c_{2}>0$ such that $|T_{t}g(y)|\le c_{2}\e^{\lambda_{1}t}\|g\|_{\infty}$ for all $y\in E$, $g\in\mathcal{B}_{b}(E)$ and $t\ge 0$. Hence, we have
\begin{align}\label{lemA.1.5}
\int_{0}^{t}T_{s}\big(\vartheta\big[\Var_{\delta_{\cdot}}[ X_{t-s}(f)]\big]\big)(x)\mathrm{d}s\le&c_{2}\int_{0}^{t}\e^{\lambda_{1}s}\|\vartheta\big[\Var_{\delta_{\cdot}}[ X_{t-s}(f)]\big]\|_{\infty}\mathrm{d}s\nonumber\\
\le&c_{2}\|\vartheta[1]\|_{\infty}\int_{0}^{t}\e^{\lambda_{1}s}\|\Var_{\delta_{\cdot}}[ X_{t-s}(f)]\|^{2}_{\infty}\mathrm{d}s\nonumber\\
=&c_{2}\|\vartheta[1]\|_{\infty}\e^{\lambda_{1}t}\int_{0}^{t}\e^{-\lambda_{1}u}\|\Var_{\delta_{\cdot}}[ X_{u}(f)]\|^{2}_{\infty}\mathrm{d}u.
\end{align}
Clearly, $\int_{0}^{t_{1}}\e^{-\lambda_{1}u}\|\Var_{\delta_{\cdot}}[ X_{u}(f)]\|^{2}_{\infty}\mathrm{d}u$ is finite. Moreover, for $t\ge t_{1}$,
$$\int_{t_{1}}^{t}\e^{-\lambda_{1}u}\|\Var_{\delta_{\cdot}}[ X_{u}(f)]\|^{2}_{\infty}\mathrm{d}u\le \int_{t_{1}}^{t}M_{1}^{2}\e^{\lambda_{1}u}u^{2+4r}\mathrm{d}u=O(t^{2+4r}\e^{\lambda_{1}t})\mbox{ as }t\to+\infty.$$
Inserting this into the right hand side of \eqref{lemA.1.5} yields that $\int_{0}^{t}T_{s}\big(\vartheta\big[\Var_{\delta_{\cdot}}\big[ X_{t-s}(f)\big]\big]\big)(x)\mathrm{d}s
=O(t^{2+4r}\e^{2\lambda_{1}t})$ as $t\to+\infty$. Combining this with \eqref{lemA.1.6}, \eqref{lemA.1.7}, we get that
$\pp_{\delta_{x}}\big[X_{t}(f)^{4}\big]=O(t^{2+4r}\e^{2\lambda_{1}t})$ as $t\to+\infty$.

Next we deal with $\pp_{\delta_{x}}[Y_{t}(f)^{4}]$.
By Minkowski's inequality, we have
\begin{align}\label{lemA.1.8}
\pp_{\delta_{x}}[ Y_{t}(f)^{4}]=&\|\int_{0}^{t} X_{s}(f)\mathrm{d}s\|^{4}_{L^{4}(\pp_{\delta_{x}})}\le \big(\int_{0}^{t}\| X_{s}(f)\|_{L^{4}(\pp_{\delta_{x}})}\mathrm{d}s\big)^{4}
\end{align}
Note that there exists $t_{2}>0$ such that $M_{2}:=\sup_{t\ge t_{2}}t^{-2-4r}\e^{-2\lambda_{1}t}\pp_{\delta_{x}}[ X_{t}(f)^{4}]<+\infty$.
Hence, for $t\ge t_{2}$,
$$\int_{t_{2}}^{t}\| X_{s}(f)\|_{L^{4}(\pp_{\delta_{x}})}\mathrm{d}s\le \int_{t_{2}}^{t}\big(M_{2}s^{2+4r}\e^{2\lambda_{1}s}\big)^{1/4}\mathrm{d}s=O(t^{r+\frac{1}{2}}\e^{\frac{\lambda_{1}}{2}t})\mbox{ as }t\to+\infty.$$
Inserting this into the right hand side of \eqref{lemA.1.8} yields that $\pp_{\delta_{x}}[ Y_{t}(f)^{4}]=O(t^{2+4r}\e^{2\lambda_{1}t})$ as $t\to+\infty$.
}
\end{proof}

\begin{lemma}\label{lem3}
Suppose the assumptions of Lemma \ref{lemA.1} hold. Then for $f,g\in \mathcal{D}^{0}(\lambda_{1}/2,r)$,
$$t^{-(r+1/2)\e^{-\frac{\lambda_{1}}{2}t}}(Y_{t}(g),X_{t}(f))\stackrel{d}{\to}\sqrt{\frac{W^{\varphi}_{\infty}}{1+2r}}(N_{1}(g^{*}),N_{2}(f^{*}))\mbox{ as }t\to+\infty,$$
where $(N_{1}(g^{*}),N_{2}(f^{*}))$ is an $\R^{2}$-valued Gaussian random variable with mean $0$ and covariance matrix given by
$$
\begin{pmatrix}
\frac{4}{\lambda^{2}_{1}}\langle \vartheta[g^{*}],\widetilde{\varphi}\rangle &\frac{2}{\lambda_{1}}\langle\vartheta[f^{*},g^{*}],\widetilde{\varphi}\rangle\\
\frac{2}{\lambda_{1}}\langle\vartheta[f^{*},g^{*}],\widetilde{\varphi}\rangle &\langle \vartheta[f^{*}],\widetilde{\varphi}\rangle
\end{pmatrix},
$$
and $(N_{1}(g^{*}),N_{2}(f*))$ is independent of $W^{\varphi}_{\infty}$.
\end{lemma}

\proof Suppose $n>1$. We have
\begin{align*}
Y_{n t}(g)=&Y_{n t}(g)-\pp_{\delta_{x}}\big[Y_{n t}(g)\,|\,\mathcal{F}_{t}\big]+\pp_{\delta_{x}}\big[Y_{n t}(g)\,|\,\mathcal{F}_{t}\big],\\
X_{n t}(f)=&X_{n t}(f)-\pp_{\delta_{x}}\big[X_{n t}(f)\,|\,\mathcal{F}_{t}\big]+\pp_{\delta_{x}}\big[X_{n t}(f)\,|\,\mathcal{F}_{t}\big].
\end{align*}
Step 1. We shall show that
\begin{equation}\label{lem3.0}
\frac{\e^{-\frac{\lambda_{1}}{2}n t}}{(n t)^{(r+1/2)}}\pp_{\delta_{x}}\big[Y_{n t}(g)\,|\,\mathcal{F}_{t}\big]\to 0\mbox{ in }L^{2}(\pp_{\delta_{x}})\mbox{ as }t\to+\infty\mbox{ and then }n\to+\infty
\end{equation}
and
\begin{equation}\label{lem3.new2}
\frac{\e^{-\frac{\lambda_{1}}{2}nt}}{(nt)^{r+1/2}}\pp_{\delta_{x}}\big[X_{nt}(f)|\mathcal{F}_{t}\big]\to 0\mbox{ in }L^{2}(\pp_{\delta_{x}})\mbox{ as }t\to+\infty \mbox{ and then } n\to+\infty.
\end{equation}
\noindent Proof of \eqref{lem3.0}: By the Markov property we have
$$\pp_{\delta_{x}}\big[Y_{n t}(g)\,|\,\mathcal{F}_{t}\big]=\pp_{X_{t}}\big[Y_{(n-1)t}(g)\big]+Y_{t}(g)=\langle M_{(n-1)t}[g],X_{t}\rangle+Y_{t}(g).$$
Since $g\in\mathcal{D}^{0}(\lambda_{1}/2,r)$, it follows by Lemma \ref{lem1}(ii) that
$$\pp_{\delta_{x}}\big[\frac{\e^{-\frac{\lambda_{1}}{2}nt}}{(nt)^{r+1/2}}Y_{t}(g)\big]=\frac{\e^{-\frac{\lambda_{1}}{2}(n-1)t}}{n^{r+1/2}t^{1/2}}\cdot\frac{\e^{-\frac{\lambda_{1}}{2}t}}{t^{r}}
M_{t}[g](x)\to 0\mbox{ as }t\to+\infty,$$
and
$$\Var_{\delta_{x}}\big[\frac{\e^{-\frac{\lambda_{1}}{2}n t}}{(n t)^{r+1/2}}Y_{t}(g)\big]=\frac{\e^{-\lambda_{1}(n-1)t}}{n^{1+2r}}\cdot\frac{\e^{-\lambda_{1}t}}{t^{1+2r}}\Var_{\delta_{x}}\big[Y_{t}(g)\big]\to 0\mbox{ as }t\to+\infty.$$
Hence
\begin{equation}\label{lem3.9}
\frac{\e^{-\frac{\lambda_{1}}{2}nt}}{(nt)^{r+1/2}}Y_{t}(g)\to 0\mbox{ in }L^{2}(\pp_{\delta_{x}})\mbox{ as }t\to+\infty.
\end{equation}

On the other hand, by Lemma \ref{lem1}(ii), we have
\begin{equation}\label{lem3.1}
\pp_{\delta_{x}}\big[\frac{\e^{-\frac{\lambda_{1}}{2}n t}}{(n t)^{r+1/2}}\langle M_{(n-1)t}[g],X_{t}\rangle\big]=\frac{1}{(nt)^{1/2}}\cdot\frac{\e^{-\frac{\lambda_{1}}{2}n t}}{(n t)^{r}}\big(M_{nt}[g](x)-M_{t}[g](x)\big)\to 0\mbox{ as }t\to +\infty.
\end{equation}
Note that
\begin{align}\label{lem3.11}
&\Var_{\delta_{x}}\big[\frac{\e^{-\frac{\lambda_{1}}{2}n t}}{(n t)^{(r+1/2)}}\langle M_{(n-1)t}[g],X_{t}\rangle\big]\nonumber\\
=&\frac{\e^{-\lambda_{1}n t}}{(n t)^{1+2r}}\int_{0}^{t}T_{s}\big(\vartheta\big[T_{t-s}\big(M_{(n-1)t}[g]\big)\big]\big)(x)\mathrm{d}s\nonumber\\
=&\frac{\e^{-\lambda_{1}n t}}{(n t)^{1+2r}}\int_{0}^{t}T_{s}\big(\vartheta\big[M_{nt-s}[g]-M_{t-s}[g]\big]\big)(x)\mathrm{d}s\nonumber\\
=&\frac{1}{nt}\int_{0}^{t}\e^{-\lambda_{1}s}\big(1-\frac{s}{nt}\big)^{2 r}T_{s}\big(\vartheta\big[\frac{\e^{-\frac{\lambda_{1}}{2}(nt-s)}}{(nt-s)^{r}}M_{nt-s}[g]
-\e^{-\frac{\lambda_{1}}{2}(n-1)t}\big(\frac{t-s}{nt-s}\big)^{r}\cdot\frac{\e^{-\frac{\lambda_{1}}{2}(t-s)}}{(t-s)^{r}}M_{t-s}[g]\big]\big)(x)\mathrm{d}s\nonumber\\
=&\frac{1}{n}\int_{0}^{1}\e^{-\lambda_{1}t u}\big(1-\frac{u}{n}\big)^{2 r}T_{t u}\big(\vartheta\big[\frac{\e^{-\frac{\lambda_{1}}{2}(n-u)t}}{((n-u)t)^{r}}M_{(n-u)t}[g]
-\e^{-\frac{\lambda_{1}}{2}(n-1)t}\big(\frac{1-u}{n-u}\big)^{r}\cdot\frac{\e^{-\frac{\lambda_{1}}{2}(1-u)t}}{((1-u)t)^{r}}M_{(1-u)t}[g]\big]\big)(x)\mathrm{d}u\nonumber\\
\le&\frac{1}{n}\int_{0}^{1}\e^{-\lambda_{1}t u}T_{t u}\big(\vartheta\big[\frac{\e^{-\frac{\lambda_{1}}{2}(n-u)t}}{((n-u)t)^{r}}M_{(n-u)t}[g]
-\e^{-\frac{\lambda_{1}}{2}(n-1)t}\big(\frac{1-u}{n-u}\big)^{r}\cdot\frac{\e^{-\frac{\lambda_{1}}{2}(1-u)t}}{((1-u)t)^{r}}M_{(1-u)t}[g]\big]\big)(x)\mathrm{d}u\nonumber\\
=:&\frac{1}{n}\int_{0}^{1}\e^{-\lambda_{1}t u}T_{t u}\big(\vartheta[f^{(n)}_{t}(u,\cdot)]\big)(x)\mathrm{d}u.
\end{align}
Using the fact that $t^{-r}\e^{-\lambda_{1}t}M_{t}[g]\rightrightarrows \frac{2}{\lambda_{1}}g^{*}$ as $t\to+\infty$, one can easily show that for every $n>1$,
$\sup_{0<u<1,\,t\ge 1,\, x\in E}|f^{(n)}_{t}(u,x)|<+\infty$, and for every $u\in (0,1)$, $f^{(n)}_{t}(u,\cdot)\rightrightarrows 2g^{*}/\lambda_{1}$ as $t\to+\infty$.
Consequently, we get $\vartheta[f^{(n)}_{t}(u,\cdot)]\stackrel{b.p.}{\longrightarrow}\frac{4}{\lambda^{2}_{1}}\vartheta[g^{*}]$ as $t\to+\infty$ for every $n>1$ and $u\in (0,1)$. Using this, {(H2)} and the dominated convergence theorem, we can further show that as $t\to+\infty$ the integral in the right hand side of \eqref{lem3.11} converges to $\frac{4}{\lambda^{2}_{1}}\langle\vartheta[g^{*}],\widetilde{\varphi}\rangle \varphi(x)$.
Hence we get $\Var_{\delta_{x}}\big[\frac{\e^{-\frac{\lambda_{1}}{2}n t}}{(n t)^{(r+1/2)}}\langle M_{(n-1)t}[g],X_{t}\rangle\big]\to 0$ as $t\to+\infty$ and then $n\to+\infty$. This together with \eqref{lem3.1} yields that
$$\frac{\e^{-\frac{\lambda_{1}}{2}n t}}{(nt)^{r+1/2}}\langle M_{(n-1)t}[g],X_{t}\rangle \to 0\mbox{ in }L^{2}(\pp_{\delta_{x}})\mbox{ as }t\to+\infty\mbox{ and then }n\to+\infty,$$
which combined with \eqref{lem3.9} yields \eqref{lem3.0}.

\smallskip

\noindent Proof of \eqref{lem3.new2}: By the Markov property, we have $\pp_{\delta_{x}}\big[X_{nt}(f)|\mathcal{F}_{t}\big]= X_{t}(T_{(n-1)t}f)$.
We note that
\begin{align*}
\mathrm{Var}_{\delta_{x}}[ X_{t}(T_{(n-1)t}f)]
=&\int_{0}^{t}T_{t-s}\left(\vartheta[T_{(n-1)t+s}f]\right)(x){ \ud s}\\
=&\int_{(n-1)t}^{nt}T_{nt-u}\left(\vartheta[T_{u}f]\right)(x){ \ud u}\\
=&\int_{0}^{nt}T_{nt-u}\left(\vartheta[T_{u}f]\right)(x){ \ud u}-\int_{0}^{(n-1)t}T_{t}\circ T_{(n-1)t-u}\left(\vartheta[T_{u}f]\right)(x){ \ud u}\\
=&\mathrm{Var}_{\delta_{x}}[ X_{nt}(f)]-T_{t}\left(\mathrm{Var}_{\delta_{\cdot}}[ X_{(n-1)t}(f)]\right)(x).
\end{align*}
Thus,
\begin{align}\label{lem2case2.1}
\frac{\e^{-\lambda_{1}nt}}{(nt)^{1+2r}}\mathrm{Var}_{\delta_{x}}[ X_{t}(T_{(n-1)t}f)]
=&\frac{\e^{-\lambda_{1}nt}}{(nt)^{1+2r}}\mathrm{Var}_{\delta_{x}}[ X_{nt}(f)]\nonumber\\
&-\left(1-\frac{1}{n}\right)^{1+2r}\e^{-\lambda_{1} t}T_{t}\left(\frac{\e^{-\lambda_{1}(n-1)t}}{((n-1)t)^{1+2r}}\mathrm{Var}_{\delta_{\cdot}}[ X_{(n-1)t}(f)]\right)(x).
\end{align}
Recall from Lemma \ref{lem0}(ii) that $t^{-1-2r}\e^{-\lambda_{1} t}\mathrm{Var}_{\delta_{y}}[X_{t}(f)]\stackrel{b.p.}{\longrightarrow}(1+2r)^{-1}\langle\vartheta[f^{*}],\widetilde{\varphi}\rangle \varphi(y)$ as $t\to+\infty$. This combined with (H2) and the dominated convergence theorem yields that
\begin{align*}
\lim_{t\to+\infty}\e^{-\lambda_{1} t}T_{t}\left(\frac{\e^{-\lambda_{1}(n-1)t}}{((n-1)t)^{1+2r}}\mathrm{Var}_{\delta_{\cdot}}[ X_{(n-1)t}(f)]\right)(x)
=&\frac{\langle\vartheta[f^{*}],\widetilde{\varphi}\rangle}{1+2r}\varphi(x).\end{align*}
Hence we get from \eqref{lem2case2.1} that
\begin{equation}\nonumber
\lim_{t\to+\infty}\frac{\e^{-\lambda_{1}n t}}{(nt)^{1+2r}}\Var_{\delta_{x}}\big[ X_{t}(T_{(n-1)t}f)\big]=\left(1-\left(1-\frac{1}{n}\right)^{1+2r}\right)\frac{\langle\vartheta[f^{*}],\widetilde{\varphi}\rangle}{1+2r}\varphi(x).
\end{equation}
On the other hand, from the definition,
$$\frac{\e^{-\frac{\lambda_{1}}{2}n t}}{(nt)^{1/2+r}}\pp_{\delta_{x}}\left[ X_{t}(T_{(n-1)t}f)\right]=\frac{1}{(nt)^{1/2}}\cdot\frac{\e^{-\frac{\lambda_{1}}{2}n t}}{(nt)^{r}}T_{nt}f(x)\to 0\mbox{ as }t\to+\infty.$$
Thus we get
\begin{align*}
&\lim_{n\to+\infty}\lim_{t\to+\infty}\frac{\e^{-\lambda_{1}n t}}{(nt)^{1+2r}}\pp_{\delta_{x}}\big[ X_{t}(T_{(n-1)t}f)^{2}\big]\\=&\lim_{n\to+\infty}\lim_{t\to+\infty}\frac{\e^{-\lambda_{1}n t}}{(nt)^{1+2r}}\big(\pp_{\delta_{x}}\big[ X_{t}(T_{(n-1)t}f)\big]^{2}
+\Var_{\delta_{x}}\big[ X_{t}(T_{(n-1)t}f)\big]\big)=0.
\end{align*}
Therefore, \eqref{lem3.new2} follows.
\smallskip

\noindent Step 2. We shall show that
\begin{equation}\label{lem3.12}
\frac{\e^{-\frac{\lambda_{1}}{2}n t}}{(nt)^{r+1/2}}\big(Y_{n t}(g)-\pp_{\delta_{x}}\big[Y_{nt}(g)|\mathcal{F}_{t}\big],X_{nt}(f)-\pp_{\delta_{x}}\big[X_{nt}(f)|\mathcal{F}_{t}\big]\big)
\stackrel{d}{\to}\big(1-\frac{1}{n}\big)^{r+1/2}\sqrt{\frac{W^{\varphi}_{\infty}}{1+2r}}(N_{1}(g^{*}),N_{2}(f^{*}))
\end{equation}
as $t\to+\infty$.
The proof idea is similar to that of \eqref{lem3.1.0}. By the Markov property, we have
$$Y_{n t}(g)-\pp_{\delta_{x}}\big[Y_{nt}(g)|\mathcal{F}_{t}\big]=Y_{nt}(g)-Y_{t}(g)-X_{t}(M_{(n-1)t}[g])\mbox{ and }
X_{nt}(f)-\pp_{\delta_{x}}\big[X_{nt}(f)|\mathcal{F}_{t}\big]=X_{nt}(f)-X_{t}(T_{(n-1)t}f).$$
Moreover, given $\mathcal{F}_{t}$, $(Y_{nt}(g)-Y_{t}(g),X_{nt}(f))$ is equal in distribution to $((Y_{(n-1)t}(g),X_{(n-1)t}(f)),\pp_{X_{t}})$.
Thus, we have for $\theta_{1},\theta_{2}\in\R$,
\begin{align}\label{lem3.5.1}
&\pp_{\delta_{x}}\big[\exp\big\{\mathrm{i}(nt)^{-(r+1/2)}\e^{-\frac{\lambda_{1}}{2}n t}\big(\theta_{1}(Y_{n t}(g)-\pp_{\delta_{x}}\big[Y_{nt}(g)|\mathcal{F}_{t}\big])+\theta_{2}(X_{nt}(f)-\pp_{\delta_{x}}\big[X_{nt}(f)|\mathcal{F}_{t}\big])\big)\big\}\big]\nonumber\\
=&\pp_{\delta_{x}}\big[\exp\big\{-\mathrm{i}(nt)^{-(r+1/2)}\e^{-\frac{\lambda_{1}}{2}nt}\langle \theta_{1}M_{(n-1)t}[g]+\theta_{2}T_{(n-1)t}f,X_{t}\rangle\big\}
\nonumber\\
&\cdot\pp_{X_{t}}\big[\exp\big\{\mathrm{i}(nt)^{-(r+1/2)}\e^{-\frac{\lambda_{1}}{2}n t}\big(\theta_{1}Y_{(n-1)t}(g)+\theta_{2}X_{(n-1)t}(f)\big)\big\}\big]\big].
\end{align}
Let $\zeta_{(n-1)t}:=((n-1)t)^{-(r+1/2)}\e^{-\frac{\lambda_{1}}{2}(n-1)t}\big(\theta_{1}Y_{(n-1)t}(g)+\theta_{2}X_{(n-1)t}(f)\big)$. The regular branching property and condition (M2) imply that for every $\mu\in\mf$, $\zeta_{(n-1)t}$ under $\pp_{\mu}$ is an infinitely divisible random variable with finite mean and variance, and hence its characteristic function can be represented as
\begin{align*}
\pp_{\mu}\big[\e^{\mathrm{i}\theta\zeta_{(n-1)t}}\big]=&\exp\big\{\mathrm{i}\theta\langle\pp_{\delta_{\cdot}}[\zeta_{(n-1)t}],\mu\rangle-\frac{1}{2}\theta^{2}\langle
\Var_{\delta_{\cdot}}[\zeta_{(n-1)t}],\mu\rangle\\
&+\langle\int_{\R\setminus\{0\}}\big(\e^{\mathrm{i}\theta y}-1-\mathrm{i}\theta y+\frac{1}{2}\theta^{2}y^{2}\big)\pi_{(n-1)t}(\cdot,\mathrm{d}y),\mu\rangle
\big\}\quad\forall \theta\in\R,
\end{align*}
where $\pi_{(n-1)t}(z,\mathrm{d}y)$ is a kernel from $E$ to $\R\setminus\{0\}$ satisfying that
\begin{equation}\label{lem4.2.n1}
\int_{\R\setminus\{0\}}y^{2}\pi_{(n-1)t}(z,\mathrm{d}y)\le\Var_{\delta_{z}}[\zeta_{(n-1)t}]\quad\forall z\in E.
\end{equation}
Note that $\pp_{\delta_{y}}[\zeta_{(n-1)t}]=((n-1)t)^{-(r+1/2)}\e^{-\frac{\lambda_{1}}{2}(n-1)t}\big(\theta_{1}M_{(n-1)t}[g](y)+\theta_{2}T_{(n-1)t}f(y)\big)$. Then the right hand side of \eqref{lem3.5.1} can be represented by
\begin{align}\label{lem3.2}
&\pp_{\delta_{x}}\big[\exp\big\{-\mathrm{i}\big(1-\frac{1}{n}\big)^{r+1/2}\e^{-\frac{\lambda_{1}}{2}t}\langle \pp_{\delta_{\cdot}}[\zeta_{(n-1)t}],X_{t}\rangle\big\}\cdot
\pp_{X_{t}}\big[\exp\big\{\mathrm{i}\big(1-\frac{1}{n}\big)^{r+1/2}\e^{-\frac{\lambda_{1}}{2}t}\zeta_{(n-1)t}\big\}\big]\big]\nonumber\\
=&\pp_{\delta_{x}}\big[\exp\big\{-\frac{1}{2}\big(1-\frac{1}{n}\big)^{2r+1}\e^{-\lambda_{1}t}\langle \Var_{\delta_{\cdot}}\big[\zeta_{(n-1)t}\big],X_{t}\rangle
+\langle \mathcal{L}_{\pi_{(n-1)t}}\big[\big(1-\frac{1}{n}\big)^{r+1/2}\e^{-\frac{\lambda_{1}}{2}t}\big],X_{t}\rangle
\big\}\big],
\end{align}
where $\mathcal{L}_{\pi_{(n-1)t}}[\theta](z):=\int_{\R\setminus\{0\}}\big(\e^{\mathrm{i}\theta y}-1-i\theta y+\frac{1}{2}\theta^{2}y^{2}\big)\pi_{(n-1)t}(z,\mathrm{d}y)$ for $z\in E$ and $\theta\in \R$.
By Lemma \ref{lem0}(ii) and Lemma \ref{lem1}(ii), we have
\begin{align}\label{lem3.3}
\frac{\e^{-\lambda_{1}t}}{t^{1+2r}}\Var_{\delta_{x}}[\theta_{1}Y_{t}(g)+\theta_{2}X_{t}(f)]
=&\frac{\e^{-\lambda_{1}t}}{t^{1+2r}}\big(\theta^{2}_{1}\Var_{\delta_{x}}[Y_{t}(g)]
+2\theta_{1}\theta_{2}\mathrm{Cov}_{\delta_{x}}\big(Y_{t}(g),X_{t}(f)\big)+\theta^{2}_{2}\Var_{\delta_{x}}[X_{t}(f)]\big)\nonumber\\
\stackrel{b.p.}{\longrightarrow}&\frac{1}{1+2r}\varphi(x)\big(\frac{4}{\lambda^{2}_{1}}\theta^{2}_{1}\langle\vartheta[g^{*}],
\widetilde{\varphi}\rangle
+\frac{4}{\lambda_{1}}\theta_{1}\theta_{2}\langle\vartheta[f^{*},g^{*}],\widetilde{\varphi}\rangle+\theta^{2}_{2}\langle \vartheta[f^{*}],\widetilde{\varphi}\rangle\big),
\end{align}
as $t\to+\infty$. Thus, it follows by Lemma \ref{lem2.2} that
\begin{align}\label{lem3.5}
&\big(1-\frac{1}{n}\big)^{1+2r}\e^{-\lambda_{1}t}\langle\Var_{\delta_{\cdot}}[\zeta_{(n-1)t}],X_{t}\rangle
\nonumber\\
=&\big(1-\frac{1}{n}\big)^{1+2r}\e^{-\lambda_{1}t}\langle \frac{\e^{-\lambda_{1}(n-1)t}}{((n-1)t)^{1+2r}}\Var_{\delta_{\cdot}}[\theta_{1}Y_{(n-1)t}(g)+\theta_{2}X_{(n-1)t}(f)],X_{t}\rangle\nonumber\\
\to& \frac{(1-1/n)^{1+2r}}{1+2r}\big(\frac{4}{\lambda^{2}_{1}}\theta^{2}_{1}\langle\vartheta[g^{*}],\widetilde{\varphi}
+\frac{4}{\lambda_{1}}\theta_{1}\theta_{2}\langle\vartheta[f^{*},g^{*}],\widetilde{\varphi}\rangle+\theta^{2}_{2}\langle \vartheta[f^{*}],\widetilde{\varphi}\rangle\big)W^{\varphi}_{\infty}
\end{align}
in $L^{2}(\pp_{\delta_{x}})$ as $t\to+\infty$.

On the other hand, we claim that for every $n>1$,
\begin{equation}\label{lem3.4}
\e^{\lambda_{1}t}\mathcal{L}_{\pi_{(n-1)t}}\big[\big(1-\frac{1}{n}\big)^{r+1/2}\e^{-\frac{\lambda_{1}}{2}t}\big]\stackrel{b.p.}{\longrightarrow}0\mbox{ as }t\to+\infty,
\end{equation}
whose proof is deferred to the end of this step.
Thus it follows by Lemma \ref{lem2.2} that
\begin{equation*}
\langle \mathcal{L}_{\pi_{(n-1)t}}\big[\big(1-\frac{1}{n}\big)^{r+1/2}\e^{-\frac{\lambda_{1}}{2}t}\big],X_{t}\rangle \to 0\mbox{ in }L^{2}(\pp_{\delta_{x}})\mbox{ as }t\to+\infty.
\end{equation*}
Combining this with \eqref{lem3.5} and \eqref{lem3.2}, we get for every $\theta_{1},\theta_{2}\in \R$,
\begin{align*}
&\lim_{t\to+\infty}\pp_{\delta_{x}}\big[\exp\big\{\mathrm{i}(nt)^{-(r+1/2)}\e^{-\frac{\lambda_{1}}{2}n t}\big(\theta_{1}(Y_{n t}(g)-\pp_{\delta_{x}}\big[Y_{nt}(g)|\mathcal{F}_{t}\big])+\theta_{2}(X_{nt}(f)-\pp_{\delta_{x}}\big[X_{nt}(f)|\mathcal{F}_{t}\big])\big)\big\}\big]\\
=&\pp_{\delta_{x}}\big[\exp\big\{-\frac{1}{2(1+2r)}\big(1-\frac{1}{n}\big)^{1+2r}W^{\varphi}_{\infty}\big(\frac{4}{\lambda^{2}_{1}}\theta^{2}_{1}\langle\vartheta[g^{*}],\widetilde{\varphi}
\rangle
+\frac{4}{\lambda_{1}}\theta_{1}\theta_{2}\langle\vartheta[f^{*},g^{*}],\widetilde{\varphi}\rangle+\theta^{2}_{2}\langle\vartheta[f^{*}],\widetilde{\varphi}\rangle\big)\big\}\big]\\
=&\pp_{\delta_{x}}\big[\exp\big\{\mathrm{i}\big(1-\frac{1}{n}\big)^{r+1/2}\sqrt{\frac{W^{\varphi}_{\infty}}{1+2r}}(\theta_{1}N_{1}(g^{*})+\theta_{2}N_{2}(f^{*}))\big\}\big].
\end{align*}
Hence we obtain \eqref{lem3.12}.

It remains to prove \eqref{lem3.4}. Note that by \eqref{ineq2} for every $n>1$ and $x\in E$,
\begin{align}\label{lem3.5.2}
&\Big|\mathcal{L}_{\pi_{(n-1)t}}\big[\big(1-\frac{1}{n}\big)^{r+1/2}\e^{-\frac{\lambda_{1}}{2}t}\big](x)\Big|\\
\le&\big(1-\frac{1}{n}\big)^{1+2r}\e^{-\lambda_{1}t}\int_{\R\setminus\{0\}}y^{2}\big(1\wedge\frac{\big(1-1/n\big)^{r+1/2}\e^{-\frac{\lambda_{1}}{2}t}|y|}{6}\big)\pi_{(n-1)t}(x,\mathrm{d}y).
\end{align}
By \eqref{lem3.3}, $\Var_{\delta_{x}}[\zeta_{(n-1)t}]$ is uniformly bounded in $(x,t)$ for $t$ sufficiently large, and hence by \eqref{lem4.2.n1} and \eqref{lem3.5.2}, $\Big|\mathcal{L}_{\pi_{(n-1)t}}\big[\big(1-\frac{1}{n}\big)^{r+1/2}\e^{-\frac{\lambda_{1}}{2}t}\big](x)\Big|$ is uniformly bounded in $(x,t)$ for $t$ sufficiently large.
Hence, to show \eqref{lem3.4} we only need to show that for every $n>1$,
\begin{equation}\label{lem3.6}
\lim_{t\to+\infty}\e^{\lambda_{1}t}\mathcal{L}_{\pi_{(n-1)t}}\big[\big(1-\frac{1}{n}\big)^{r+1/2}\e^{-\frac{\lambda_{1}}{2}t}\big](x)=0\quad\forall x\in E.
\end{equation}
Let $\epsilon>0$. We have by \eqref{lem4.2.n1} and \eqref{lem3.5.2}
\begin{align*}
&\Big|\mathcal{L}_{\pi_{(n-1)t}}\big[\big(1-\frac{1}{n}\big)^{r+1/2}\e^{-\frac{\lambda_{1}}{2}t}\big](x)\Big|\\
\le&\big(1-\frac{1}{n}\big)^{1+2r}\e^{-\lambda_{1}t}\Big(\int_{\big(1-n^{-1}\big)^{r+1/2}\e^{-\frac{\lambda_{1}}{2}t}|y|\le \epsilon}\frac{\epsilon}{6}y^{2}\pi_{(n-1)t}(x,\mathrm{d}y)
+\int_{\big(1-n^{-1}\big)^{r+1/2}\e^{-\frac{\lambda_{1}}{2}t}|y|>\epsilon}y^{2}\pi_{(n-1)t}(x,\mathrm{d}y)\Big)\\
\le&\big(1-\frac{1}{n}\big)^{1+2r}\e^{-\lambda_{1}t}\Big(\frac{\epsilon}{6}\Var_{\delta_{x}}[\zeta_{(n-1)t}]+\int_{\big(1-n^{-1}\big)^{r+1/2}\e^{-\frac{\lambda_{1}}{2}t}|y|>\epsilon}y^{2}\pi_{(n-1)t}(x,\mathrm{d}y)
\big).
\end{align*}
Consequently,
\begin{align}
\e^{\lambda_{1}t}\Big|\mathcal{L}_{\pi_{(n-1)t}}\big[\big(1-\frac{1}{n}\big)^{r+1/2}\e^{-\frac{\lambda_{1}}{2}t}\big](x)\Big|
\le&\big(1-\frac{1}{n}\big)^{1+2r}\Big(\frac{\epsilon}{6}\frac{\e^{-\lambda_{1}(n-1)t}}{((n-1)t)^{1+2r}}\Var_{\delta_{x}}\big[\theta_{1}Y_{(n-1)t}(g)+\theta_{2}X_{(n-1)t}(f)\big]\nonumber\\
&+\int_{|y|>\epsilon\big(1-n^{-1}\big)^{-(r+1/2)}\e^{\frac{\lambda_{1}}{2}t}}y^{2}\pi_{(n-1)t}(x,\mathrm{d}y)\Big).\label{lem4.2.n2}
\end{align}
In view of \eqref{lem3.3}, if we can show that for every $n>1$ and $x\in E$,
\begin{equation}\label{lem3.7}
\int_{|y|>\epsilon\big(1-n^{-1}\big)^{-(r+1/2)}\e^{\frac{\lambda_{1}}{2}t}}y^{2}\pi_{(n-1)t}(x,\mathrm{d}y)\to 0\mbox{ as }t\to+\infty,
\end{equation}
then we get \eqref{lem3.6} by letting first $t\to+\infty$ and then $\epsilon\to 0$ in \eqref{lem4.2.n2}.
Note that
\begin{equation}\label{lem3.8}
\int_{|y|>\epsilon\big(1-n^{-1}\big)^{-(r+1/2)}\e^{\frac{\lambda_{1}}{2}t}}y^{2}\pi_{(n-1)t}(x,\mathrm{d}y)\le \frac{1}{\epsilon^{2}}\big(1-\frac{1}{n}\big)^{1+2r}
\e^{-\lambda_{1}t}\int_{\R\setminus\{0\}}y^{4}\pi_{(n-1)t}(x,\mathrm{d}y).
\end{equation}
Under (M4), $(\zeta_{(n-1)t},\pp_{\delta_{x}})$ is an infinitely divisible random variable with a finite fourth moment, and its L\'{e}vy measure is given by $\pi_{(n-1)t}(x,\mathrm{d}y)$. We have $\pp_{\delta_{x}}\big[\big(\zeta_{(n-1)t}-\pp_{\delta_{x}}[\zeta_{(n-1)t}]\big)^{4}\big]=\int_{\R\setminus\{0\}}y^{4}\pi_{(n-1)t}(x,\mathrm{d}y)+3\Var_{\delta_{x}}[\zeta_{(n-1)t}]^{2}$.
Hence, for $x\in E$ and $t>0$,
\begin{align*}
\int_{\R\setminus\{0\}}y^{4}\pi_{(n-1)t}(x,\mathrm{d}y)\le&\pp_{\delta_{x}}\big[\big(\zeta_{(n-1)t}-\pp_{\delta_{x}}[\zeta_{(n-1)t}]\big)^{4}\big]\\
\lesssim&\pp_{\delta_{x}}\big[\zeta^{4}_{(n-1)t}\big]\\
\lesssim&\frac{\e^{-2\lambda_{1}(n-1)t}}{((n-1)t)^{2+4r}}\Big(\theta^{4}_{2}\pp_{\delta_{x}}[X_{(n-1)t}(f)^{4}]+\theta^{4}_{1}\pp_{\delta_{x}}[Y_{(n-1)t}(g)^{4}]\Big).
\end{align*}
It follows by Lemma \ref{lemA.1} that the right hand side is bounded from above for $t$ sufficiently large. This implies that for every $x\in E$ and $n>1$,
$\int_{\R\setminus\{0\}}y^{4}\pi_{(n-1)t}(x,\mathrm{d}y)$ is bounded in $t$ when $t$ is large enough. This together with \eqref{lem3.8} yields \eqref{lem3.7}.

\smallskip

\noindent Step 3. We have
\begin{align*}
&\mathbf{d}\big(\frac{\e^{-\frac{\lambda_{1}}{2}nt}}{(nt)^{r+1/2}}(Y_{nt}(g),X_{nt}(f)),\sqrt{\frac{W^{\varphi}_{\infty}}{1+2r}}(N_{1}(g^{*}),N_{2}(f^{*}))\big)\\
\le&\mathbf{d}\big(\frac{\e^{-\frac{\lambda_{1}}{2}nt}}{(nt)^{r+1/2}}(Y_{nt}(g),X_{nt}(f)),\frac{\e^{-\frac{\lambda_{1}}{2}nt}}{(nt)^{r+1/2}}(Y_{nt}(g)-\pp_{\delta_{x}}[Y_{nt}(g)|\mathcal{F}_{t}],
X_{nt}(f)-\pp_{\delta_{x}}[X_{nt}(f)|\mathcal{F}_{t}])\big)\\
&+\mathbf{d}\big(\frac{\e^{-\frac{\lambda_{1}}{2}nt}}{(nt)^{r+1/2}}(Y_{nt}(g)-\pp_{\delta_{x}}\big[Y_{nt}(g)|\mathcal{F}_{t}\big],X_{nt}(f)-\pp_{\delta_{x}}\big[X_{nt}(f)|\mathcal{F}_{t}\big])
,\big(1-\frac{1}{n}\big)^{r+1/2}\sqrt{\frac{W^{\varphi}_{\infty}}{1+2r}}(N_{1}(g^{*}),N_{2}(f^{*}))\big)\\
&+\mathbf{d}\big(\big(1-\frac{1}{n}\big)^{r+1/2}\sqrt{\frac{W^{\varphi}_{\infty}}{1+2r}}(N_{1}(g^{*}),N_{2}(f^{*})),\sqrt{\frac{W^{\varphi}_{\infty}}{1+2r}}(N_{1}(g^{*}),N_{2}(f^{*}))\big).
\end{align*}
Clearly, the final term converges to $0$ as $n\to+\infty$.
It has been proved in Step 2 that the second term converges to $0$ as $t\to +\infty$ and then $n\to+\infty$.
By \eqref{ineq1}, the first term in the right hand side is bounded by
$$\|\frac{\e^{-\frac{\lambda_{1}}{2}nt}}{(nt)^{r+1/2}}\pp_{\delta_{x}}\big[Y_{nt}(g)|\mathcal{F}_{t}\big]\|_{L^{2}(\pp_{\delta_{x}})}+
\|\frac{\e^{-\frac{\lambda_{1}}{2}nt}}{(nt)^{r+1/2}}\pp_{\delta_{x}}\big[X_{nt}(f)|\mathcal{F}_{t}\big]\|_{L^{2}(\pp_{\delta_{x}})},$$
and hence by \eqref{lem3.0} and \eqref{lem3.new2} it converges to $0$ as first $t\to+\infty$ and then $n\to+\infty$. Therefore, we obtain that
$$\lim_{n\to+\infty}\lim_{t\to+\infty}\mathbf{d}\big(\frac{\e^{-\frac{\lambda_{1}}{2}nt}}{(nt)^{r+1/2}}(Y_{nt}(g),X_{nt}(f)),\sqrt{\frac{W^{\varphi}_{\infty}}{1+2r}}(N_{1}(g^{*}),N_{2}(f^{*}))\big)=0,$$ which yields the result of this lemma.\qed

\begin{proposition}\label{prop:fddcri}
Suppose the assumptions of Lemma \ref{lemA.1} hold. Then for $f\in \mathcal{D}^{0}(\lambda_{1}/2,r)$,
$$\big(t^{-(r+1/2)}\e^{-\frac{\lambda_{1}}{2}(t+s)}Y_{t+s}(f)\big)_{s\ge 0}\stackrel{f.d.d.}{\longrightarrow}\frac{2}{\lambda_{1}}\sqrt{\frac{\langle\vartheta[f^{*}],\widetilde{\varphi}\rangle}{1+2r}\,W^{\varphi}_{\infty}}C\mbox{ as }t\to+\infty,$$
where $C\equiv N$ is a constant process with $N$ a standard normal random variable independent of $W^{\varphi}_{\infty}$.
\end{proposition}

\proof Fix an arbitrary $k\in\mathbb{N}$ and $0=s_{1}\le s_{2}<\cdots\le s_{k}<+\infty$. Let $a_{1}=0$ and $a_{j}:=s_{j}-s_{j-1}$ for $2\le j\le k$.
Let $\vec{\theta}=(\theta_{1},\cdots,\theta_{k})\in\R^{k}$ and $K^{1}_{t}(\vec{\theta})$ be the characteristic function of the random vector $\big(t^{-(r+1/2)}\e^{-\frac{\lambda_{1}}{2}(t+s_{j})}Y_{t+s_{j}}(f)\big)_{1\le j\le k}$.
By the Markov property, given $\mathcal{F}_{t+s_{k-1}}$, $Y_{t+s_{k}}(f)-Y_{t+s_{k-1}}(f)$ is equal in distribution to $(Y_{a_{k}}(f),\pp_{X_{t+s_{k-1}}})$.
Thus we have
\begin{align*}
K^{1}_{t}(\vec{\theta})=&\pp_{\delta_{x}}\big[\exp\big\{\sum_{j=1}^{k}\mathrm{i}\theta_{j}t^{-(r+1/2)}\e^{-\frac{\lambda_{1}}{2}(t+s_{j})}Y_{t+s_{j}}(f)\big\}\big]\\
=&\pp_{\delta_{x}}\big[\exp\big\{\sum_{j=1}^{k-1}\mathrm{i}\theta_{j}t^{-(r+1/2)}\e^{-\frac{\lambda_{1}}{2}(t+s_{j})}Y_{t+s_{j}}(f)+
\mathrm{i}\theta_{k}t^{-(r+1/2)}\e^{-\frac{\lambda_{1}}{2}(t+s_{k})}Y_{t+s_{k-1}}(f)\big\}\\
&\cdot\pp_{X_{t+s_{k-1}}}\big[\exp\big\{
\mathrm{i}\theta_{k}t^{-(r+1/2)}\e^{-\frac{\lambda_{1}}{2}(t+s_{k})}Y_{a_{k}}(f)\big\}\big]\big].
\end{align*}
The regular branching property and condition (M2) imply that for every $\mu\in\mf$, $\zeta_{a_{k}}:=Y_{a_{k}}(f)$ under $\pp_{\mu}$ is an infinitely divisible random variable with finite mean and variance. Denote by $\pi_{a_{k}}(z,\mathrm{d}y)$ the corresponding L\'{e}vy kernel from $E$ to $\R\setminus\{0\}$. Applying a computation similar to the one leading to \eqref{lem3.1.4}, we get that
\begin{align}\label{prop2.1}
K^{1}_{t}(\vec{\theta})
=&\pp_{\delta_{x}}\big[\exp\big\{\sum_{j=1}^{k-2}\mathrm{i}\theta_{j}t^{-(r+1/2)}\e^{-\frac{\lambda_{1}}{2}(t+s_{j})}Y_{t+s_{j}}(f)+
\mathrm{i}t^{-(r+1/2)}\e^{-\frac{\lambda_{1}}{2}t}\big(\theta_{k-1}\e^{-\frac{\lambda_{1}}{2}s_{k-1}}+\theta_{k}\e^{-\frac{\lambda_{1}}{2}s_{k}}\big)Y_{t+s_{k-1}}(f)\nonumber\\
&+\langle\mathrm{i}\theta_{k}t^{-(r+1/2)}\e^{-\frac{\lambda_{1}}{2}(t+s_{k})}M_{a_{k}}[f],X_{t+s_{k-1}}\rangle\big\}\nonumber\\
&\cdot\exp\big\{\langle -\frac{1}{2}\theta^{2}_{k}t^{-(1+2r)}\e^{-\lambda_{1}(t+s_{k})}\Var_{\delta_{\cdot}}[Y_{a_{k}}(f)]
-\mathcal{L}_{\pi_{a_{k}}}\big[\theta_{k}t^{-(r+1/2)}\e^{-\frac{\lambda_{1}}{2}(t+s_{k})}f\big],X_{t+s_{k-1}}\rangle\big\}\big],
\end{align}
where $\mathcal{L}_{\pi_{a_{k}}}[\theta](z):=\int_{\R\setminus\{0\}}\big(\e^{\mathrm{i}\theta y}-1-\mathrm{i}\theta y+\frac{1}{2}\theta^{2}y^{2}\big)\pi_{a_{k}}(z,\mathrm{d}y)$ for $z\in E$ and $\theta\in\R$.
Similarly to \eqref{lem3.1.8}, we can show by \eqref{ineq2} that $\e^{\lambda_{1}t}\mathcal{L}_{\pi_{a_{k}}}\big[\theta_{k}t^{-(r+1/2)}\e^{-\frac{\lambda_{1}}{2}(t+s_{k})}f\big]\stackrel{b.p.}{\longrightarrow}0$ as $t\to+\infty$. Hence, by Lemma \ref{lem2.2}(i),
\begin{align}\label{prop3.6.1}
&t^{-(1+2r)}\e^{-\lambda_{1}(t+s_{k})}\langle \Var_{\delta_{\cdot}}[Y_{a_{k}}(f)],X_{t+s_{k-1}}\rangle\to 0\mbox{ in }L^{2}(\pp_{\delta_{x}})\mbox{ as }t\to+\infty,\nonumber\\
&\langle \mathcal{L}_{\pi_{a_{k}}}\big[\theta_{k}t^{-(r+1/2)}\e^{-\frac{\lambda_{1}}{2}(t+s_{k})}f\big],X_{t+s_{k-1}}\rangle\to 0\mbox{ in }L^{2}(\pp_{\delta_{x}})\mbox{ as }t\to+\infty.
\end{align}
It follows that
$$\lim_{t\to+\infty}\pp_{\delta_{x}}\big[\big|\exp\big\{\langle -\frac{1}{2}\theta^{2}_{k}t^{-(1+2r)}\e^{-\lambda_{1}(t+s_{k})}\Var_{\delta_{\cdot}}[Y_{a_{k}}(f)]
-\mathcal{L}_{\pi_{a_{k}}}\big[\theta_{k}t^{-(r+1/2)}\e^{-\frac{\lambda_{1}}{2}(t+s_{k})}f\big],X_{t+s_{k-1}}\rangle\big\}-1\big|\big]=0.$$
Combining this with \eqref{prop2.1}, we obtain that
\begin{equation}\label{prop3.6.2}
K^{1}_{t}(\vec{\theta})-K^{2}_{t}(\vec{\theta})\to 0\mbox{ as }t\to+\infty,
\end{equation}
where
\begin{align*}
K^{2}_{t}(\vec{\theta})=&\pp_{\delta_{x}}\big[\exp\big\{\sum_{j=1}^{k-2}\mathrm{i}\theta_{j}t^{-(r+1/2)}\e^{-\frac{\lambda_{1}}{2}(t+s_{j})}Y_{t+s_{j}}(f)+\mathrm{i}t^{-(r+1/2)}\e^{-\frac{\lambda_{1}}{2}t}\big(\e^{-\frac{\lambda_{1}}{2}s_{k-1}}\theta_{k-1}+\e^{-\frac{\lambda_{1}}{2}s_{k}}\theta_{k}\big)Y_{t+s_{k-1}}(f)\\
&+\mathrm{i}\theta_{k}t^{-(r+1/2)}\e^{-\frac{\lambda_{1}}{2}(t+s_{k})}\langle M_{a_{k}}[f],X_{t+s_{k-1}}\rangle\big\}\big].
\end{align*}
By the Markov property, given $\mathcal{F}_{t+s_{k-2}}$, $(Y_{t+s_{k-1}}(f)-Y_{t+s_{k-2}}(f),\langle M_{a_{k}}[f],X_{t+s_{k-1}}\rangle)$
is equal in distribution to $((Y_{a_{k-1}}(f),\langle M_{a_{k}}[f],X_{a_{k-1}}\rangle),\pp_{X_{t+s_{k-2}}})$. Thus,
we have
\begin{align}\label{prop2.2}
K^{2}_{t}(\vec{\theta})
=&\pp_{\delta_{x}}\big[\exp\big\{\sum_{j=1}^{k-3}\mathrm{i}\theta_{j}t^{-(r+1/2)}\e^{-\frac{\lambda_{1}}{2}(t+s_{j})}Y_{t+s_{j}}(f)\nonumber\\
&\quad+\mathrm{i}t^{-(r+1/2)}\e^{-\frac{\lambda_{1}}{2}t}\big(\e^{-\frac{\lambda_{1}}{2}s_{k-2}}\theta_{k-2}+\e^{-\frac{\lambda_{1}}{2}s_{k-1}}\theta_{k-1}+\e^{-\frac{\lambda_{1}}{2}s_{k}}\theta_{k}\big)
Y_{t+s_{k-2}}(f)\big\}\nonumber\\
&\cdot\pp_{X_{t+s_{k-2}}}\big[\exp\big\{\mathrm{i}t^{-(r+1/2)}\e^{-\frac{\lambda_{1}}{2}(t+s_{k-1})}\big(\theta_{k-1}+\theta_{k}\e^{-\frac{\lambda_{1}}{2}a_{k}}\big)Y_{a_{k-1}}(f)\nonumber\\
&\quad+\mathrm{i}\theta_{k}t^{-(r+1/2)}\e^{-\frac{\lambda_{1}}{2}(t+s_{k})}\langle M_{a_{k}}[f],X_{a_{k-1}}\rangle\big\}\big].
\end{align}
We note that under $\pp_{\mu}$, $\zeta_{a_{k-1}}:=(\theta_{k-1}+\theta_{k}\e^{-\frac{\lambda_{1}}{2}a_{k}})Y_{a_{k-1}}(f)+\theta_{k}\e^{-\frac{\lambda_{1}}{2}a_{k}}\langle M_{a_{k}}[f],X_{a_{k-1}}\rangle$ is an infinitely divisible random variable with finite mean and variance. Let $\pi_{a_{k-1}}(x,\mathrm{d}y)$ be the corresponding L\'{e}vy kernel from $E$ to $\R\setminus\{0\}$.
Again, applying a computation similar to that leading to \eqref{lem3.1.4}, we get
\begin{align*}
K^{2}_{t}(\vec{\theta})
=&\pp_{\delta_{x}}\big[\exp\big\{\sum_{j=1}^{k-3}\mathrm{i}\theta_{j}t^{-(r+1/2)}\e^{-\frac{\lambda_{1}}{2}(t+s_{j})}Y_{t+s_{j}}(f)\\
&+\mathrm{i}t^{-(r+1/2)}\e^{-\frac{\lambda_{1}}{2}t}\big(\e^{-\frac{\lambda_{1}}{2}s_{k-2}}\theta_{k-2}+\e^{-\frac{\lambda_{1}}{2}s_{k-1}}\theta_{k-1}+\e^{-\frac{\lambda_{1}}{2}s_{k}}\theta_{k}\big)
Y_{t+s_{k-2}}(f)\\
&+\mathrm{i}t^{-(r+1/2)}\e^{-\frac{\lambda_{1}}{2}(t+s_{k-1})}\langle \pp_{\delta_{\cdot}}\big[\zeta_{a_{k-1}}\big],X_{t+s_{k-2}}\rangle\big\}\\
&\cdot\exp\big\{-\frac{1}{2}t^{-(1+2r)}\e^{-\lambda_{1}(t+s_{k-1})}\langle \Var_{\delta_{\cdot}}[\zeta_{a_{k-1}}],X_{t+s_{k-2}}\rangle+\langle \mathcal{L}_{\pi_{a_{k-1}}}[t^{-(r+1/2)}\e^{-\frac{\lambda_{1}}{2}(t+s_{k-1})}],X_{t+s_{k-2}}\rangle\big\}\big],
\end{align*}
where $\mathcal{L}_{\pi_{a_{k-1}}}[\theta](z):=\int_{\R\setminus\{0\}}\big(\e^{\mathrm{i}\theta y}-1-\mathrm{i}\theta y+\frac{1}{2}\theta^{2}y^{2}\big)\pi_{a_{k-1}}(z,\mathrm{d}y)$ for $z\in E$ and $\theta\in \R$.
Moreover, similarly to \eqref{prop3.6.1}, we can prove that
\begin{align*}
&t^{-(1+2r)}\e^{-\lambda_{1}(t+s_{k-1})}\langle \Var_{\delta_{\cdot}}[\zeta_{a_{k-1}}],X_{t+s_{k-2}}\rangle \to 0\mbox{ in }L^{2}(\pp_{\delta_{x}})\mbox{as }t\to+\infty,\\
&\langle \mathcal{L}_{\pi_{a_{k-1}}}[t^{-(r+1/2)}\e^{-\frac{\lambda_{1}}{2}(t+s_{k-1})}],X_{t+s_{k-2}}\rangle\to 0\mbox{ in }L^{2}(\pp_{\delta_{x}})\mbox{as }t\to+\infty.
\end{align*}
Thus, it follows that
$$K^{2}_{t}(\vec{\theta})-K^{3}_{t}(\vec{\theta})\to 0\mbox{ as }t\to+\infty,$$
where
\begin{align*}
K^{3}_{t}(\vec{\theta})
=&\pp_{\delta_{x}}\big[\exp\big\{\sum_{j=1}^{k-3}\mathrm{i}\theta_{j}t^{-(r+1/2)}\e^{-\frac{\lambda_{1}}{2}(t+s_{j})}Y_{t+s_{j}}(f)\\
&+\mathrm{i}t^{-(r+1/2)}\e^{-\frac{\lambda_{1}}{2}t}\big(\e^{-\frac{\lambda_{1}}{2}s_{k-2}}\theta_{k-2}+\e^{-\frac{\lambda_{1}}{2}s_{k-1}}\theta_{k-1}+\e^{-\frac{\lambda_{1}}{2}s_{k}}\theta_{k}\big)
Y_{t+s_{k-2}}(f)\\
&+\mathrm{i}t^{-(r+1/2)}\e^{-\frac{\lambda_{1}}{2}(t+s_{k-1})}\langle\theta_{k-1} M_{s_{k-1}-s_{k-2}}[f]+\theta_{k}\e^{-\frac{\lambda_{1}}{2}a_{k}}M_{s_{k}-s_{k-2}}[f],X_{t+s_{k-2}}\rangle\big\}\big]
\end{align*}
In the above computation we use the fact that $\pp_{\delta_{y}}[\zeta_{a_{k-1}}]=\theta_{k-1}M_{s_{k-1}-s_{k-2}}[f](y)+\theta_{k}\e^{-\frac{\lambda_{1}}{2}a_{k}}M_{s_{k}-s_{k-2}}[f](y)$.
Consequently, by \eqref{prop3.6.2} we have
$$K^{1}_{t}(\vec{\theta})-K^{3}_{t}(\vec{\theta})\to 0\mbox{ as }t\to+\infty.$$

Repeating the above procedure $k-2$ more times, we obtain $K^{j}_{t}(\vec{\theta})-K^{j+1}_{t}(\vec{\theta})\to 0$ as $t\to+\infty$ for each $j=1,\cdots,k-2$. Summing up these differences yields
\begin{equation}\label{prop2.3}
K^{1}_{t}(\vec{\theta})-K^{k-1}_{t}(\vec{\theta})\to 0\mbox{ as }t\to+\infty,
\end{equation}
where
\begin{align*}
K^{k-1}_{t}(\vec{\theta})=&\pp_{\delta_{x}}\big[\exp\big\{\mathrm{i}t^{-(r+1/2)}\e^{-\frac{\lambda_{1}}{2}t}\big(
\sum_{j=1}^{k}\e^{-\frac{\lambda_{1}}{2}s_{j}}\theta_{j}\big)Y_{t}(f)+\mathrm{i}t^{-(r+1/2)}\e^{-\frac{\lambda_{1}}{2}t}\langle\sum_{j=2}^{k}\e^{-\frac{\lambda_{1}}{2}s_{j}}\theta_{j}
M_{s_{j}}[f],X_{t}\rangle
\big\}\big].
\end{align*}
Since $t^{-r}\e^{-\frac{\lambda_{1}}{2}t}M_{t}[f]\rightrightarrows \frac{2}{\lambda_{1}}f^{*}$ by Lemma \ref{lem1}(ii), we have for $j\in \{2,\cdots,k\}$,
$$\frac{\e^{-\frac{\lambda_{1}}{2}t}}{t^{r}}T_{t}\big(\e^{-\frac{\lambda_{1}}{2}s_{j}}M_{s_{j}}[f]\big)=\frac{\e^{-\frac{\lambda_{1}}{2}(t+s_{j})}}{t^{r}}\big(M_{t+s_{j}}[f]-M_{t}[f]\big)
\rightrightarrows \frac{2}{\lambda_{1}}\big(1-\e^{-\frac{\lambda_{1}}{2}s_{j}}\big)f^{*}\mbox{ as }t\to+\infty,$$
which in turn yields that $\sum_{j=2}^{k}\theta_{j}\e^{-\frac{\lambda_{1}}{2}s_{j}}M_{s_{j}}[f]\in\mathcal{D}^{0}(\lambda_{1}/2,r)$ with limiting function
$\frac{2}{\lambda_{1}}f^{*}\sum_{j=2}^{k}\theta_{j}\big(1-\e^{-\frac{\lambda_{1}}{2}s_{j}}\big)$.
Thus it follows by \eqref{prop2.3} and Lemma \ref{lem3} that
\begin{align*}
\lim_{t\to+\infty}K^{1}_{t}(\vec{\theta})=&\lim_{t\to+\infty}K^{k-1}_{t}(\vec{\theta})\\
=&\pp_{\delta_{x}}\big[\exp\big\{\mathrm{i}\sqrt{\frac{W^{\varphi}_{\infty}}{1+2r}}\big(N_{1}\big(f^{*}\sum_{j=1}^{k}\e^{-\frac{\lambda_{1}}{2}s_{j}}\theta_{j}\big)+
N_{2}\big(\frac{2}{\lambda_{1}}f^{*}\sum_{j=2}^{k}\theta_{j}\big(1-\e^{-\frac{\lambda_{1}}{2}s_{j}}\big)\big)\big)\big\}\big]\\
=&\pp_{\delta_{x}}\Big[\exp\big\{-\frac{1}{2(1+2r)}W^{\varphi}_{\infty}\big(\frac{4}{\lambda^{2}_{1}}\langle\vartheta\big[f^{*}\sum_{j=1}^{k}\e^{-\frac{\lambda_{1}}{2}s_{j}}\theta_{j}\big],\widetilde{\varphi}\rangle\\
&+\frac{4}{\lambda_{1}}\langle\vartheta\big[f^{*}\sum_{j=1}^{k}\e^{-\frac{\lambda_{1}}{2}s_{j}}\theta_{j},\frac{2}{\lambda_{1}}f^{*}\sum_{j=2}^{k}\theta_{j}(1-\e^{-\frac{\lambda_{1}}{2}s_{j}})\big],\widetilde{\varphi}\rangle\\
&+\langle\vartheta\big[\frac{2}{\lambda_{1}}f^{*}\sum_{j=2}^{k}\theta_{j}(1-\e^{-\frac{\lambda_{1}}{2}s_{j}})\big],\widetilde{\varphi}\rangle\big)\big\}\Big]\\
=&\pp_{\delta_{x}}\Big[\exp\big\{-\frac{1}{2}\cdot\frac{4}{(1+2r)\lambda^{2}_{1}}\langle \vartheta[f^{*}],\widetilde{\varphi}\rangle W^{\varphi}_{\infty}\big(\theta_{1}+\cdots+\theta_{k}\big)^{2}\big\}\Big].
\end{align*}
Note that the final term  is equal to the characteristic function of $\frac{2}{\lambda_{1}}\sqrt{\frac{\langle\vartheta[f^{*}],\widetilde{\varphi}\rangle}{1+2r}W^{\varphi}_{\infty}}(C_{s_{j}})_{1\le j\le k}$.
Hence we complete the proof.\qed

\begin{proposition}\label{prop:fddsub}
Suppose $(\lambda,r)\in (\lambda_{1}/2,\lambda_{1})\times [0,+\infty)\cup\{(\lambda_{1},0)\}$ and $f\in\mathcal{D}^{0}(\lambda,r)$ with limiting function $f^{*}$. Set $W^{f^{*}}_{t}:=\e^{-\lambda t}\langle f^{*},X_{t}\rangle$. Then the following hold.
\begin{description}
\item[(i)] $\{W^{f^{*}}_{t}:\ t\ge 0\}$ is an $L^{2}(\pp_{\delta_{x}})$-bounded martingale, and hence converges $\pp_{\delta_{x}}$-a.s. and in $L^{2}(\pp_{\delta_{x}})$ to a finite limit $W^{f^{*}}_{\infty}$. Furthermore, $\Var_{\delta_{x}}[W^{f^{*}}_{\infty}]=\varpi^{2}_{f^{*}}(x)$.
\item[(ii)] $t^{-r}\e^{-\lambda t}X_{t}(f)\to W^{f^{*}}_{\infty}\mbox{ in }L^{2}(\pp_{\delta_{x}})\mbox{ as }t\to+\infty.$
\item[(iii)] $t^{-r}\e^{-\lambda t}Y_{t}(f)\to \frac{1}{\lambda}W^{f^{*}}_{\infty}\mbox{ in }L^{2}(\pp_{\delta_{x}})\mbox{ as }t\to+\infty.$
\item[(iv)] Under $\pp_{\delta_{x}}$, $\big((t+s)^{-r}\e^{(\lambda_{1}-\lambda)(t+s)}Z^{(t)}_{s}(f)\big)_{s\ge 0}\stackrel{f.d.d.}{\longrightarrow}\frac{1}{\lambda}W^{f^{*}}_{\infty}$ as $t\to+\infty$.
\end{description}
\end{proposition}

\proof (i) Since $\langle f,\widetilde{\varphi}\rangle=0$ and $t^{-r}\e^{-\lambda t}T_{t}f\rightrightarrows f^{*}$ by the definition, we have by the dominated convergence theorem,
$$\langle f^{*},\widetilde{\varphi}\rangle=\lim_{s\to+\infty}s^{-r}\e^{-\lambda s}\langle T_{s}f,\widetilde{\varphi}\rangle=\lim_{s\to+\infty}s^{-r}\e^{(\lambda_{1}-\lambda)s}\langle f,\widetilde{\varphi}\rangle=0.$$
Moreover, for every $t\ge 0$ and $y\in E$,
\begin{equation}\label{prop3.3}
T_{t}f^{*}(y)=\lim_{s\to+\infty}\frac{\e^{-\lambda s}}{s^{r}}T_{t+s}f(y)=\lim_{s\to+\infty}\e^{\lambda t}\big(1+\frac{t}{s}\big)^{r}\frac{\e^{-\lambda(t+s)}}{(t+s)^{r}}T_{t+s}f(y)=\e^{\lambda t}f^{*}(y).
\end{equation}
Thus, by the Markov property, for $t,s\ge 0$,
$$\pp_{\delta_{x}}[W^{f^{*}}_{t+s}|\mathcal{F}_{t}]=\e^{-\lambda(t+s)}\langle T_{s}f^{*},X_{t}\rangle=W^{f^{*}}_{t}.$$
Equation \eqref{prop3.3} also implies that $f^{*}\in \mathcal{D}^{0}(\lambda,0)$ with limiting function being itself. Hence, by Lemma \ref{lem0}(iii),
$$\Var_{\delta_{x}}[W^{f^{*}}_{t}]=\e^{-2\lambda t}\Var_{\delta_{x}}[\langle f^{*},X_{t}\rangle]\stackrel{b.p.}{\longrightarrow}\varpi^{2}_{f^{*}}(x)\mbox{ as }t\to+\infty.$$
Therefore, we conclude (i).

\smallskip

\noindent (ii) For $t>0$, let $\check{f}_{t}:=t^{-r}f-f^{*}$. Then $t^{-r}\e^{-\lambda t}X_{t}(f)=\e^{-\lambda t} X_{t}(\check{f}_{t})+W^{f^{*}}_{t}$. To prove (ii), it suffices to prove that
\begin{equation}\label{prop3.1}
\lim_{t\to+\infty}\e^{-\lambda t} X_{t}(\check{f}_{t})=0\mbox{ in }L^{2}(\pp_{\delta_{x}}).
\end{equation}
We have
\begin{align}\label{prop3.2}
\Var_{\delta_{x}}\big[\e^{-\lambda t} X_{t}(\check{f}_{t})\big]=&\e^{-2\lambda t}\int_{0}^{t}T_{s}\big(\vartheta[T_{t-s}\check{f}_{t}]\big)(x)\mathrm{d}s\nonumber\\
=&\int_{0}^{+\infty}\e^{-(2\lambda-\lambda_{1})s}1_{\{s\le t\}}\e^{-\lambda_{1}s}T_{s}\big(\vartheta\big[\big(1-\frac{s}{t}\big)^{r}\frac{\e^{-\lambda(t-s)}}{(t-s)^{r}}T_{t-s}f-f^{*}\big]\big)(x)\mathrm{d}s\nonumber\\
=:&\int_{0}^{+\infty}\e^{-(2\lambda-\lambda_{1})s}h_{t,x}(s)\mathrm{d}s.
\end{align}
Using the fact that $t^{-r}\e^{-\lambda t}T_{t}f-f^{*}\rightrightarrows 0$, one can show that for every $s>0$, as $t\to+\infty$,
$$\big(1-\frac{s}{t}\big)^{r}\frac{\e^{-\lambda(t-s)}}{(t-s)^{r}}T_{t-s}f-f^{*}\rightrightarrows 0\mbox{ and }\vartheta\big[\big(1-\frac{s}{t}\big)^{r}\frac{\e^{-\lambda(t-s)}}{(t-s)^{r}}T_{t-s}f-f^{*}\big]\stackrel{b.p.}{\longrightarrow}0.$$
Hence by (H2) and the bounded convergence theorem, for every $s>0$ and $x\in E$, $h_{t,x}(s)\to 0$ as $t\to+\infty$. Note that there is some $t_{0}>0$, such that
$M_{1}:=\sup_{t\ge t_{0}}\|t^{-r}\e^{-\lambda t}T_{t}f\|_{\infty}<+\infty$. Thus, for $t\ge t_{0}$ and $s\in [0,t-t_{0}]$,
$$\|\big(1-\frac{s}{t}\big)^{r}\frac{\e^{-\lambda(t-s)}}{(t-s)^{r}}T_{t-s}f\|_{\infty}\le\|\frac{\e^{-\lambda(t-s)}}{(t-s)^{r}}T_{t-s}f\|_{\infty}\le M_{1}.$$
On the other hand, by \eqref{eq1.9}, for $t\ge t_{0}$ and $s\in [t-t_{0},t]$,
$$\|\big(1-\frac{s}{t}\big)^{r}\frac{\e^{-\lambda(t-s)}}{(t-s)^{r}}T_{t-s}f\|_{\infty}=\|\frac{\e^{(\lambda_{1}-\lambda)(t-s)}}{t^{r}}\e^{-\lambda_{1}(t-s)}T_{t-s}f\|_{\infty}\le
\frac{\e^{(\lambda_{1}-\lambda)t_{0}}}{t^{r}_{0}}c_{1}\|f\|_{\infty}.$$
Therefore, we have proved $M_{2}:=\sup_{t\ge t_{0},\,0\le s\le t}\|\big(1-\frac{s}{t}\big)^{r}\frac{\e^{-\lambda(t-s)}}{(t-s)^{r}}T_{t-s}f\|_{\infty}<+\infty$. It then follows by \eqref{eq1.9} that for $t\ge t_{0}$ and $x\in E$,
$$h_{t,x}(s)\le (M_{2}+\|f^{*}\|_{\infty})^{2}\e^{-\lambda_{1}s}T_{s}(\vartheta[1])(x)1_{\{s\le t\}}\le (M_{2}+\|f^{*}\|_{\infty})^{2}c_{1}\|\vartheta[1]\|_{\infty}\quad\forall s>0.$$
Hence, we get by \eqref{prop3.2} and the bounded convergence theorem that
$\lim_{t\to+\infty}\Var_{\delta_{x}}[\e^{-\lambda t} X_{t}(\check{f}_{t})]=0$. Note that $\pp_{\delta_{x}}\big[\e^{-\lambda t} X_{t}(\check{f}_{t})\big]=t^{-r}\e^{-\lambda t}T_{t}f(x)-f^{*}(x)\to 0$ as $t\to+\infty$. Hence \eqref{prop3.1} follows.

\smallskip

\noindent (iii) It follows by (ii) that for any $\epsilon>0$, there is $t_{1}>0$ such that for all $t\ge t_{1}$, $\|t^{-r}\e^{-\lambda t}X_{t}(f)-W^{f^{*}}_{\infty}\|_{L^{2}(\pp_{\delta_{x}})}\le \epsilon$.
For $t\ge t_{1}$, let $F(t):=\int_{0}^{t-t_{1}}\e^{-\lambda s}\big(1-\frac{s}{t}\big)^{r}\mathrm{d}s$.
Obviously, by the dominated convergence theorem, $\lim_{t\to+\infty}F(t)=\int_{0}^{+\infty}\e^{-\lambda s}\mathrm{d}s=1/\lambda$.
For $t\ge t_{1}$, we may write $Y_{t}(f)=Y_{t_{1}}(f)+\int_{0}^{t-t_{1}}X_{t-s}(f)\mathrm{d}s$.
Note that
$$\frac{\e^{-\lambda t}}{t^{r}}\int_{0}^{t-t_{1}}X_{t-s}(f)\mathrm{d}s-F(t)W^{f^{*}}_{\infty}=\int_{0}^{t-t_{1}}\e^{-\lambda s}\big(1-\frac{s}{t}\big)^{r}
\big(\frac{\e^{-\lambda(t-s)}}{(t-s)^{r}}X_{t-s}(f)-W^{f^{*}}_{\infty}\big)\mathrm{d}s.$$
Hence we have
\begin{align*}
&\big\|\frac{\e^{-\lambda t}}{t^{r}}\int_{0}^{t-t_{1}}X_{t-s}(f)\mathrm{d}s-\frac{1}{\lambda}W^{f^{*}}_{\infty}\big\|_{L^{2}(\pp_{\delta_{x}})}\\
\le &\big\|\frac{\e^{-\lambda t}}{t^{r}}\int_{0}^{t-t_{1}}X_{t-s}(f)\mathrm{d}s-F(t)W^{f^{*}}_{\infty}\big\|_{L^{2}(\pp_{\delta_{x}})}+|F(t)-\frac{1}{\lambda}|\cdot\|W^{f^{*}}_{\infty}\|_{L^{2}(\pp_{\delta_{x}})}\\
\le & \int_{0}^{t-t_{1}}\e^{-\lambda s}\big(1-\frac{s}{t}\big)^{r}\big\|\frac{\e^{-\lambda(t-s)}}{(t-s)^{r}}X_{t-s}(f)-W^{f^{*}}_{\infty}\big\|_{L^{2}(\pp_{\delta_{x}})}\mathrm{d}s+|F(t)-\frac{1}{\lambda}|\cdot\|W^{f^{*}}_{\infty}\|_{L^{2}(\pp_{\delta_{x}})}\\
\le&\epsilon\int_{0}^{+\infty}\e^{-\lambda s}\mathrm{d}s+|F(t)-\frac{1}{\lambda}|\cdot\|W^{f^{*}}_{\infty}\|_{L^{2}(\pp_{\delta_{x}})}.
\end{align*}
The second inequality follows from Minkowski's inequality.
Letting $t\to+\infty$ and then $\epsilon\to 0$, we get
$$\lim_{t\to+\infty}t^{-r}\e^{-\lambda t}\int_{0}^{t-t_{1}}X_{t-s}(f)\mathrm{d}s=\frac{1}{\lambda}W^{f^{*}}_{\infty}\mbox{ in }L^{2}(\pp_{\delta_{x}}),$$
which in turn implies that $\lim_{t\to+\infty}t^{-r}\e^{-\lambda t}Y_{t}(f)=\lambda^{-1}W^{f^{*}}_{\infty}$.

\smallskip

\noindent (iv) Note that $(t+s)^{-r}\e^{(\lambda_{1}-\lambda)(t+s)}Z^{(t)}_{s}(f)=(t+s)^{-r}\e^{-\lambda(t+s)}Y_{t+s}(f)$ for all $t,s\ge 0$. Thus (iv) follows as a direct result of (iii) and \eqref{ineq1}.\qed

\subsection{Tightness}\label{sec3.2}
{
For $t\ge 0$, we set $a_{\mathrm{S}}(t):=\e^{-\lambda_{1}t/2}$ and
\begin{equation*}
a_{\lambda,r}(t):=
\begin{cases}
t^{-r-1/2}\e^{-\lambda_{1}t/2},&\mbox{ if }\lambda=\lambda_{1}/2,\ r\ge 0;\\
t^{-r}\e^{-\lambda t},&\mbox{ if }(\lambda,r)\in (\lambda_{1}/2,\lambda_{1})\times [0,+\infty)\cup\{(\lambda_{1},0)\}.\end{cases}
\end{equation*}

\begin{lemma}\label{lem4.1}
Suppose $(\lambda,r)\in [\lambda_{1}/2,\lambda_{1})\times [0,+\infty)\cup \{(\lambda_{1},0)\}$. For $f\in \mathcal{D}^{0}(\lambda,r)$,
set $a_{f}(t):=a_{\mathrm{S}}(t)$ if $f\in \mathcal{D}^{0}_{\mathrm{S}}$, and otherwise $a_{f}(t):=a_{\lambda,r}(t)$.
Then there exist constants $t_{1},c_{1}>0$, such that for all $t\ge t_{1}$ and $0\le u\le v<+\infty$,
\begin{equation}\label{lem4.1.0}
a_{f}(t+v)^{2}\pp_{\delta_{x}}\big[\big(Y_{t+v}(f)-Y_{t+u}(f)\big)^{2}\big]+\big(a_{f}(t+v)-a_{f}(t+u)\big)^{2}\pp_{\delta_{x}}\big[Y_{t+u}(f)^{2}\big]
\le c_{1}|u-v|^{2}.
\end{equation}
\end{lemma}

\begin{proof}
By the definition of $\mathcal{D}^{0}(\lambda,r)$ and Lemma \ref{lem0}, we have
$$\limsup_{t\to+\infty}a_{f}(t)^{2}\pp_{\delta_{x}}\big[ X_{t}(f)^{2}\big]\le\limsup_{t\to+\infty}a_{f}(t)^{2}\Var_{\delta_{x}}\big[ X_{t}(f)\big]+\limsup_{t\to+\infty}(a_{f}(t)T_{t}f(x))^{2}<+\infty.$$
Thus, there exists some $t_{2}>0$ such that $M_{1}:=\sup_{t\ge t_{2}}a_{f}(t)^{2}\pp_{\delta_{x}}\big[ X_{t}(f)^{2}\big]<+\infty$.
By H\"{o}lder's inequality, we have for $t\ge 0$ and $0\le u\le v<+\infty$,
\begin{equation}\label{lem4.1.4}
\pp_{\delta_{x}}\big[\big(Y_{t+v}(f)-Y_{t+u}(f)\big)^{2}\big]=\pp_{\delta_{x}}\big[\big(\int_{0}^{v-u}X_{t+u+s}(f)\mathrm{d}s\big)^{2}\big]\le (v-u)\int_{0}^{v-u}\pp_{\delta_{x}}\big[X_{t+u+s}(f)^{2}\big]\mathrm{d}s.
\end{equation}
Since for $t\ge t_{2}$ and $0\le s\le v-u$, $\pp_{\delta_{x}}\big[X_{t+u+s}(f)^{2}\big]\le M_{1}a_{f}(t+u+s)^{-2}\le M_{1}a_{f}(t+v)^{-2}$, we have by \eqref{lem4.1.4}
for $t\ge t_{2}$,
\begin{equation}\label{lem4.1.3}
a_{f}(t+v)^{2}\pp_{\delta_{x}}\big[\big(Y_{t+v}(f)-Y_{t+u}(f)\big)^{2}\big]\le (v-u)\int_{0}^{v-u}M_{1}\mathrm{d}s=M_{1}|u-v|^{2}.
\end{equation}

On the other hand, it follows by Lemma \ref{lem1} that
$$\limsup_{t\to+\infty}a_{f}(t)^{2}\pp_{\delta_{x}}\big[Y_{t}(f)^{2}\big]\le \limsup_{t\to+\infty}a_{f}(t)^{2}\Var_{\delta_{x}}[Y_{t}(f)]+\limsup_{t\to+\infty}(a_{f}(t)M_{t}[f](x))^{2}<+\infty.$$
Hence there exists some $t_{3}>0$ such that $M_{2}:=\sup_{t\ge t_{3}}a_{f}(t)^{2}\pp_{\delta_{x}}\big[Y_{t}(f)^{2}\big]<+\infty$. Consequently, for $t\ge t_{3}$ and $0\le u\le v<+\infty$,
\begin{equation}\label{lem4.1.2}
\big(a_{f}(t+v)-a_{f}(t+u)\big)^{2}\pp_{\delta_{x}}\big[Y_{t+u}(f)^{2}\big]\le M_{2}\big(\frac{a_{f}(t+v)}{a_{f}(t+u)}-1\big)^{2}.
\end{equation}
If $a_{f}(t)=a_{\mathrm{S}}(t)=\e^{-\lambda_{1}t/2}$, then
$$\big(\frac{a_{f}(t+v)}{a_{f}(t+u)}-1\big)^{2}=\big(1-\e^{-\frac{\lambda_{1}}{2}(v-u)}\big)^{2}\le \big(\frac{\lambda_{1}}{2}\big)^{2}(v-u)^{2}\quad\forall 0\le u\le v<+\infty.$$
If $a_{f}(t)=a_{\lambda_{1}/2,r}(t)=t^{-r-1/2}\e^{-\lambda_{1}t/2}$, then for any $t\ge t_{3}$ and $\delta\ge 0$,
\begin{multline}\label{lem4.1.1}
0\le 1-\frac{a_{f}(t+\delta)}{a_{f}(t)}=1-\e^{-\lambda_{1}\delta/2}\big(\frac{t}{t+\delta}\big)^{r+1/2}\le 1-\e^{-\lambda_{1}\delta/2}\big(\frac{t_{3}}{t_{3}+\delta}\big)^{r+1/2}\\
=\int_{0}^{\delta}\e^{-\lambda_{1}s/2}\big(\frac{t_{3}}{t_{3}+s}\big)^{r+1/2}\big(\frac{\lambda_{1}}{2}+(r+\frac{1}{2})(t_{3}+s)^{-1}\big)\mathrm{d}s\le \big(\frac{\lambda_{1}}{2}+(r+\frac{1}{2})t^{-1}_{3}\big)\delta.
\end{multline}
It follows that, for $t\ge t_{3}$ and $0\le u\le v<+\infty$,
\begin{equation*}
\big(\frac{a_{f}(t+v)}{a_{f}(t+u)}-1\big)^{2}\le \big(\frac{\lambda_{1}}{2}+(r+\frac{1}{2})t^{-1}_{3}\big)^{2}(u-v)^{2}.
\end{equation*}
Otherwise if $a_{f}(t)=a_{\lambda,r}(t)$ for some $(\lambda,r)\in (\lambda_{1}/2,\lambda_{1})\times [0,+\infty)\cup\{(\lambda_{1},0)\}$, we can show similarly to \eqref{lem4.1.1} that
for $t\ge t_{3}$ and $\delta\ge 0$
$$0\le 1-\frac{a_{f}(t+\delta)}{a_{f}(t)}\le (\lambda+r t^{-1}_{3})\delta,$$
and consequently, for $t\ge t_{3}$ and $0\le u\le v<+\infty$,
\begin{equation*}
\big(\frac{a_{f}(t+v)}{a_{f}(t+u)}-1\big)^{2}\le (\lambda+r t^{-1}_{3})^{2}(u-v)^{2}.
\end{equation*}
Therefore, in each case there exists a constant $c_{2}>0$ such that for $t\ge t_{3}$ and $0\le u\le v<+\infty$,
$\big(a_{f}(t+v)/a_{f}(t+u)-1\big)^{2}\le c_{2}(u-v)^{2}$. Inserting this into the right hand side of \eqref{lem4.1.2} yields that
$$\big(a_{f}(t+v)-a_{f}(t+u)\big)^{2}\pp_{\delta_{x}}\big[Y_{t+u}(f)^{2}\big]\le c_{2}M_{2}|u-v|^{2}\quad\forall t\ge t_{3},\ 0\le u\le v<+\infty.$$
Combining this with \eqref{lem4.1.3} we get the conclusion.
\end{proof}

\begin{proposition}\label{prop:tight1}
Suppose the assumptions of Lemma \ref{lem4.1} hold. If $\lambda=\lambda_{1}/2$, assume in addition that (M4) holds.
For $t,s\ge 0$, define
$$Z^{1,(t)}_{s}(f):=a_{f}(t+s)Y_{t+s}(f).$$
Then the family of processes $\{Z^{1,(t)}(f):=(Z^{1,(t)}_{s}(f))_{s\ge 0}:\ t\ge 0\}$ is tight in $\mathcal{D}[0,+\infty)$.
\end{proposition}
\begin{proof}
Note that for any $t\ge 0$ and $0\le u\le v<+\infty$,
$$Z^{1,(t)}_{v}(f)-Z^{1,(t)}_{u}(f)=a_{f}(t+v)(Y_{t+v}(f)-Y_{t+u}(f))+(a_{f}(t+v)-a_{f}(t+u))Y_{t+u}(f).$$
By the fact that $(a+b)^{2}\le 2(a^{2}+b^{2})$ for all $a,b\in\R$, we have
$$\pp_{\delta_{x}}\big[|Z^{1,(t)}_{v}(f)-Z^{1,(t)}_{u}(f)|^{2}\big]\le 2a_{f}(t+v)^{2}\pp_{\delta_{x}}\big[\big(Y_{t+v}(f)-Y_{t+u}(f)\big)^{2}\big]+2(a_{f}(t+v)-a_{f}(t+u))^{2}\pp_{\delta_{x}}\big[Y_{t+u}(f)^{2}\big].$$
It then follows by Lemma \ref{lem4.1} that, there exist constants $t_{1},c_{1}>0$ such that
\begin{equation}\label{lem4.2.1}
\pp_{\delta_{x}}\big[|Z^{1,(t)}_{v}(f)-Z^{1,(t)}_{u}(f)|^{2}\big]\le c_{1}|u-v|^{2}\quad\forall t\ge t_{1},\ 0\le u\le v<+\infty.
\end{equation}
Thus, for any $\epsilon>0$ and $\delta>0$,
$$\pp_{\delta_{x}}\big(|Z^{1,(t)}_{\delta}(f)-Z^{1,(t)}_{0}(f)|>\epsilon\big)\le \epsilon^{-2}\pp_{\delta_{x}}\big[|Z^{1,(t)}_{\delta}(f)-Z^{1,(t)}_{0}(f)|^{2}\big]\le c_{1}\epsilon^{-2}\delta^{2}\quad\forall t\ge t_{1},$$
which in turn yields that
\begin{equation}\label{lem4.2.2}
\lim_{\delta\downarrow 0}\limsup_{t\to+\infty}\pp_{\delta_{x}}\big(|Z^{1,(t)}_{\delta}(f)-Z^{1,(t)}_{0}(f)|>\epsilon\big)=0.
\end{equation}
Moreover, for any $\epsilon>0$, $t\ge t_{1}$ and $0\le u<v<w<+\infty$, by H\"{o}lder's inequality and \eqref{lem4.2.1},
\begin{align}\label{lem4.2.3}
&\pp_{\delta_{x}}\big(|Z^{1,(t)}_{v}(f)-Z^{1,(t)}_{u}(f)|>\epsilon,\ |Z^{1,(t)}_{w}(f)-Z^{1,(t)}_{v}(f)|>\epsilon\big)\nonumber\\
\le&\epsilon^{-2}\pp_{\delta_{x}}\big[|Z^{1,(t)}_{v}(f)-Z^{1,(t)}_{u}(f)|\cdot|Z^{1,(t)}_{w}(f)-Z^{1,(t)}_{v}(f)|\big]\nonumber\\
\le&\epsilon^{-2}\pp_{\delta_{x}}\big[|Z^{1,(t)}_{v}(f)-Z^{1,(t)}_{u}(f)|^{2}\big]^{1/2}\pp_{\delta_{x}}\big[|Z^{1,(t)}_{w}(f)-Z^{1,(t)}_{v}(f)|^{2}\big]^{1/2}\nonumber\\
\le&\epsilon^{-2}c_{1}|w-u|^{2}.
\end{align}
It follows by Propositions \ref{prop:fddsuper}, \ref{prop:fddcri} and \ref{prop:fddsub}(iv) that $\{Z^{1,(t)}_{0}(f)=a_{f}(t)Y_{t}(f):\ t\ge 0\}$ is tight in $\R$. This combined with \eqref{lem4.2.2} and \eqref{lem4.2.3} verifies conditions (i)-(iii) of Lemma \ref{lemA.3}. Therefore, $\{Z^{1,(t)}(f):\ t\ge 0\}$ is tight.
\end{proof}

\begin{lemma}\label{lem3.9}
For $t,s\ge 0$, define
$$\mathfrak{M}^{(t)}_{s}:=\e^{-\frac{\lambda_{1}}{2}(t+s)}\big( Y_{t+s}(\varphi)-\frac{1}{\lambda_{1}} X_{t+s}(\varphi)\big).$$
Then the family of processes $\{\mathfrak{M}^{(t)}:=\big(\mathfrak{M}^{(t)}_{s}\big)_{s\ge 0}:\ t\ge 0\}$ is tight in $\mathcal{D}[0,+\infty)$.
\end{lemma}
\begin{proof}
For $t,s\ge 0$, let $\mathcal{F}^{(t)}_{s}:=\mathcal{F}_{t+s}$. By Aldous' criterion for tightness (cf. \cite[Theorem VI.4.5]{JS}), we only need to show the following hold.
\begin{description}
\item[(a)] For any $N>0$,
$$\lim_{k\to+\infty}\limsup_{t\to+\infty}\pp_{\delta_{x}}\big(\sup_{s\in [0,N]}|\mathfrak{M}^{(t)}_{s}|>K\big)=0.$$
\item[(b)] For any $\epsilon>0$ and $N>0$,
$$\lim_{\theta\downarrow 0}\limsup_{t\to+\infty}\sup_{\sigma,\tau\in \mathcal{T}^{(t)}_{N},\,0\le \tau-\sigma\le \theta}\pp_{\delta_{x}}\big(|\mathfrak{M}^{(t)}_{\tau}-\mathfrak{M}^{(t)}_{\sigma}|\ge\epsilon\big)=0.$$
Here $\mathcal{T}^{(t)}_{N}$ denotes the set of $\mathcal{F}^{(t)}$-stopping times that are bounded by $N$.
\end{description}
Recall from Lemma \ref{lem3.n1} that $\{\mathrm{M}_{t}(\varphi)=Y_{t}(\varphi)-\lambda^{-1}_{1}X_{t}(\varphi):\ t\ge 0\}$ is an $\mathcal{F}$-adapted c\`{a}dl\`{a}g martingale with predictable quadratic variation $\langle M(\varphi)\rangle_{t}= \lambda^{-2}_{1}Y_{t}(\vartheta[\varphi])$.
We may write $\mathfrak{M}^{(t)}_{s}=\mathfrak{M}^{(t)}_{1}(s)+\mathfrak{M}^{(t)}_{2}(s)$, where
$$\mathfrak{M}^{(t)}_{1}(s):=\e^{-\frac{\lambda_{1}}{2}(t+s)}\big(M_{t+s}(\varphi)-\mathrm{M}_{t}(\varphi)\big)\mbox{ and }\mathfrak{M}^{(t)}_{2}(s):=\e^{-\frac{\lambda_{1}}{2}(t+s)}\mathrm{M}_{t}(\varphi)
=\e^{-\frac{\lambda_{1}}{2}s}\mathfrak{M}^{(t)}_{0}.$$
Clearly, we have $\pp_{\delta_{x}}\big(\sup_{s\in [0,N]}|\mathfrak{M}^{(t)}_{s}|>K\big)\le \sum_{i=1}^{2}\pp_{\delta_{x}}\big(\sup_{s\in [0,N]}|\mathfrak{M}^{(t)}_{i}(s)|>K/2\big)$ and
$\pp_{\delta_{x}}\big(|\mathfrak{M}^{(t)}_{\tau}-\mathfrak{M}^{(t)}_{\sigma}|\ge\epsilon\big)\le\sum_{i=1}^{2}
\pp_{\delta_{x}}\big(|\mathfrak{M}^{(t)}_{i}(\tau)-\mathfrak{M}^{(t)}_{i}(\sigma)|\ge\epsilon/2\big)$.
Hence to prove (a) and (b), it suffices to prove that
\begin{description}
\item[(a')] For any $N>0$,
$$\lim_{k\to+\infty}\limsup_{t\to+\infty}\pp_{\delta_{x}}\big(\sup_{s\in [0,N]}|\mathfrak{M}^{(t)}_{i}(s)|>K\big)=0\quad\forall i=1,2.$$
\item[(b')] For any $\epsilon>0$ and $N>0$,
$$\lim_{\theta\downarrow 0}\limsup_{t\to+\infty}\sup_{\sigma,\tau\in \mathcal{T}^{(t)}_{N},\,0\le \tau-\sigma\le \theta}\pp_{\delta_{x}}\big(|\mathfrak{M}^{(t)}_{i}(\tau)-\mathfrak{M}^{(t)}_{i}(\sigma)|\ge\epsilon\big)=0\quad\forall i=1,2.$$
\end{description}
\noindent Proof of (a'): Note that
$$\pp_{\delta_{x}}\big(\sup_{s\in [0,N]}|\mm^{(t)}_{2}(s)|>k\big)=\pp_{\delta_{x}}\big(\sup_{s\in [0,N]}|\e^{-\frac{\lambda_{1}}{2}s}\mm^{(t)}_{0}|>K\big)=\pp_{\delta_{x}}\big(|\mm^{(t)}_{0}|>K\big).$$
Since by Lemma \ref{lem3.n1} $\{\mm^{(t)}_{0}=\e^{-\frac{\lambda_{1}}{2}t}\mathrm{M}_{t}(\varphi):\ t\ge 0\}$ is tight in $\R$, we have
$$\lim_{K\to+\infty}\limsup_{t\to+\infty}\pp_{\delta_{x}}\big(\sup_{s\in [0,N]}|\mm^{(t)}_{2}(s)|>k\big)=\lim_{K\to+\infty}\limsup_{t\to+\infty}\pp_{\delta_{x}}\big(|\mm^{(t)}_{0}|>K\big)=0.$$
On the other hand, we have for $N>0$ and $K>0$,
\begin{align}\label{lem4.3.1}
\pp_{\delta_{x}}\big(\sup_{s\in [0,N]}|\mm^{(t)}_{1}(s)|>K\big)=&\pp_{\delta_{x}}\big(\sup_{s\in [0,N]}\e^{-\frac{\lambda_{1}}{2}(t+s)}|M_{t+s}(\varphi)-\mathrm{M}_{t}(\varphi)|>K\big)\nonumber\\
\le&\pp_{\delta_{x}}\big(\sup_{s\in [0,N]}|M_{t+s}(\varphi)-\mathrm{M}_{t}(\varphi)|>K\e^{\frac{\lambda_{1}}{2}t}\big)\nonumber\\
\le&4K^{-2}\e^{-\lambda_{1}t}\pp_{\delta_{x}}\big[\big(M_{t+N}(\varphi)-\mathrm{M}_{t}(\varphi)\big)^{2}\big]\nonumber\\
=&4K^{-2}\e^{-\lambda_{1}t}\pp_{\delta_{x}}\big[\langle M(\varphi)\rangle_{t+N}-\langle M(\varphi)\rangle_{t}\big]\nonumber\\
=&4K^{-2}\e^{-\lambda_{1}t}\lambda^{-2}_{1}\pp_{\delta_{x}}\big[Y_{t+N}(\vartheta[\varphi])-Y_{t}(\vartheta[\varphi])\big].
\end{align}
Here the second inequality follows from the $L^{2}$-maximal inequality.
By the Markov property,
$$\pp_{\delta_{x}}\big[Y_{t+N}(\vartheta[\varphi])-Y_{t}(\vartheta[\varphi])\big]=\pp_{\delta_{x}}\big[\pp_{X_{t}}\big[Y_{N}(\vartheta[\varphi])\big]\big]
=T_{t}\big(M_{N}[\vartheta[\varphi]]\big)(x).$$
Inserting this into the right hand side of \eqref{lem4.3.1} yields that
$$\pp_{\delta_{x}}\big(\sup_{s\in [0,N]}|\mm^{(t)}_{1}(s)|>K\big)\le 4K^{-2}\lambda^{-2}_{1}\e^{-\lambda_{1}t}T_{t}\big(M_{N}[\vartheta[\varphi]]\big)(x).$$
Since $\limsup_{t\to+\infty}\e^{-\lambda_{1}t}T_{t}\big(M_{N}[\vartheta[\varphi]]\big)(x)<+\infty$ by (H2), we get from the above inequality that
$$\lim_{K\to+\infty}\limsup_{t\to+\infty}\pp_{\delta_{x}}\big(\sup_{s\in [0,N]}|\mm^{(t)}_{1}(s)|>K\big)=0.$$

\smallskip

\noindent Proof of (b'): Let $\theta, t>0$. Suppose $\tau,\sigma\in \mathcal{T}^{(t)}_{N}$ and $0\le \tau-\sigma\le \theta$. Note that
$$\big|\mm^{(t)}_{2}(\tau)-\mm^{(t)}_{2}(\sigma)\big|=|(\e^{-\frac{\lambda_{1}}{2}\tau}-\e^{-\frac{\lambda_{1}}{2}\sigma})\mm^{(t)}_{0}|=\e^{-\frac{\lambda_{1}}{2}\sigma}
\big(1-\e^{-\frac{\lambda_{1}}{2}(\tau-\sigma)}\big)|\mm^{(t)}_{0}|\le\big(1-\e^{-\frac{\lambda_{1}}{2}\theta}\big)|\mm^{(t)}_{0}|.$$
Thus, for any $\epsilon>0$,
$$\lim_{\theta\downarrow 0}\limsup_{t\to+\infty}\sup_{\sigma,\tau\in\mathcal{T}^{(t)}_{N}\,0\le\tau-\sigma\le\theta}\pp_{\delta_{x}}
\big(\big|\mm^{(t)}_{2}(\tau)-\mm^{(t)}_{2}(\sigma)\big|\ge \epsilon\big)\le\lim_{\theta\downarrow 0}\limsup_{t\to+\infty}\pp_{\delta_{x}}\big(|\mm^{(t)}_{0}|\ge \epsilon\big(1-\e^{-\frac{\lambda_{1}}{2}\theta}\big)^{-1}\big)=0.$$
The final equality follows from the tightness of $\{\mm^{(t)}_{0}:\ t\ge 0\}$.

On the other hand, we have
\begin{align*}
\mm^{(t)}_{1}(\tau)-\mm^{(t)}_{1}(\sigma)=&\e^{-\frac{\lambda_{1}}{2}(t+\tau)}\big(M_{t+\tau}(\varphi)-M_{t+\sigma}(\varphi)\big)+\big(\e^{-\frac{\lambda_{1}}{2}(t+\tau)}-
\e^{-\frac{\lambda_{1}}{2}(t+\sigma)}\big)\big(M_{t+\sigma}(\varphi)-\mathrm{M}_{t}(\varphi)\big)\\
=&\e^{-\frac{\lambda_{1}}{2}(t+\tau)}\big(M_{t+\tau}(\varphi)-M_{t+\sigma}(\varphi)\big)+\big(\e^{-\frac{\lambda_{1}}{2}(\tau-\sigma)}-1\big)\mm^{(t)}_{1}(\sigma).
\end{align*}
Clearly, we have for any $\epsilon>0$,
\begin{equation}\label{lem4.3.4}
\pp_{\delta_{x}}\big(\big|\mm^{(t)}_{1}(\tau)-\mm^{(t)}_{1}(\sigma)\big|\ge \epsilon\big)\le\pp_{\delta_{x}}\big(\e^{-\frac{\lambda_{1}}{2}(t+\tau)}\big|M_{t+\tau}(\varphi)-M_{t+\sigma}(\varphi)\big|\ge\epsilon/2\big)
+\pp_{\delta_{x}}\big(\big|\big(\e^{-\frac{\lambda_{1}}{2}(\tau-\sigma)}-1\big)\mm^{(t)}_{1}(\sigma)\big|\ge\epsilon/2\big).
\end{equation}
Since $0\le\tau\le N$ and $0\le\tau-\sigma\le\theta$,
\begin{equation*}
\pp_{\delta_{x}}\big(\big|\big(\e^{-\frac{\lambda_{1}}{2}(\tau-\sigma)}-1\big)\mm^{(t)}_{1}(\sigma)\big|\ge \epsilon\big)
\le\pp_{\delta_{x}}\big(\big(1-\e^{-\frac{\lambda_{1}}{2}\theta}\big)\sup_{s\in [0,N]}|\mm^{(t)}_{1}(s)|\ge \epsilon\big).
\end{equation*}
Hence by (a'),
\begin{align}\label{lem4.3.2}
&\lim_{\theta\downarrow 0}\limsup_{t\to+\infty}\sup_{\sigma,\tau\in\mathcal{T}^{(t)}_{N}\,0\le\tau-\sigma\le\theta}\pp_{\delta_{x}}\big(\big|\big(\e^{-\frac{\lambda_{1}}{2}(\tau-\sigma)}-1\big)\mm^{(t)}_{1}(\sigma)\big|\ge \epsilon\big)\nonumber\\
\le&\lim_{\theta\downarrow 0}\limsup_{t\to+\infty}\pp_{\delta_{x}}\big(\sup_{s\in [0,N]}|\mm^{(t)}_{1}(s)|\ge \epsilon\big(1-\e^{-\frac{\lambda_{1}}{2}\theta}\big)^{-1}\big)=0.
\end{align}
From the definition, for every $t\ge 0$, both $\tau+t$ and $\sigma+t$ are bounded $\mathcal{F}$-stopping times. By the optional stopping theorem, we have
$\pp_{\delta_{x}}\big[M_{t+\tau}(\varphi)|\mathcal{F}_{t+\sigma}\big]=M_{t+\sigma}(\varphi)$ and
$$\pp_{\delta_{x}}\big[M_{t+\tau}(\varphi)^{2}-\langle M(\varphi)\rangle_{t+\tau}|\mathcal{F}_{t+\sigma}\big]=M_{t+\sigma}(\varphi)^{2}-\langle M(\varphi)\rangle_{t+\sigma},$$
which in turn yields that
\begin{equation}\label{lem4.3.3}
\pp_{\delta_{x}}\big[\big(M_{t+\tau}(\varphi)-M_{t+\sigma}(\varphi)\big)^{2}|\mathcal{F}_{t+\sigma}\big]=\pp_{\delta_{x}}\big[\langle M(\varphi)\rangle_{t+\tau}-\langle M(\varphi)\rangle_{t+\sigma}|\mathcal{F}_{t+\sigma}\big].
\end{equation}
By this and the Markov inequality, we get
\begin{align}\label{lem4.3.5}
\pp_{\delta_{x}}\big(\e^{-\frac{\lambda_{1}}{2}(t+\tau)}|M_{t+\tau}(\varphi)-M_{t+\sigma}(\varphi)|\ge \epsilon\big)
\le&\pp_{\delta_{x}}\big(|M_{t+\tau}(\varphi)-M_{t+\sigma}(\varphi)|\ge \epsilon\e^{\frac{\lambda_{1}}{2}t}\big)\nonumber\\
\le&\epsilon^{-2}\e^{-\lambda_{1}t}\pp_{\delta_{x}}\big[\big(M_{t+\tau}(\varphi)-M_{t+\sigma}(\varphi)\big)^{2}\big]\nonumber\\
=&\epsilon^{-2}\e^{-\lambda_{1}t}\pp_{\delta_{x}}\big[\langle M(\varphi)\rangle_{t+\tau}-\langle M(\varphi)\rangle_{t+\sigma}\big]\nonumber\\
=&\epsilon^{-2}\e^{-\lambda_{1}t}\lambda^{-2}_{1}\pp_{\delta_{x}}\big[Y_{t+\tau}(\vartheta[\varphi])-Y_{t+\sigma}(\vartheta[\varphi])\big].
\end{align}
Note that
\begin{align*}
&\e^{-\lambda_{1}t}\pp_{\delta_{x}}\big[Y_{t+\tau}(\vartheta[\varphi])-Y_{t+\sigma}(\vartheta[\varphi])\big]\\
=&\pp_{\delta_{x}}\big[\int_{\sigma}^{\tau}\e^{-\lambda_{1}t}\langle\vartheta[\varphi],X_{t+r}\rangle\mathrm{d}r\big]\\
=&\pp_{\delta_{x}}\big[\int_{\sigma}^{\tau}\e^{\lambda_{1}r}\big(\e^{-\lambda_{1}(t+r)}\langle\vartheta[\varphi],X_{t+r}\rangle-\langle\vartheta[\varphi],\widetilde{\varphi}\rangle W^{\varphi}_{\infty}\big)\mathrm{d}r+\langle\vartheta[\varphi],\widetilde{\varphi}\rangle W^{\varphi}_{\infty}\int_{\sigma}^{\tau}\e^{\lambda_{1}r}\mathrm{d}r\big]\\
\le&\pp_{\delta_{x}}\big[\int_{\sigma}^{\tau}\e^{\lambda_{1}r}\big|\e^{-\lambda_{1}(t+r)}\langle\vartheta[\varphi],X_{t+r}\rangle-\langle\vartheta[\varphi],\widetilde{\varphi}\rangle W^{\varphi}_{\infty}\big|\mathrm{d}r\big]+\frac{1}{\lambda_{1}}\langle\vartheta[\varphi],\widetilde{\varphi}\rangle \pp_{\delta_{x}}\big[W^{\varphi}_{\infty}\e^{\lambda_{1}\sigma}\big(\e^{\lambda_{1}(\tau-\sigma)}-1\big)\big].
\end{align*}
Hence, by \eqref{lem4.3.5}
\begin{align*}
&\sup_{\sigma,\tau\in\mathcal{T}^{(t)}_{N}\,0\le\tau-\sigma\le\theta}\pp_{\delta_{x}}\big(\e^{-\frac{\lambda_{1}}{2}(t+\tau)}|M_{t+\tau}(\varphi)-M_{t+\sigma}(\varphi)|\ge \epsilon\big)\\
\le&\epsilon^{-2}\lambda^{-2}_{1}\Big(\int_{0}^{N}\e^{\lambda_{1}r}\pp_{\delta_{x}}\big[\big|\e^{-\lambda_{1}(t+r)}\langle\vartheta[\varphi],X_{t+r}\rangle-\langle\vartheta[\varphi],\widetilde{\varphi}\rangle W^{\varphi}_{\infty}\big|\big]\mathrm{d}r+\frac{1}{\lambda_{1}}\langle\vartheta[\varphi],\widetilde{\varphi}\rangle\varphi(x)\e^{\lambda_{1}N}\big(\e^{\lambda_{1}\theta}-1\big)\Big).
\end{align*}
By Lemma \ref{lem2.2}(i), $\lim_{t\to+\infty}\pp_{\delta_{x}}\big[\big|\e^{-\lambda_{1}(t+r)}\langle\vartheta[\varphi],X_{t+r}\rangle-\langle\vartheta[\varphi],\widetilde{\varphi}\rangle W^{\varphi}_{\infty}\big|\big]=0$. Using this and the dominated convergence theorem, one can easily show that the integral in the right hand side converges to $0$ as $t\to+\infty$. Consequently,
$$\lim_{\theta\downarrow 0}\limsup_{t\to+\infty}\sup_{\sigma,\tau\in\mathcal{T}^{(t)}_{N}\,0\le\tau-\sigma\le\theta}\pp_{\delta_{x}}\big(\e^{-\frac{\lambda_{1}}{2}(t+\tau)}|M_{t+\tau}(\varphi)-M_{t+\sigma}(\varphi)|\ge \epsilon\big)=0.$$
Combining this with \eqref{lem4.3.2}, we prove (b'). Therefore, we complete the proof.
\end{proof}

\begin{proposition}\label{prop:tight2}
For $t,s\ge 0$, set
$$Z^{2,(t)}_{s}(\varphi):=\e^{\frac{\lambda_{1}}{2}(t+s)}Z^{(t)}_{s}(\varphi).$$
Then under $\pp_{\delta_{x}}$,
$$Z^{2,(t)}(\varphi):=(Z^{2,(t)}_{s}(\varphi))_{s\ge 0}\stackrel{d}{\to}\sqrt{W^{\varphi}_{\infty}}G^{2}(\varphi)\mbox{ in }\mathcal{C}[0,+\infty)\mbox{ as }t\to+\infty,$$
where $G^{2}(\varphi)=(G^{2}_{s}(\varphi))_{s\ge 0}$ is a centered Gaussian process with covariance
$$\Cov(G^{2}_{s}(\varphi),G^{2}_{t}(\varphi))=\e^{-\frac{\lambda_{1}}{2}|s-t|}\kappa_{\varphi}(|s-t|)\quad\forall s,t\ge 0.$$
\end{proposition}
\begin{proof}
Noting that the convergence of finite-dimensional distributions has been proved by Proposition \ref{prop:fddsuper}, we only need to show that $\{Z^{2,(t)}(\varphi):\ t\ge 0\}$ is tight in $\mathcal{D}[0,+\infty)$.
To this end, we may write $Z^{2,(t)}_{s}(\varphi)=\mathfrak{M}^{(t)}_{s}+\lambda^{-1}_{1}\mathfrak{N}^{(t)}_{s}$ for $s,t\ge 0$, where $\mathfrak{M}^{(t)}_{s}$ is defined as in Lemma \ref{lem3.9}, and
$\mathfrak{N}^{(t)}_{s}:=\e^{\frac{\lambda_{1}}{2}(t+s)}\big(W^{\varphi}_{t+s}-W^{\varphi}_{\infty}\big)$.
We claim that $\{\mathfrak{N}^{(t)}:\ t\ge 0\}$ is C-tight in the sense of \cite{JS}, i.e., the set of laws of $\mathfrak{N}^{(t)}$ is tight in $\mathcal{P}(\mathcal{D}[0,+\infty))$, and all limit points of the subsequences are laws of continuous processes. The proof of this claim is a line-by-line copy of \cite[Theorem 1.5]{Yang25}, thus we omit it.
Since $\{\mathfrak{M}^{(t)}:\ t\ge 0\}$ is tight by Lemma \ref{lem3.9}, it follows by \cite[Corollary VI.3.33]{JS} that $\{Z^{2,(t)}(\varphi):\ t\ge 0\}$ is tight.
\end{proof}

\subsection{Proofs of Theorems \ref{them1}-\ref{them3}}\label{sec3.3}

Recall that $\tilde{f}=f-\langle f,\widetilde{\varphi}\rangle \varphi$. We may write for every $f\in\mathcal{B}_{b}(E)$ and $t,s\ge 0$,
\begin{equation}\label{eq2}
Z^{(t)}_{s}(f)=\e^{-\lambda_{1}(t+s)}Y_{t+s}(\tilde{f})+\langle f,\widetilde{\varphi}\rangle Z^{(t)}_{s}(\varphi).
\end{equation}

\medskip

\noindent\textit{Proof of Theorem \ref{them1}:}
Note that $\tilde{f}\in \mathcal{D}^{0}_{\mathrm{S}}$ for $f\in \mathcal{D}_{\mathrm{S}}$. We have by \eqref{eq2}
\begin{equation}\label{them1.1}
\e^{\frac{\lambda_{1}}{2}(t+s)}Z^{(t)}_{s}(f)=\e^{-\frac{\lambda_{1}}{2}(t+s)}Y_{t+s}(\tilde{f})+\langle f,\widetilde{\varphi}\rangle \e^{\frac{\lambda_{1}}{2}(t+s)}Z^{(t)}_{s}(\varphi)=Z^{1,(t)}_{s}(\tilde{f})+\langle f,\widetilde{\varphi}\rangle Z^{2,(t)}_{s}(\varphi).
\end{equation}
Recall from Proposition \ref{prop:fddsuper} that $\big(Z^{1,(t)}_{s}(\tilde{f}),Z^{2,(t)}_{s}(\varphi)\big)_{s\ge 0}$ converges in finite-dimensional distributions to
$\sqrt{W^{\varphi}_{\infty}}\big(G^{1}(\tilde{f}),G^{2}(\varphi)\big)$, where $\big(G^{1}(\tilde{f}),G^{2}(\varphi)\big)$ is a Gaussian process independent of $W^{\varphi}_{\infty}$.
Hence by \eqref{them1.1} $\big(\e^{\frac{\lambda_{1}}{2}(t+s)}Z^{(t)}_{s}(f)\big)_{s\ge 0}$
converges in finite-dimensional distributions to a process $\sqrt{W^{\varphi}_{\infty}}G(f)$,
where $G(f)$ is independent of $W^{\varphi}_{\infty}$ and is equal in law to $G^{1}(\tilde{f})+\langle f,\widetilde{\varphi}\rangle G^{2}(\varphi)$. Consequently, $G(f)=(G_{s}(f))_{s\ge 0}$ is a centered Gaussian process with covariance given by
\begin{align*}
&\Cov (G_{s}(f),G_{t}(f))\\
=&\Cov(G^{1}_{s}(\tilde{f}),G^{1}_{t}(\tilde{f}))+\langle f,\widetilde{\varphi}\rangle^{2}\Cov(G^{2}_{s}(\varphi),G^{2}_{t}(\varphi))+\langle f,\widetilde{\varphi}\rangle \big(\Cov(G^{1}_{s}(\tilde{f}),G^{2}_{t}(\varphi))+\Cov(G^{1}_{t}(\tilde{f}),G^{2}_{s}(\varphi))\big)\\
=&\e^{-\frac{\lambda_{1}}{2}|s-t|}\big(\sigma_{\tilde{f}}(\infty,|s-t|)+\langle f,\widetilde{\varphi}\rangle^{2}\kappa_{\varphi}(|s-t|)+\frac{1}{2}\langle f,\widetilde{\varphi}\rangle\big({\varsigma}_{\varphi,\tilde{f}}(\infty,0)+{\varsigma}_{\varphi,\tilde{f}}(\infty,|s-t|)\big)\big)
\end{align*}
for $s,t\ge 0$.

On the other hand, since $\{Z^{1,(t)}(\tilde{f}):\ t\ge 0\}$ is tight by Proposition \ref{prop:tight1} and $\{Z^{2,(t)}(\varphi):\ t\ge 0\}$ is C-tight by Proposition \ref{prop:tight2}, it follows by \cite[Corollary VI.3.33]{JS} that $\{\big(\e^{\frac{\lambda_{1}}{2}(t+s)}Z^{(t)}_{s}(f)\big)_{s\ge 0}:\ t\ge 0\}$ is tight. \qed

\medskip

\noindent\textit{Proof of Theorem \ref{them2}:} For $f\in \mathcal{D}(\lambda_{1}/2,r)$, we have $\tilde{f}\in\mathcal{D}^{0}(\lambda_{1}/2,r)$, and by \eqref{eq2}
\begin{equation}\label{them1.2}
(t+s)^{-(r+1/2)}\e^{\frac{\lambda_{1}}{2}(t+s)}Z^{(t)}_{s}(f)=Z^{1,(t)}_{s}(\tilde{f})+\langle f,\widetilde{\varphi}\rangle (t+s)^{-(r+1/2)}Z^{2,(t)}_{s}(\varphi).
\end{equation}
It follows by Propositions \ref{prop:fddcri} and \ref{prop:tight2}, respectively, that
$$Z^{1,(t)}(\tilde{f})=\big((t+s)^{-(r+1/2)}\e^{-\frac{\lambda_{1}}{2}(t+s)}Y_{t+s}(\tilde{f})\big)_{s\ge 0}\stackrel{f.d.d.}{\longrightarrow}\frac{2}{\lambda_{1}}\sqrt{\frac{\langle\vartheta[f^{*}],\widetilde{\varphi}\rangle}{1+2r}W^{\varphi}_{\infty}}C\mbox{ as }t\to+\infty,$$
and
$$\big((t+s)^{-(r+1/2)}Z^{2,(t)}_{s}(\varphi)\big)_{s\ge 0}\stackrel{f.d.d.}{\longrightarrow}0\mbox{ as }t\to+\infty.$$
Hence, by Slutsky's lemma, we have
$$\big((t+s)^{-(r+1/2)}\e^{\frac{\lambda_{1}}{2}(t+s)}Z^{(t)}_{s}(f)\big)_{s\ge 0}\stackrel{f.d.d.}{\longrightarrow}\frac{2}{\lambda_{1}}\sqrt{\frac{\langle\vartheta[f^{*}],\widetilde{\varphi}\rangle}{1+2r}W^{\varphi}_{\infty}}C\mbox{ as }t\to+\infty$$
Since $\{Z^{2,(t)}(\varphi):\ t\ge 0\}$ is tight by Proposition \ref{prop:tight2}, it follows by \cite[Theorem VI.3.21]{JS} that, for any $N>0$ and $\epsilon>0$, there are constants $t_{0}>0$ and $K>0$ such that
$$\pp_{\delta_{x}}\big(\sup_{s\in [0,N]}|Z^{2,(t)}_{s}(\varphi)|>K\big)\le \epsilon\quad\forall t\ge t_{0}.$$
Thus, for any $\delta>0$ and $t\ge t_{0}\vee (K/\delta)^{1/(r+1/2)}$, we have
$$\pp_{\delta_{x}}\big(\sup_{s\in [0,N]}\big|(t+s)^{-(r+1/2)}Z^{2,(t)}_{s}(\varphi)\big|>\delta\big)\le \pp_{\delta_{x}}\big(\sup_{s\in [0,N]}|Z^{2,(t)}_{s}(\varphi)|>\delta t^{r+1/2}\big)\le \epsilon,$$
which in turn yields that
\begin{equation}\label{them2.1}
\limsup_{t\to+\infty}\pp_{\delta_{x}}\big(\sup_{s\in [0,N]}\big|(t+s)^{-(r+1/2)}Z^{2,(t)}_{s}(\varphi)\big|>\delta\big)=0.
\end{equation}
Since $\{Z^{1,(t)}(\tilde{f}):\ t\ge 0\}$ is tight by Proposition \ref{prop:tight1}, it follows by \eqref{them1.2}, \eqref{them2.1} and \cite[Lemma VI.3.31]{JS} that $\{\big((t+s)^{-(r+1/2)}\e^{\frac{\lambda_{1}}{2}(t+s)}Z^{(t)}_{s}(f)\big)_{s\ge 0}:\ t\ge 0\}$ is tight.\qed

\medskip

\noindent\textit{Proof of Theorem \ref{them3}:} Note that, for $f\in\mathcal{D}(\lambda,r)$, $\tilde{f}\in\mathcal{D}^{0}(\lambda,r)$ with the limiting function equal to $f^{*}$. The conclusion (i) follows directly from Proposition \ref{prop:fddsub}(i). Conclusion (ii) can be proved similarly to Theorem \ref{them2}.
Indeed, for $f\in\mathcal{D}(\lambda,r)$, we may write
$$(t+s)^{-r}\e^{(\lambda_{1}-\lambda)(t+s)}Z^{(t)}_{s}(f)=Z^{1,(t)}_{s}(\tilde{f})+(t+s)^{-r}\e^{-(\lambda-\frac{\lambda_{1}}{2})(t+s)}Z^{2,(t)}_{s}(\varphi).$$
Then the convergence of finite-dimensional distributions follows from Propositions \ref{prop:fddsub}(iv), \ref{prop:tight2} and Slutsky's lemma. Moreover, we can similarly prove that \eqref{them2.1} holds with $(t+s)^{-(r+1/2)}$ replaced by $(t+s)^{-r}\e^{-(\lambda-\lambda_{1}/2)(t+s)}$. This combined with Proposition \ref{prop:tight1} and \cite[Lemma VI.3.31]{JS} yields the tightness of $
\{\big((t+s)^{-r}\e^{(\lambda_{1}-\lambda)(t+s)}Z^{(t)}_{s}(f)\big)_{s\ge 0}:\ t\ge 0\}$.\qed
}

\appendix
\section{Appendix}
\subsection{Proof of Lemma \ref{lem:moment formulae}}\label{sec:A.1}
\noindent\textit{Proof of Lemma \ref{lem:moment formulae}:} By Fubini's theorem and the Markov property, we have
 \begin{align*}
 \pp_{\delta_{x}}\big[Y_{t}(f)\cdot X_{t}(g)\big]=&\int_{0}^{t}\pp_{\delta_{x}}\big[X_{s}(f)\pp_{X_{s}}\big[X_{t-s}(g)\big]\big]\mathrm{d}s=\int_{0}^{t}\pp_{\delta_{x}}\big[X_{s}(f)\cdot X_{s}(T_{t-s}g)\big]\mathrm{d}s\\
 =&\int_{0}^{t}\Cov_{\delta_{x}}(X_{s}(f),X_{s}(T_{t-s}g))+T_{s}f(x)\cdot T_{t}g(x)\mathrm{d}s
 \end{align*}
 Note that $\int_{0}^{t}T_{s}f(x)\cdot T_{t}g(x)\mathrm{d}s=M_{s}[f](x)T_{t}g(x)=\pp_{\delta_{x}}[Y_{t}(f)]\pp_{\delta_{x}}[X_{t}(g)]$. Thus, by \eqref{eq:covforX_{t}} and Fubini's theorem,
 \begin{align*}
 \Cov_{\delta_{x}}(Y_{t}(f),X_{t}(g))=&\int_{0}^{t}\Cov_{\delta_{x}}(X_{s}(f),X_{s}(T_{t-s}g))\mathrm{d}s=\int_{0}^{t}\int_{0}^{s}T_{r}\big(\vartheta[T_{s-r}f,T_{t-r}g]\big)(x)\mathrm{d}r\mathrm{d}s\\
 =&\int_{0}^{t}T_{r}\big(\vartheta\big[\int_{r}^{t}T_{s-r}f\mathrm{d}s,T_{t-r}g\big]\big)(x)\mathrm{d}r=\int_{0}^{t}T_{r}\big(\vartheta\big[M_{t-r}[f],T_{t-r}g\big]\big)(x)\mathrm{d}r.
 \end{align*}
Similarly, by Fubini's theorem and the Markov property, we have
\begin{align*}
\pp_{\delta_{x}}\big[Y_{t}(f)Y_{t}(g)\big]=&\pp_{\delta_{x}}\big[\int_{0}^{t}X_{s}(f)\mathrm{d}s\int_{s}^{t}X_{r}(g)\mathrm{d}r\big]+\pp_{\delta_{x}}\big[\int_{0}^{t}X_{s}(f)\mathrm{d}s
\int_{0}^{s}X_{r}(g)\mathrm{d}r\big]\\
=&\int_{0}^{t}\mathrm{d}s\int_{s}^{t}\pp_{\delta_{x}}\big[X_{s}(f)X_{s}(T_{r-s}g)\big]\mathrm{d}r+\int_{0}^{t}\mathrm{d}s\int_{0}^{s}\pp_{\delta_{x}}\big[X_{r}(g)X_{r}(T_{s-r}f)\big]\mathrm{d}r\\
=&\int_{0}^{t}\mathrm{d}s\int_{s}^{t}\Cov_{\delta_{x}}(X_{s}(f),X_{s}(T_{r-s}g))+T_{s}f(x)T_{r}g(x)\mathrm{d}r\\
&+\int_{0}^{t}\mathrm{d}s\int_{0}^{s}\Cov_{\delta_{x}}(X_{r}(g),X_{r}(T_{s-r}f))+T_{s}f(x)T_{r}g(x)\mathrm{d}r.
\end{align*}
Since $\int_{0}^{t}\int_{0}^{t}T_{s}f(x)T_{r}g(x)\mathrm{d}s\mathrm{d}r=M_{t}[f](x)M_{t}[g](x)=\pp_{\delta_{x}}\big[Y_{t}(f)\big]\pp_{\delta_{x}}\big[Y_{t}(g)\big]$, we have
\begin{align*}
&\Cov_{\delta_{x}}(Y_{t}(f),Y_{t}(g))\\
=&\int_{0}^{t}\mathrm{d}s\int_{s}^{t}\Cov_{\delta_{x}}(X_{s}(f),X_{s}(T_{r-s}g))\mathrm{d}r
+\int_{0}^{t}\mathrm{d}s\int_{0}^{s}\Cov_{\delta_{x}}(X_{r}(g),X_{r}(T_{s-r}f))\mathrm{d}r\\
=&\int_{0}^{t}\mathrm{d}s\int_{s}^{t}\mathrm{d}r\int_{0}^{s}T_{u}\big(\vartheta[T_{s-u}f,T_{r-u}g]\big)(x)\mathrm{d}u
+\int_{0}^{t}\mathrm{d}s\int_{0}^{s}\mathrm{d}r\int_{0}^{r}T_{u}\big(\vartheta[T_{s-u}f,T_{r-u}g]\big)(x)\mathrm{d}u\\
=&\int_{0}^{t}\Big(\int_{u}^{t}\int_{u}^{t}T_{u}\big(\vartheta[T_{s-u}f,T_{r-u}g]\big)(x)\mathrm{d}r\mathrm{d}s\Big)\mathrm{d}u=
\int_{0}^{t}T_{u}\big(\vartheta[\int_{u}^{t}T_{s-u}f\mathrm{d}s,\int_{u}^{t}T_{r-u}g\mathrm{d}r]\big)(x)\mathrm{d}u\\
=&\int_{0}^{t}T_{u}\big(M_{t-u}[f],M_{t-u}[g]\big)(x)\mathrm{d}u.
\end{align*}
To prove the second conclusion of this lemma, we deduce from \eqref{eq:U_{t}} that, for any $\mu\in\mf$,
\begin{align*}
\Cov_{\mu}(Y_{t}(f),X_{t}(g))=&\pp_{\mu}\big[Y_{t}(f)X_{t}(g)\big]-\pp_{\mu}\big[Y_{t}(f)\big]\pp_{\mu}\big[X_{t}(g)\big]\\
=&\left.\frac{\partial^{2}}{\partial\theta_{1}\partial\theta_{2}}\e^{-\langle U_{t}[\theta_{1}g,\theta_{2}f],\mu\rangle}\right|_{\theta_{1}=\theta_{2}=0}-\left.\frac{\partial}{\partial\theta_{2}}\e^{-\langle U_{t}[\theta_{1}g,\theta_{2}f],\mu\rangle}\cdot\frac{\partial}{\partial\theta_{1}}\e^{-\langle U_{t}[\theta_{1}g,\theta_{2}f],\mu\rangle}\right|_{\theta_{1}=\theta_{2}=0}\\
=&-\langle \left.\frac{\partial^{2}}{\partial\theta_{1}\partial\theta_{2}}U_{t}[\theta_{1}g,\theta_{2}f]\right|_{\theta_{1}=\theta_{2}=0},\mu\rangle.
\end{align*}
The final equality follows from the fact that $\left.U_{t}[\theta_{1}g,\theta_{2}f]\right|_{\theta_{1}=\theta_{2}=0}=0$.
Similarly, we have $\Cov_{\mu}(Y_{t}(f),Y_{t}(g))=-\langle \left.\frac{\partial^{2}}{\partial\theta_{1}\partial\theta_{2}}U_{t}[0,\theta_{1}g+\theta_{2}f]\right|_{\theta_{1}=\theta_{2}=0},\mu\rangle$. Hence the second conclusion follows immediately.

\subsection{Proof of Lemma \ref{lem3.n2}}\label{sec:A.3}
\noindent\textit{Proof of Lemma \ref{lem3.n2}:}
We deduce from \eqref{lem1.9} that for $f\in\mathcal{D}^{0}_{\mathrm{S}}$, there exists $\theta\in (0,\lambda_{1}/2)$ such that
\begin{equation}\label{lem3.2.1}
c_{0}:=\sup_{x\in E,\,t\ge 0}\big|\e^{-(\frac{\lambda_{1}}{2}-\theta)t}M_{t}[f]\big|<+\infty
\end{equation}
Thus, for $t,r\ge 0$ and $x\in E$,
$$\e^{-\lambda_{1}r}\vartheta[M_{r+t}[f],M_{r}[f]](x)=\e^{(\lambda_{1}-2\theta)t-2\theta r}\vartheta[\e^{-(\frac{\lambda_{1}}{2}-\theta)(t+r)}M_{t+r}[f],\e^{-(\frac{\lambda_{1}}{2}-\theta)t}M_{t}[f]](x)\le \e^{(\lambda_{1}-2\theta)t-2\theta r}c_{0}^{2}\vartheta[1](x).$$
Consequently, the integral $\int_{0}^{+\infty}\e^{-\lambda_{1}r}\langle\vartheta[M_{r+t}[f],M_{r}[f]],\widetilde{\varphi}\rangle\mathrm{d}r$ is finite for every $t\ge 0$, and by the monotone convergence theorem, it is equal to $\lim_{s\to+\infty}\sigma_{f}(s,t)$. Similarly, we can verify the existence and finiteness of the other limit of the first conclusion.

Next we calculate the four covariances appearing in the left hand side of \eqref{lem3.2.2}. Suppose $0\le u\le v\le s<+\infty$.
For $x\in E$, we have
\begin{align*}
&\Cov_{\delta_{x}}(\hat{A}_{f}(u,s),\hat{A}_{f}(v,s))\\
=&\e^{-\lambda_{1}(s+\frac{u+v}{2})}\Big(\Cov_{\delta_{x}}(Y_{s+u}(f),Y_{s+v}(f))-\Cov_{\delta_{x}}(Y_{s+u}(f),Y_{v}(f)+ X_{v}(M_{s}[f]))\\
&-\Cov_{\delta_{x}}(Y_{u}(f)+ X_{u}(M_{s}[f]),Y_{s+v}(f)-Y_{v}(f)- X_{v}(M_{s}[f]))\Big)\\
=&\e^{-\lambda_{1}(s+\frac{u+v}{2})}\Big(\Cov_{\delta_{x}}(Y_{s+u}(f),Y_{s+u}(f)+ X_{s+u}(M_{v-u}[f]))
-\Cov_{\delta_{x}}(Y_{v}(f)+ X_{v}(M_{s}[f]),Y_{v}(f)+ X_{v}(M_{s+u-v}[f]))\\
=&\e^{-\lambda_{1}(s+\frac{u+v}{2})}\Big(\int_{0}^{s+u}T_{s+u-r}\big(\vartheta[M_{r}[f],M_{r}[f]+T_{r}(M_{v-u}[f])]\big)(x)\mathrm{d}r\\
&-\int_{0}^{v}T_{v-r}\big(\vartheta[M_{r}[f]+T_{r}(M_{s}[f]),M_{r}[f]+T_{r}(M_{s+u-v}[f])]\big)(x)\mathrm{d}r\Big)\\
=&\e^{-\lambda_{1}(s+\frac{u+v}{2})}\Big(\int_{0}^{s+u}T_{s+u-r}\big(\vartheta[M_{r}[f],M_{v-u+r}[f])]\big)(x)\mathrm{d}r\\
&-\int_{0}^{v}T_{v-r}\big(\vartheta[M_{r+s}[f],M_{s+u-v+r}[f])]\big)(x)\mathrm{d}r\Big).
\end{align*}
Here the second equality follows from the Markov property and the third equality follows from Lemma \ref{lem:moment formulae}.
By a change of variables, the final integral equals $\int_{s+u-v}^{s+u}T_{s+u-r}\big(\vartheta[M_{v-u+r}[f],M_{r}[f]]\big)(x)\mathrm{d}r$. Hence we get
\begin{equation}\label{lem3.2.3}
\Cov_{\delta_{x}}(\hat{A}_{f}(u,s),\hat{A}_{f}(v,s))=\e^{-\lambda_{1}(s+\frac{u+v}{2})}\int_{0}^{s+u-v}T_{s+u-r}\big(\vartheta[M_{r+v-u}[f],M_{r}[f]]\big)(x)\mathrm{d}r.
\end{equation}
Similarly, by the Markov property and Lemma \ref{lem:moment formulae} we have
\begin{align}\label{lem3.2.4}
&\Cov_{\delta_{x}}(B_{\varphi}(s+u),B_{\varphi}(s+v))\nonumber\\
=&\e^{-\lambda_{1}(s+\frac{u+v}{2})}\Cov_{\delta_{x}}( Y_{s+u}(\varphi), Y_{s+v}(\varphi)
-\frac{1}{\lambda_{1}}
\e^{\lambda_{1}(s+v)}W^{\varphi}_{\infty})\nonumber\\
&-\e^{-\lambda_{1}(s+\frac{u+v}{2})}\Cov_{\delta_{x}}( Y_{s+v}(\varphi),\frac{1}{\lambda_{1}}\e^{\lambda_{1}(s+u)}W^{\varphi}_{\infty})
+\frac{1}{\lambda^{2}_{1}}\e^{\lambda_{1}(s+\frac{u+v}{2})}\Var_{\delta_{x}}[W^{\varphi}_{\infty}]\nonumber\\
=&\e^{-\lambda_{1}(s+\frac{u+v}{2})}\Cov_{\delta_{x}}( Y_{s+u}(\varphi), (Y_{s+v}(\varphi)+ X_{s+u}(M_{v-u}[\varphi]))-\frac{1}{\lambda_{1}}
\e^{\lambda_{1}(s+v)}W^{\varphi}_{s+u})\nonumber\\
&-\e^{-\lambda_{1}(s+\frac{u+v}{2})}\Cov_{\delta_{x}}( Y_{s+v}(\varphi),\frac{1}{\lambda_{1}}\e^{\lambda_{1}(s+u)}W^{\varphi}_{s+v})
+\frac{1}{\lambda^{2}_{1}}\e^{\lambda_{1}(s+\frac{u+v}{2})}\Theta(x)\nonumber\\
=&\e^{-\lambda_{1}(s+\frac{u+v}{2})}\Cov_{\delta_{x}}(Y_{s+u}(\varphi),Y_{s+u}(\varphi)-\frac{1}{\lambda_{1}} X_{s+u}(\varphi))\nonumber\\
&-\e^{-\lambda_{1}(s+\frac{3v-u}{2})}\Cov_{\delta_{x}}(Y_{s+v}(\varphi),\frac{1}{\lambda_{1}} X_{s+v}(\varphi))
+\frac{1}{\lambda^{2}_{1}}\e^{\lambda_{1}(s+\frac{u+v}{2})}\Theta(x)\nonumber\\
=&\e^{-\lambda_{1}(s+\frac{u+v}{2})}\int_{0}^{s+u}T_{s+u-r}\big(\vartheta[M_{r}[\varphi],M_{r}[\varphi]-\frac{1}{\lambda_{1}}T_{r}\varphi]\big)(x)\mathrm{d}r\nonumber\\
&-\e^{-\lambda_{1}(s+\frac{3v-u}{2})}\int_{0}^{s+v}T_{s+v-r}\big(\vartheta[M_{r}[\varphi],\frac{1}{\lambda_{1}}T_{r}\varphi]\big)(x)\mathrm{d}r
+\frac{1}{\lambda^{2}_{1}}\e^{\lambda_{1}(s+\frac{u+v}{2})}\Theta(x)\nonumber\\
=&-\e^{-\lambda_{1}(s+\frac{u+v}{2})}\frac{1}{\lambda^{2}_{1}}\int_{0}^{s+u}\big(\e^{\lambda_{1}r}-1\big)T_{s+u-r}\big(\vartheta[\varphi]\big)(x)\mathrm{d}r\nonumber\\
&-\e^{-\lambda_{1}(s+\frac{3v-u}{2})}\frac{1}{\lambda^{2}_{1}}\int_{0}^{s+v}\big(\e^{\lambda_{1}r}-1\big)\e^{\lambda_{1}r}
T_{s+v-r}\big(\vartheta[\varphi]\big)(x)\mathrm{d}r
+\frac{1}{\lambda^{2}_{1}}\e^{\lambda_{1}(s+\frac{u+v}{2})}\Theta(x),
\end{align}
where $\Theta(x)=\Var_{\delta_{x}}[W^{\varphi}_{\infty}]$ is given by \eqref{def:Theta(x)}. Moreover, we have
\begin{align}\label{lem3.2.7}
&\Cov_{\delta_{x}}(\hat{A}_{f}(u,s),B_{\varphi}(s+v))\nonumber\\
=&\e^{-\lambda_{1}(s+\frac{u+v}{2})}\Cov_{\delta_{x}}(Y_{s+u}(f)-Y_{u}(f)- X_{u}(M_{s}[f]),Y_{s+v}(\varphi)-\frac{1}{\lambda_{1}}\e^{\lambda_{1}(s+v)}W^{\varphi}_{\infty})\nonumber\\
=&\e^{-\lambda_{1}(s+\frac{u+v}{2})}\Cov_{\delta_{x}}(Y_{s+u}(f)-Y_{u}(f)- X_{u}(M_{s}[f]),(Y_{s+u}(\varphi)+ X_{s+u}(M_{v-u}[\varphi]))-\frac{1}{\lambda_{1}}\e^{\lambda_{1}(s+v)}W^{\varphi}_{s+u})\nonumber\\
=&\e^{-\lambda_{1}(s+\frac{u+v}{2})}\Cov_{\delta_{x}}(Y_{s+u}(f)-Y_{u}(f)- X_{u}(M_{s}[f]),Y_{s+u}(\varphi)-\frac{1}{\lambda_{1}} X_{s+u}(\varphi))\nonumber\\
=&\e^{-\lambda_{1}(s+\frac{u+v}{2})}\Big(\Cov_{\delta_{x}}(Y_{s+u}(f),Y_{s+u}(f)-\frac{1}{\lambda_{1}} X_{s+u}(\varphi))-\Cov_{\delta_{x}}(Y_{u}(f)+ X_{u}(M_{s}[f]),Y_{s+u}(\varphi)-\frac{1}{\lambda_{1}} X_{s+u}(\varphi))\Big).
\end{align}
By the Markov property, the second covariance in the right hand side of the last equality equals
\begin{align*}
&\Cov_{\delta_{x}}(Y_{u}(f)+ X_{u}(M_{s}[f]),Y_{u}(\varphi)+ X_{u}(M_{s}[\varphi])-\frac{1}{\lambda_{1}} X_{u}(T_{s}\varphi))\\
=&\Cov_{\delta_{x}}(Y_{u}(f)+ X_{u}(M_{s}[f]),Y_{u}(\varphi)-\frac{1}{\lambda_{1}}\langle\varphi,X_{u}\rangle).
\end{align*}
Hence we get by \eqref{lem3.2.7}
\begin{align}\label{lem3.2.5}
\Cov_{\delta_{x}}(\hat{A}_{f}(u,s),B_{\varphi}(s+v))
=&\e^{-\lambda_{1}(s+\frac{u+v}{2})}\Big(\int_{0}^{s+u}T_{s+u-r}\big(\vartheta[M_{r}[f],M_{r}[\varphi]-\frac{1}{\lambda_{1}}T_{r}\varphi]\big)(x)\mathrm{d}r\nonumber\\
&-\int_{0}^{u}T_{u-r}\big(\vartheta[M_{r}[f]+T_{r}(M_{s}[f]),M_{r}[\varphi]-\frac{1}{\lambda_{1}}T_{r}\varphi]\big)(x)\mathrm{d}r\Big)\nonumber\\
=&\frac{1}{\lambda_{1}}\e^{-\lambda_{1}(s+\frac{u+v}{2})}\Big(-\int_{0}^{s+u}T_{s+u-r}\big(\vartheta[M_{r}[f],\varphi]\big)(x)\mathrm{d}r\nonumber\\
&+\int_{0}^{u}T_{u-r}\big(\vartheta[M_{r+s}[f],\varphi]\big)(x)\mathrm{d}r\Big)\nonumber\\
=&-\frac{1}{\lambda_{1}}\e^{-\lambda_{1}(s+\frac{u+v}{2})}\int_{0}^{s}T_{s+u-r}\big(\vartheta[M_{r}[f],\varphi]\big)(x)\mathrm{d}r.
\end{align}
The final equality follows from a change of variable. Finally, we have
\begin{align}\label{lem3.2.6}
&\Cov_{\delta_{x}}(\hat{A}_{f}(v,s),B_{\varphi}(s+u))\nonumber\\
=&\e^{-\lambda_{1}(s+\frac{u+v}{2})}\Big(\Cov_{\delta_{x}}(Y_{s+u}(\varphi),Y_{s+v}(f))-\Cov_{\delta_{x}}(Y_{v}(f)+ X_{v}(M_{s}[f]),Y_{s+u}(\varphi))\Big)\nonumber\\
&+\frac{1}{\lambda_{1}}\e^{\frac{\lambda_{1}}{2}(u-v)}\Big(-\Cov_{\delta_{x}}(Y_{s+v}(f),W^{\varphi}_{\infty})+\Cov_{\delta_{x}}(Y_{v}(f)+ X_{v}(M_{s}[f]),W^{\varphi}_{\infty})\Big)\nonumber\\
=&\e^{-\lambda_{1}(s+\frac{u+v}{2})}\Big(\Cov_{\delta_{x}}(Y_{s+u}(\varphi),Y_{s+u}(f)+ X_{s+u}(M_{v-u}[f]))\nonumber\\
&-\Cov_{\delta_{x}}(Y_{v}(f)+ X_{v}(M_{s}[f]),Y_{v}(\varphi)+X_{v}(M_{s+u-v}[\varphi]))\Big)\nonumber\\
&+\frac{1}{\lambda_{1}}\e^{\frac{\lambda_{1}}{2}(u-v)}\Big(-\Cov_{\delta_{x}}(Y_{s+v}(f),W^{\varphi}_{s+v})+\Cov_{\delta_{x}}(Y_{v}(f)+ X_{v}(M_{s}[f]),W^{\varphi}_{v})\Big)\nonumber\\
=&\e^{-\lambda_{1}(s+\frac{u+v}{2})}\Big(\int_{0}^{s+u}T_{s+u-r}\big(\vartheta[M_{r}[\varphi],M_{r}[f]+T_{r}(M_{v-u}[f])]\big)(x)\mathrm{d}r\nonumber\\
&-\int_{0}^{v}T_{v-r}\big(\vartheta[M_{r}[f]+T_{r}(M_{s}[f]),M_{r}[\varphi]+T_{r}(M_{s+u-v}[\varphi])]\big)(x)\mathrm{d}r\Big)\nonumber\\
&+\frac{1}{\lambda_{1}}\e^{\frac{\lambda_{1}}{2}(u-v)}\Big(-\int_{0}^{s+v}T_{s+v-r}\big(\vartheta[M_{r}[f],\e^{-\lambda_{1}(s+v)}T_{r}\varphi]\big)(x)\mathrm{d}r\nonumber\\
&+\int_{0}^{v}T_{v-r}\big(\vartheta[M_{r}[f]+T_{r}(M_{s}[f]),\e^{-\lambda_{1}v}T_{r}\varphi]\big)(x)\mathrm{d}r\Big)\nonumber\\
=&\frac{1}{\lambda_{1}}\e^{-\lambda_{1}(s+\frac{u+v}{2})}\Big(\int_{0}^{s+u}\big(\e^{\lambda_{1}r}-1\big)T_{s+u-r}\big(\vartheta[M_{v-u+r}[f],\varphi]\big)(x)\mathrm{d}r\nonumber\\
&+\int_{0}^{v}T_{v-r}\big(\vartheta[M_{r+s}[f],\varphi]\big)(x)\mathrm{d}r\Big)\nonumber\\
&-\frac{1}{\lambda_{1}}\e^{\frac{\lambda_{1}}{2}(u-v)-\lambda_{1}(s+v)}\int_{0}^{s+v}\e^{\lambda_{1}r}T_{s+v-r}\big(\vartheta[M_{r}[f],\varphi]\big)(x)\mathrm{d}r.
\end{align}
Elementary calculation using \eqref{lem3.2.3}, \eqref{lem3.2.4}, \eqref{lem3.2.5}, \eqref{lem3.2.6} and the fact that $\e^{-\lambda_{1}t}T_{t}$ is invariant with respect to the measure $\widetilde{\varphi}$ gives the formulae presented in \eqref{lem3.2.2}.\qed

\subsection{Other useful results}\label{sec:A.4}

\begin{lemma}\label{lemA.0}
Suppose $(\Omega,\mathcal{G},\mathrm{P})$ is a probability space with a right continuous complete filtration $\mathcal{G}=\{\mathcal{G}_{t}:\ t\ge 0\}$, and $\mathrm{M}:=\{\mathrm{M}(t):\ t\ge 0\}$ is a $\mathcal{G}$-adapted c\`{a}dl\`{a}g local martingale with $\mathrm{M}(0)=0$. Let $\langle \mathrm{M}\rangle$ and $\nu^{\mathrm{M}}$, respectively, denote the predictable quadratic variation of $\mathrm{M}$ and the predictable compensator of its jump measure. Suppose the following conditions hold.
\begin{description}
\item[(i)] There exist a continuous and strictly increasing nonnegative function $\alpha$ on $[0,+\infty)$  and a nonnegative random variable $\eta$, such that $\lim_{t\to+\infty}\alpha(t)=+\infty$ and
$$\frac{\langle \mathrm{M}\rangle_{t}}{\alpha(t)}\stackrel{\mathrm{P}}{\longrightarrow}\eta\mbox{ as }t\to+\infty.$$
Here $\stackrel{\mathrm{P}}{\longrightarrow}$ denotes the convergence in probability under $\mathrm{P}$.
\item[(ii)] For any $\epsilon>0$,
$$\frac{1}{\alpha(t)}\int_{0}^{t}\int_{\mathbb{R}}|x|^{2}1_{\{|x|>\epsilon\sqrt{\alpha(t)}\}}\nu^{\mathrm{M}}(\mathrm{d}s,\mathrm{d}x)\stackrel{\mathrm{P}}{\longrightarrow}0\mbox{ as }t\to+\infty.$$
\end{description}
Then $\{\mathrm{M}_{t}/\sqrt{\alpha(t)}:\ t\ge 0\}$ converges stably in distribution to $\sqrt{\eta}\cdot N$, where $N$ is a standard normal random variable independent of $\eta$.
\end{lemma}

\begin{proof}
Without loss of generality, we assume $\alpha(0)=0$. Indeed, since $\alpha(t)\to+\infty$ as $t\to+\infty$, the function $\bar{\alpha}(t):=\alpha(t)-\alpha(0)$ satisfies $\bar{\alpha}(t)/\alpha(t)\to 1$, so that conditions (i)-(ii) and the desired convergence hold for $\alpha$ if and only if they hold for $\bar{\alpha}$. Under this assumption, the inverse function of $\alpha$ is well defined on $[0,+\infty)$.
We use $\alpha^{-1}$ to denote the inverse function of $\alpha$ which is also a continuous and strictly increasing nonnegative function on $[0,+\infty)$.
Let $\{t_{n}:\ n\ge 1\}$ be an arbitrary increasing sequence of positive numbers satisfying that $\lim_{n\to+\infty}t_{n}=+\infty$.
For every $n\ge 1$, define $\beta(n):=\alpha(t_{n})$ and
$$\mathrm{M}^{(n)}(t):=\frac{\mathrm{M}(\alpha^{-1}(\beta(n)t))}{\sqrt{\beta(n)}}\quad\forall t\in [0,1].$$
Let $\mathcal{G}^{(n)}_{t}:=\mathcal{G}_{\alpha^{-1}(\beta(n)t)}$ and $\mathcal{G}^{(n)}:=\{\mathcal{G}^{(n)}_{t}:\ t\ge 0\}$.
We can easily verify that for every $n\ge 1$,
$\mathrm{M}^{(n)}:=(\mathrm{M}^{(n)}(t))_{t\in [0,1]}$ is a
$\mathcal{G}^{(n)}$-adapted c\`{a}dl\`{a}g local martingale with $\mathrm{M}^{(n)}(0)=0$ and the predictable quadratic variation given by
$$\langle \mathrm{M}^{(n)}\rangle_{t}=\frac{\langle \mathrm{M}\rangle_{\alpha^{-1}(\beta(n)t)}}{\beta(n)}\quad\forall t\in [0,1].$$
By condition (i), we have for every $t\in (0,1]$,
\begin{equation}\label{lemA.2.1}
\langle \mathrm{M}^{(n)}\rangle_{t}=t\cdot \frac{\langle \mathrm{M}\rangle_{\alpha^{-1}(\beta(n)t)} }{\alpha(\alpha^{-1}(\beta(n)t))}\stackrel{\mathrm{P}}{\longrightarrow}t\cdot\eta\mbox{ as }n\to+\infty.
\end{equation}
If we use $\triangle \mathrm{M}^{(n)}_{s}$ (resp. $\triangle \mathrm{M}_{s}$) to denote the jump of $\mathrm{M}^{(n)}$ (resp. $\mathrm{M}$) at time $s$, then from the definition, for every $t\in (0,1]$ and $\epsilon>0$,
\begin{align*}
\sum_{s\le t}|\triangle \mathrm{M}^{(n)}_{s}|^{2}1_{\{|\triangle \mathrm{M}^{(n)}_{s}|>\epsilon\}}=&\frac{1}{\beta(n)}\sum_{s\le \alpha^{-1}(\beta(n)t)}|\triangle \mathrm{M}_{s}|^{2}1_{\{|\triangle \mathrm{M}_{s}|>\epsilon\sqrt{\beta(n)}\}}.
\end{align*}
Consequently, if we use $\nu^{(n)}$ to denote the predictable compensator of the jump measure of $\mathrm{M}^{(n)}$, then for every $t\in (0,1]$ and $\epsilon>0$,
\begin{align}\label{lemA.2.2}
\int_{0}^{t}\int_{\R}|x|^{2}1_{\{|x|>\epsilon\}}\nu^{(n)}(\mathrm{d}s,\mathrm{d}x)=&\frac{1}{\beta(n)}\int_{0}^{\alpha^{-1}(\beta(n)t)}\int_{\R}|x|^{2}1_{\{|x|>\epsilon\sqrt{\beta(n)}\}}
\nu^{\mathrm{M}}(\mathrm{d}s,\mathrm{d}x)\nonumber\\
=&t\cdot \frac{1}{\alpha(\alpha^{-1}(\beta(n)t))}\int_{0}^{\alpha^{-1}(\beta(n)t)}\int_{\R}|x|^{2}1_{\{|x|>\frac{\epsilon}{\sqrt{t}}\sqrt{\alpha(\alpha^{-1}(\beta_{n}t))}\}}
\nu^{\mathrm{M}}(\mathrm{d}s,\mathrm{d}x)\nonumber\\
\stackrel{\mathrm{P}}{\longrightarrow}&0,\quad\mbox{ as }n\to+\infty.
\end{align}
The final convergence follows from condition (ii) and the fact that $\alpha^{-1}(\beta(n)t)\to +\infty$ as $n\to+\infty$.
In view of \eqref{lemA.2.1} and \eqref{lemA.2.2}, we can apply Touati's central limit theorem (\cite[Theorem 3 and Corollary]{Tou92}) to the sequence
$\{\mathrm{M}^{(n)}:\ n\ge 1\}$ of local martingales to deduce that $\{\mathrm{M}^{(n)}(1)=\mathrm{M}(t_{n})/\sqrt{\alpha(t_{n})}:\ n\ge 1\}$ converges stably in distribution to $\sqrt{\eta}\cdot N$, where $N$ is a standard normal random variable independent of $\eta$. Since the sequence $\{t_{n}:\ n\ge 1\}$ is taken arbitrarily, it follows that $\{\mathrm{M}(t)/\sqrt{a(t)}:\ t\ge 0\}$ converges stably in distribution to $\sqrt{\eta}\cdot N$.
\end{proof}

The following result follows from \cite[Theorem VI.4.1]{JS}, see also \cite[Theorem 15.6]{B68}.
\begin{lemma}\label{lemA.3}
Let $\{X^{(n)}=(X^{(n)}_{t})_{t\ge 0}:\ n\ge  1\}$ be a sequence of processes in $\mathcal{D}[0,+\infty)$. Suppose the following conditions hold.
\begin{description}
\item[(i)] $\{X^{(n)}_{0}:\ n\ge 1\}$ is tight.
\item[(ii)] For any $\epsilon>0$,
$$\lim_{\delta\downarrow 0}\limsup_{n\to+\infty}\mathrm{P}\big(|X^{(n)}_{\delta}-X^{(n)}_{0}|>\epsilon\big)=0.$$
\item[(iii)] There exist a non-decreasing function $F$, two constants $\gamma\ge 0$ and $\alpha>1$, such that for any $\epsilon>0$ and $0<s<r<t$,
$$\sup_{n\ge 1}\mathrm{P}\big(|X^{(n)}_{r}-X^{(n)}_{s}|\ge \epsilon,\ |X^{(n)}_{t}-X^{(n)}_{r}|\ge \epsilon\big)\le \epsilon^{-\gamma}|F(t)-F(s)|^{\alpha}.$$
Then $\{X^{(n)}:\ n\ge 1\}$ is tight in $\mathcal{D}[0,+\infty)$.
\end{description}
\end{lemma}


\small


\begin{thebibliography}{99}
\bibitem{AM15}  Adamczak, R., Mi\l o\'{s}, P.: CLT for Ornstein--Uhlenbeck branching particle system. Electron. J. Probab. 20(1), 1--35, (2015).
\bibitem{AH83} Asmussen, S. and Hering, H.: Branching Processes (Progress Prob. Statist. 3), Birkh\"{a}user, Boston, MA, 1983.
\bibitem{Ath69a} Athreya, K.B.: Limit theorems for multitype continuous time Markov branching processes. I. The case of an eigenvector linear functional. { Z. Wahrscheinlichkeitstheorie verw Gebiete}, 12, 320--332, (1969).

\bibitem{Ath69b} Athreya, K.B: Limit theorems for multitype continuous time Markov branching processes. II. The case of an arbitrary linear functional. { Z. Wahrscheinlichkeitstheorie verw Gebiete}, 13, 204--214, (1969).
\bibitem{BLP15} Barczy, M., Li, Z., Pap, G.: Stochastic differential equation with jumps for multi-type continuous state and
continuous time branching processes with immigration, ALEA, Lat. Am. J. Probab. Math. Stat. 12,
129--169, (2015).
\bibitem{B68} Billingsley, P.: Convergence of Probability Measures, Wiley and Sons, New York, 1968.

\bibitem{BZ07} Birkner, M., Z\"{a}hle, I.: A functional CLT for the occupation time of a state-dependent branching random walk. Ann. Probab., 35, 2063--2090, (2007).

\bibitem{BGT04} Bojdecki, T., Gorostiza, L.G., Talarczyk, A.: Sub-fractional Brownian
 motion and its relation to occupation times. Statist. Probab. Lett., 69, 405--419, (2004).
\bibitem{BGT06a} Bojdecki, T., Gorostiza, L.G., Talarczyk, A.: Limit theorems for occupation time fluctuations of branching
systems I: Long-range dependence. Stoch. Proc. Appl., 116, 1--18, (2006).
\bibitem{BGT06b} Bojdecki, T., Gorostiza, L.G., Talarczyk, A.: Limit theorems for occupation time fluctuations of branching
systems II: Critical and large dimensions. Stoch. Proc. Appl., 116, 19--35, (2006).
\bibitem{BGT07} Bojdecki, T., Gorostiza, L.G., Talarczyk, A.: A long range dependence stable process and an infinite variance branching system. Ann. Probab., 35, 500--527, (2007).

\bibitem{CG85} Cox, J.T., Griffeath, D.: Occupation times for critical branching Brownian motions. Ann. Probab., 13, 1108--1132, (1985).

\bibitem{DGW01} Dawson, D.A., Gorostiza, L.G., Wakolbinger, A.: Occupation time fluctuations in branching systems. J.
Theoret. Probab., 14, 729--796, (2001).

\bibitem{DH25}  Dean, C. B. C., Horton, E.: Fluctuations of non-local branching Markov processes. arXiv:2502.19382 (2025).


\bibitem{Dudley} Dudley, R.M.: Real Analysis and Probability, Cambridge University Press, 2002.

\bibitem{FHX25} Fan, Z., Hong, J., Xiong, J.: Some quenched and annealed limit
theorems for superprocesses in random environments. Stoch. Proc. Appl., 188, 104686 (2025).


\bibitem{Guttorp91} Guttorp, P.: Statistical Inference for Branching Processes. Wiley, New York, 1991.
\bibitem{HHKW}Harris, S., Horton, E., Kyprianou, A. E., Wang, M.: Yaglom limit for critical nonlocal branching Markov processes.
	\emph{Ann. Probab.}, \textbf{50}(6), 2373--2408 (2022).
\bibitem{Hong04} Hong, W.: Functional central limit theorem for super $\alpha$-stable processes. Sci. China Ser. A, 47, 874--881, (2004).
\bibitem{HK23} Horton, E., Kyprianou, A.E.: Stochastic Neutron Transport: And Non-Local Branching Markov Processes. Birkh\"{a}user Cham, 2023.
\bibitem{HKMPR} Horton, E., Kyprianou, A.E., Mart\'{\i}n-Ch\'{a}vez, P., Powell, E., Rivero, V.: Stability of (sub) critical non-local spatial branching processes with and without immigration. To appear in \emph{Ann. Probab.}, https://arxiv.org/abs/2407.05472.	
	

\bibitem{Iscoe86a} Iscoe, I.: A weighted occupation time for a class of measure-valued branching processes. Probab. Theory Rel. Fields, 71, 85--116, (1986).
\bibitem{Iscoe86b} Iscoe, I.: Ergodic theory and a local occupation time for measure-valued critical branching Brownian motion. Stochastics, 18, 197--243, (1986).

\bibitem{JS} Jacod, J., Shiryaev, A.N.: Limit Theorems for Stochastic Processes. Springer, Berlin, 2003.

\bibitem{Jan04} Janson, S.: Functional limit theorems for multitype branching
        processes and generalized P\'{o}lya urns. Stoch. Proc. Appl., 110(2), 177--245, (2004).


\bibitem{KS66} Kesten, H., Stigum, B.P.: A limit theorem for multidimensional Galton--Watson processes. Ann. Math. Statist. 37(5), 1211--1223,(1966).

\bibitem{Li11} Li, Z.-H.:
Measure-valued Branching Markov Processes,
Springer, Heidelberg, 2011.


\bibitem{Milos09} Mi{\l}o\'{s}, P.: Occupation times of subcritical branching immigration systems with Markov motions. Stoch. Proc. Appl., 119, 3211--3237, (2009).

\bibitem{Milos18} Mi\l o\'s, P.: Spatial central limit theorem for supercritical superprocesses. J. Theoret. Probab. 31(1), 1--40 (2018).

\bibitem{PY20} Palau, S., Yang, T.: Law of large numbers for supercritical superprocesses
with non-local branching. {Stoch. Proc. Appl.}, 130, 1074--1102, (2020).

\bibitem{RSZ14} Ren, Y.-X., Song, R., Zhang, R.: Central limit theorems for supercritical branching Markov processes.
{J. Funct. Anal.}, 266(3), 1716--1756, (2014).
\bibitem{RSZ15} Ren, Y.-X., Song, R. and Zhang, R.: Central limit theorems for supercritical superprocesses. Stoch. Proc. Appl.,125, 428--457, (2015).
\bibitem{RSZ17} Ren, Y.-X., Song, R., Zhang, R.: Functional central limit theorems for supercritical superprocesses. Acta Appl. Math. 147, 137--175, (2017).
\bibitem{RSZ17b} Ren, Y.-X., Song, R. and Zhang, R.: Central limit theorems for supercritical branching nonsymmetric Markov processes. Ann. Probab., 45(1), 564--623, (2017).

\bibitem{Tou92} Touati, A.: On the functional convergence in distribution of sequences of semimartingales to a mixture of Brownian motions. \textit{Theory Probab. Appl.}, 1992, 36(4), 752--771.

\bibitem{Yang25} Yang, T.: Fluctuations of the linear functionals for supercritical non-local branching superprocesses. \emph{Electron. J. Probab.} \textbf{30} (2025) 1--60. https://doi.org/10.1214/25-EJP1443




\end{thebibliography}
\end{document}